\documentclass[aop,MSNbibl,nameyear,seceqn,dvips]{arximspdf}

\doi{10.1214/13-AOP846} 
\volume{42}
\issue{1}
\pubyear{2014}
\firstpage{237}
\lastpage{310}

\makeatletter
\newcommand{\rrvert}{\vert}
\newcommand{\llvert}{\vert}
\def\aligned{}
\def\endaligned{}
\def\cal{\mathcal}
\newcommand{\eqref}[1]{(\ref{#1})}
\newtheorem{thmm}{Theorem}[section]
\newtheorem{cor}[thmm]{Corollary}

\newtheorem{lemma}[thmm]{Lemma}
\newtheorem{lemmas}{Lemma}[section]
\newtheorem{prop}[thmm]{Proposition}
\newproclaim{asmptn}[thmm]{Assumption}
\newproclaim{fact}[thmm]{Fact}

\newtheorem{claim}[thmm]{Claim}

\newproclaim{dfn}[thmm]{Definition}
\newproclaim{example}[thmm]{Example}

\newproclaim{remark}[thmm]{Remark}

\newcommand{\var}{\operatorname{Var}}

\newcommand{\zz}[1]{\mathbb{#1}}
\newcommand{\la}{\langle}
\newcommand{\ra}{\rangle}

\def\sA{{\cal A}}
\def\sC{{\cal C}}
\def\sF{{\cal F}}
\def\sG{{\cal G}}
\def\sM{{\cal M}}

\def\bR{\mathbb{R}}

\def\to{\rightarrow}
\def\Lto{\Longrightarrow}

\def\wt{\widetilde}

\def\vep{\varepsilon}
\def\phi{\varphi}
\def\lam{{\lambda}}
\def\ol{\overline}
\def\wkc{ {\buildrel w \over\to} }

\def\indic{{\mathbf{1}}}
\def\ol{\overline}
\def\ul{\underline}
\makeatother

\begin{document}
\begin{frontmatter}

\title{A phase transition for measure-valued SIR epidemic processes}
\runtitle{A phase transition for measure-valued SIR}

\begin{aug}
\author[a]{\fnms{Steven P.} \snm{Lalley}\ead[label=e1]{lalley@galton.uchicago.edu}\thanksref{t1}},
\author[b]{\fnms{Edwin A.} \snm{Perkins}\ead[label=e2]{perkins@math.ubc.ca}\thanksref{t2}}
\and
\author[c]{\fnms{Xinghua} \snm{Zheng}\corref{}\ead[label=e3]{xhzheng@ust.hk}\thanksref{t3}}
\thankstext{t1}{Supported in part by NSF Grant DMS-11-06669.}
\thankstext{t2}{Supported in part by an NSERC Discovery grant.}
\thankstext{t3}{Supported in part by NSERC (Canada) and GRF
606010 of the HKSAR.}
\runauthor{S.~P. Lalley, E.~A. Perkins and X.~Zheng}
\affiliation{University of Chicago, University of British Columbia
and\break
Hong Kong University of Science and Technology}
\address[a]{S.~P. Lalley\\
Department of Statistics\\
University of Chicago\\
Chicago, Illinois 60637\\
USA\\
\printead{e1}}

\address[b]{E.~A. Perkins\\
Department of Mathematics\\
University of British Columbia\\
Vancouver\\
BC V6T 1Z2 Canada\\
\printead{e2}}

\address[c]{X.~Zheng\\
Department of Information Systems,\\
\quad Business Statistics and Operations Management\\
Hong Kong University of Science and Technology\\
Clear Water Bay, Kowloon\\
Hong Kong\\
\printead{e3}}
\end{aug}

\received{\smonth{11} \syear{2011}}
\revised{\smonth{2} \syear{2013}}

%
\begin{abstract}
We consider measure-valued processes $X=(X_t)$ that solve the
following martingale problem: for a given initial measure $X_{0}$, and for
all smooth, compactly supported test functions $\varphi$,
\begin{eqnarray*}
X_t(\phi)&=& X_{0} (\varphi)+\frac{1}{2}\int
_0^t X_s( \Delta\phi) \,ds +
\theta \int_0^t X_s(\phi) \,ds\\
&&{} -
\int_0^t X_s(L_s
\phi) \,ds + M_t(\phi).
\end{eqnarray*}
Here $L_s(x)$ is the local time density process associated with $X$,
and $M_t(\phi)$ is a martingale with quadratic variation
$[M(\phi)]_t=\int_0^t X_s( \phi^2) \,ds$. Such processes arise as
scaling limits
of SIR epidemic models. We show that there exist critical
values $\theta_c(d)\in(0,\infty)$ for dimensions $d=2,3$ such that if
$\theta>\theta_c(d)$, then the solution survives forever with positive
probability, but if $\theta<\theta_c(d)$, then the solution dies out
in finite time with probability~1. For $d=1$ we prove that the
solution dies out almost surely for all values of $\theta$. We also
show that in dimensions $d=2,3$ the process dies out locally almost
surely for any value of $\theta$; that is, for any compact set $K$,
the process $X_{t} (K)=0$ eventually.
\end{abstract}



\end{frontmatter}

\tableofcontents[level=2]

\section{Introduction}
\subsection{Epidemic models and their continuum
limits}\label{sseck-m} The use of stochastic processes to model
epidemics can be traced to \citet{McK1926} and \citet{KM27}, who
proposed a
simple continuous-time, mean-field model of an \emph{SIR} (for
\emph{susceptible-infected-removed} or
\emph{susceptible-infected-recovered}) epidemic. The corresponding
discrete-time model [known variously as the \emph{Reed--Frost} or the
\emph{chain-binomial} model---see \citet{daley-gani1999} for
background]
was proposed several years later, in 1928, by Reed and Frost in
lectures at Johns Hopkins University. In these models, an
\emph{infected} individual remains infected for a certain period of
time, during which he/she can transmit the disease to
\emph{susceptible} individuals, and then \emph{recovers}, after which
he/she is immune to further infection. Both models are
\emph{mean-field} models: the rate of infection transmission is the
same for all pairs of infected and susceptible members of the
population. The Reed--Frost model is of particular interest not only
because of its use in modeling epidemics and epidemic-like processes
but because of its close relation to the \emph{Erd\"{o}s--Renyi random
graph} model. In particular, given a realization of an Erd\"{o}s--Renyi
graph, whose vertices are marked either $S$ or $I$, a realization of
the Reed--Frost process can be obtained by defining $I_{n}$, the
infected set at time $n$, to be the set of all vertices at (graph)
distance $n$ from the set $I$. The union of the connected components of the
Erd\"{o}s--Renyi graph that contain vertices in the set $I=I_{0}$
consists of all individuals ever infected during the course of the
epidemic.

Spatial versions of the above models have a rich history in both the
mathematical
and biological literature. \citet{ba67} considered a spatial version of
the Reed--Frost model, and \citet{mo77} is a good source of information
about a range of related stochastic spatial models.
\citet{cd88} prove a shape theorem for a related continuous
time/discrete space model in two dimensions which is clearly similar in
spirit to our main theorems below on survival and local extinction for
a continuum model in two and three dimensions.

The \emph{SIR} models differ qualitatively from \emph{SIS} and
\emph{SIRS} models, such as the \emph{stochastic logistic} model, in
that the progress of the epidemic depends on an exhaustible resource
which is gradually consumed. This leads to interesting critical
behavior, as was discovered by \citet{martin-lof1998} and
\citet{aldous1999}. Martin-L\"{o}f proved, in particular, that at
criticality (when the probability of transmission from an infected to
a susceptible individual is $p=p_{c}=1/N$, where $N$ is the size of
the population), then as $N \rightarrow\infty$, after suitable
scaling, the total number of individuals ever infected converges in
law to the first passage time of a Wiener process to a parabolic
boundary. \citet{dolgoarshinnykh-lalley-2006} subsequently showed that
for suitable initial conditions the Kermack--McKendrick epidemic
process, after rescaling, converges weakly as $N \rightarrow\infty$
to a continuous-time process $I=(I_{t})$ that satisfies the stochastic
differential equation
\begin{eqnarray}
\label{eqdolgoarshinnykh-lalley} dI_{t}&= &(\lambda
I_{t}-I_{t}R_{t}) \,dt +\sqrt{I_{t}}
\,dW_{t}, \qquad\mbox{where}
\nonumber
\\[-8pt]
\\[-8pt]
\nonumber
\nonumber
dR_{t}&=&I_{t} \,dt.
\end{eqnarray}
The proof can easily be adapted to show that the Reed--Frost process
has the same limit. The parameter $\lambda\in\zz{R}$ represents the
\emph{transmission rate} of the disease: it is related to the
infected-susceptible transmission probability $p$ in the Reed--Frost
model by $p=1/N +\lambda/N^{4/3}$. It is not difficult to see (using
well-known facts about Feller's diffusion) that for any value of
$\lambda$ the process $I_{t}$ defined by
\eqref{eqdolgoarshinnykh-lalley} is eventually absorbed at $0$.

The subject of this paper is a stochastic partial differential
analogue of the system \eqref{eqdolgoarshinnykh-lalley} that arises
as a scaling limit of a spatial version of the Reed--Frost process
proposed by \citet{lalley09} as a crude model for an epidemic in a
geographically stratified population. In this model, populations of
size $N$ (``villages'') are located at each lattice point of
$\zz{Z}^{d}$; the rules of transmission are the same as in the
Reed--Frost model, \emph{except} that infectious contacts are permitted
only for infected-susceptible pairs in the same or neighboring
villages. (The model is described in more detail in Section~\ref{secdiscreteSIR} below.) Large-population ($N \rightarrow
\infty$) limit theorems for near-critical versions of these spatial
$\mathit{SIR}$ processes were proved for $d=1$ in \citet{lalley09} and for
$d=2,3$ in \citet{lz10}. The limit processes are now {continuous} finite
measure-valued processes $X=(X_t)_{t\geq0}$; for each time $t$, the
random measure $X_{t}$ represents the infected set (more precisely,
its distribution in space), and $R_{t}=\int_0^t X_s \,ds$ the recovered
set. The dynamics of the model are specified by the following
martingale problem. For any Radon measure $\mu$ on $\zz{R}^{d}$ and
any integrable or nonnegative measurable function $\phi\dvtx \zz{R}^{d}
\rightarrow\zz{R}$, write $\mu(\phi)$ or $\langle\mu,\phi\rangle$
for the integral $\int\phi \,d\mu$. Then for any initial mass
distribution $X_{0}=\mu$ and any test function $\varphi\in C^{2}_{c}
(\zz{R}^{d})$,
%
%
\begin{eqnarray}
\label{eqDWmgSIR} X_t(\phi)&=&\mu(
\phi)+ \frac{1}{2}\int_0^t
X_s( \Delta\phi) \,ds + \theta\int_0^t
X_s(\phi) \,ds
\nonumber
\\[-8pt]
\\[-8pt]
\nonumber
&&{}- \int_0^t
X_s (L_s\phi) \,ds + M_t(\phi).
\end{eqnarray}
Here $C^2_c(\zz{R}^d)$ stands for the space of compactly supported
twice differentiable with continuous second derivative functions on
$\zz{R}^d$,
$M_t(\phi)$ is a continuous martingale with quadratic variation
$[M(\phi)]_t=\int_0^t X_s(\phi^2) \,ds$ and $L_t(x)$ is the Sugitani
\emph{local time density process} of $X$, that is, for each $t\geq0$
the function $L_{t} (x)$ is the density of the occupation measure
$R_{t}$. [Throughout this article, unless otherwise specified, the
martingale $M_{t} (\phi)$ in a martingale problem such as \eqref{eqDWmgSIR} will be a martingale relative to the minimal
right continuous filtration of the process $X$, that is,
$\sF_t^X:=\bigcap_{u>t} \sigma(X_s,s\leq u)$]. Dawson's Girsanov formula
(Section~\ref{secgirsanov} below) implies that on a suitable
probability space there exists a solution to \eqref{eqDWmgSIR},
that solutions are unique in law, and that the law is absolutely
continuous on finite time intervals with respect to the law of
super-Brownian motion; see the definition below in
Section~\ref{ssecsuppression}. However, because the Sugitani local
time process
$L_{t}$ depends on the entire past of the spatial epidemic $X$,
solutions $X$ will not generally be Markov [although the
vector-valued process $(X_{t},L_{t})$ will be]. Henceforth, we shall
call a measure-valued processes $X$ satisfying \eqref{eqDWmgSIR} a
\emph{spatial epidemic} process with \emph{transmission rate} $\theta
$ and \emph{initial mass distribution} $\mu$.

The martingale problem \eqref{eqDWmgSIR} is a natural spatial
analogue of the stochastic differential equation
\eqref{eqdolgoarshinnykh-lalley}. In both problems, the key
qualitative feature is a ``resource depletion'' term: in
\eqref{eqdolgoarshinnykh-lalley}, it is the integral $\int_{0}^{t}
I_{s}R_{s} \,ds$, whereas in the martingale problem
\eqref{eqDWmgSIR} it is the integral $\int_{0}^{t}
X_s (L_s\phi) \,ds$. It seems likely that processes $X$
governed by \eqref{eqDWmgSIR}---or similar equations incorporating
depletion terms---should also arise as continuum limits of models for
various other physical (combustion), chemical (reaction--diffusion),
and biological processes (foraging) in which there is an exhaustible
resource upon which the process depends. In fact, \citet{MT11} have
suggested (see their Remark at the end of Section~6) that they
should also occur as scaling limits of certain stochastic
reaction--diffusion systems.

\subsection{Main results: Survival}\label{ssecsurvival}
A~measure-valued process  $X$ \emph{survives} if\break $X_t(1)>0$ for all
$t>0$; it \emph{dies out}, or becomes \emph{extinct}, if
$X_t=0$ for large enough $t$. For processes governed by equation
\eqref{eqDWmgSIR}, the question of whether or not there is survival
or extinction is of fundamental importance. \citet{MT11} (see again the
Remark at the end of Section~6) have
conjectured that there is a critical value $\theta_{c}=\theta_{c}
(d)\in(0,\infty)$ for the transmission rate below which extinction is
certain and above which survival has positive probability. Our main
result states that under a mild restriction on the initial measure
$\mu$ this is true in dimensions $d=2$ and $d=3$, but that in $d=1$
extinction is certain at all values of the parameter $\theta$. The
restriction on the initial measure is as follows:

\begin{asmptn}\label{asmtpinireg1} The measure $\mu$ has compact
support and finite total mass, and when $d=2$ or $3$, its convolution
$\mu*q_{t}$ with
the integrated Gauss kernel
%
%
\begin{equation}
\label{eqheatKernel} q_t(x)=\int_0^t
p_s(x) \,ds,\qquad \mbox{where } p_{t} (x) =\frac{e^{-|x|^{2}/2t}}{(2\pi t)^{d/2}},
\end{equation}
is jointly continuous in $(t,x)\in[0,\infty)\times\zz{R}^d$.
\end{asmptn}

Theorem 2 of \citet{Sugitani89} asserts that in dimensions $d=2$ and
$d=3$ a super-Brownian motion with initial mass distribution
$\mu$ satisfying Assumption \ref{asmtpinireg1} has a {local time
density process} $L_{t} (x)$ that is jointly continuous in $t\geq0$
and $x\in\zz{R}^{d}$. Since the law of a spatial epidemic $X$ is
absolutely continuous relative to that of super-Brownian motion,
spatial epidemics must also have jointly continuous local time
processes in $d=2,3$. In dimension $d=1$, the existence and continuity
of the local time process follows from the fact that the state of a
super-Brownian motion at any time $t$ is absolutely continuous with
respect to Lebesgue measure, with a
jointly continuous density. Thus,
equation \eqref{eqDWmgSIR} makes sense in all dimensions $d\leq
3$, and so henceforth we assume that $d\leq3$.

\begin{thmm}\label{thmmmain}
There exist critical values $\theta_{c}=\theta_{c} (2)>0$ and
$\theta_{c}=\theta_{c} (3)>0$ such that the following is true:\vadjust{\goodbreak} if $d=2$
or $d=3$, and $X$ is a spatial epidemic process
in $\zz{R}^{d}$ with transmission rate $\theta$ and initial mass
distribution $\mu$ satisfying Assumption \ref{asmtpinireg1}, then:
\begin{longlist}[(a)]
\item[(a)] if $\theta<\theta_{c}$, then $X$ dies out almost surely, but
\item[(b)] if $\theta>\theta_{c}$, then $X$ survives with positive
probability.
\end{longlist}
If $X$ is a spatial epidemic in $\zz{R}^{1}$ with any transmission
rate $\theta$ and any finite initial mass distribution $\mu$, then
$X$ dies out almost surely.
\end{thmm}

Thus, in dimensions $2$ and $3$ a spatial epidemic can survive if the
transmission rate is sufficiently high. However, since the process
feeds on a substrate which is gradually consumed in infected areas, it
is natural to conjecture that the epidemic should survive in a transient
wave which sweeps through space. The following result partly
establishes the validity of this picture.

\begin{thmm}\label{thmmlocext} Let $X$ be a spatial epidemic with
arbitrary transmission rate $\theta\in\zz{R}$ and initial mass
distribution satisfying Assumption \ref{asmtpinireg1}. For any
compact set $K\subset\zz{R}^{d}$, with probability one,
%
%
\begin{equation}
\label{locext} X_{t} (K)=0\qquad \mbox{eventually.}
\end{equation}
\end{thmm}

Consequently, with probability one the local time $L_{t} (x)$ at any
point $x$ is eventually constant. Since the local time $L_{t} (x)$ is
jointly continuous in its arguments, it follows that
$L_\infty(x):=\lim_{t \rightarrow\infty}L_{t} (x)$ is finite and
continuous in $x$ almost surely.

\subsection{Proof strategy and heuristics}\label{ssecheuristics}
The proofs of Theorems \ref{thmmmain}--\ref{thmmlocext} are rather
technical, largely because of difficulties that will arise in carrying
out comparison arguments for measure-valued processes defined by
stochastic partial differential equations in which the entire
histories of the solutions (e.g., local time density) influence the
coefficients. However, the ideas behind the results can be explained,
at least roughly, in simple terms. Consider first the assertion of
global extinction in one dimension. If the epidemic process $X$
were to survive with positive probability, then on this event its
total mass $X_t(1)$ would diverge to $\infty$, since otherwise the
process would be presented with infinitely many opportunities to
become extinct; see Lemma \ref{lemmasurvchar}. In addition, by a
large deviations calculation on a dominating super-Brownian motion
with drift $\theta$ [see \citet{Pinsky95}], there exists $c<\infty$
such that $\operatorname{Supp}(X_t)\subset[-ct,ct]^d$ for all large $t$. Therefore,
for $d=1$, on the event of survival and for large $t$, the average
value of $L_t(\cdot)$ must satisfy
\[
(2ct)^{-1}\int_{-ct}^{ct}L_t(x)
\,dx=(2ct)^{-1}\Vert L_{t} \Vert_{1}=(2ct)^{-1}
\int_0^tX_s(1) \,ds\to\infty.
\]
If \eqref{eqDWmgSIR} were valid for the function $\varphi\equiv1$
(it is only assumed for compactly supported functions), then it would
follow that for
large $t$ the drift term in \eqref{eqDWmgSIR} for the total mass
$X_{t} (1)$ would eventually turn (very) negative, making it
impossible for $X_{t} (1)$ to remain positive. The formal proof in
Section~\ref{sec1dext} makes this heuristic argument precise.

A local variation of this argument (which is harder to justify rigorously---see Section~\ref{seclocalExtinction}) explains the strong local extinction
asserted in Theorem \ref{thmmlocext}. We will show that in order for
$X_{t} (K)>0$ to occur at indefinitely large times, for some
ball $B$
centered at the origin, it must be the case that $X_{t} (3B)$
integrates to $\infty$. This, however, would imply that the local time in
$2B\setminus B$ would grow indefinitely, eventually making the drift
in the equation \eqref{eqDWmgSIR} for $X_{t} (B)$ negative.

A different line of argument makes it at least plausible that in
dimensions $d\geq2$ the epidemic $X$ might survive with positive
probability when the transmission rate $\theta$ is sufficiently
large. If $\theta$ is large, then equation \eqref{eqDWmgSIR}
implies that when the infection first enters a region $K$ of space it
will, at least for a while, grow at least as fast as a super-Brownian
motion with a large constant drift. Thus, with high probability, the
total mass $X_{t} (K)$ will become large long before the local time
$L_{t}$ becomes appreciable in $K$. In particular, for a cube $K$, if
the size $X_{t} (K)$ of the infected set reaches a certain threshold
before the local time exceeds a fraction of this level, then the
epidemic will have high probability of spreading to neighboring cubes
quickly, and the infection in these cubes will have similarly high
probability of spreading to neighboring cubes, and so on. Since
high-density oriented site percolation in dimensions $d\geq2$ has
infinite clusters, with positive probability, it should then follow
that the epidemic will reach infinitely many cubes with positive
probability. It will take some work to implement this plan. This is
done in Section~\ref{secsurv} after some important groundwork is laid
in Sections~\ref{secpreliminariesepi} and
\ref{secpreliminariessbm}.

For the extinction assertion of Theorem \ref{thmmmain} we will adapt the
corresponding argument of \citet{MT94}. For small
$\theta>0$ it is possible to rescale $X$ so that the total mass
process can be dominated by a subcritical branching process which dies
out. The actual implementation of this idea is carried out in a
slightly different manner in Section~\ref{secext}; see
Proposition \ref{eqnUscaled}. A key observation, used here and
elsewhere in this work, is that if the initial state is split up into
pieces, then one can couple the epidemics so that the survival
probability is dominated by the sum of the survival probabilities
corresponding to the pieces; see Lemma \ref{lemmacompmu}.

\subsection{Epidemics with suppression}\label{ssecsuppression}

Our results extend to a somewhat larger class of measure-valued
processes that incorporate location-dependent \emph{local
suppression}. Let $K\dvtx \zz{R}^{d}\rightarrow\zz{R}_{+}$ be a bounded,
continuous (or, more generally, piecewise continuous), nonnegative
function; call this the \emph{suppression rate.} A \emph{spatial
epidemic with local suppression rate} $K$, \emph{transmission rate}
$\theta$, \emph{branching rate} $\gamma>0$ and \emph{inhibition
parameter} $\beta\geq0$ is a solution to the
martingale problem $(\mathrm{MP})_{\mu,K}^{\theta,\beta,\gamma}$
%
%
\begin{eqnarray}
\label{eqDWmgSIRgen}\quad X_t(\phi) =\mu(\phi)+\int_0^t
X_s \bigl( \Delta\phi/2 +\theta\phi-K\phi-\beta
L_s(X)\phi \bigr) \,ds+\sqrt{\gamma} M_t(\phi),
\nonumber
\\[-8pt]
\\[-8pt]
\eqntext{\phi\in
C^2_c \bigl(\bR^d \bigr),}
\end{eqnarray}
where $X$ is a continuous finite measure-valued process, $M_{t}
(\varphi
)$ is a continuous martingale with quadratic
variation $[M(\phi)]_t=\int_0^t X_s(\phi^2) \,ds$ and $L_{t} (x)$ is
the local time density of $X_{t}$. When $K\equiv0$ and $\beta
=\theta=0$ and $\gamma=1$, a process $X$ satisfying~\eqref{eqDWmgSIRgen} is a \emph{super-Brownian motion}; more
generally, when $K\equiv0$ and $\beta=0$, it is a
\emph{super-Brownian motion with drift} $\theta$ and branching
rate $\gamma$; and when $\beta=0$ it is a \emph{super-Brownian motion
with killing}, with local killing rate $K$, drift $\theta$ and
branching rate $\gamma$. Theorems \ref{thmmmain} and \ref{thmmlocext}
extend to all processes governed by \eqref{eqDWmgSIRgen} with
$\beta>0$; the critical values $\theta_{c}$ will depend on the
parameters, but not on the suppression rate function $K$ if we
restrict $K$ to be compactly supported. In the interests of
simplicity we shall prove our main results only in the case $K\equiv
0$. However, solutions of the martingale problem
\eqref{eqDWmgSIRgen} will be needed in the proofs of the main
results even in the special case $K\equiv0$ as they will arise
naturally in the Markov property of solutions to \eqref{eqDWmgSIR}.

\subsection{Relations with scaling laws for contact processes}\label
{ssecrelations} As
noted above, in $\mathit{SIR}$ epidemics, unlike $\mathit{SIS}$ and $\mathit{SIR}S$ epidemics,
the population of susceptible individuals is gradually depleted during
the course of the epidemic. It is this that accounts for the depletion
term $-\int_{0}^{t} X_s (L_s\phi) \,ds$ in the martingale
problem \eqref{eqDWmgSIR}, which in turn is responsible for the
local extinction asserted in Theorem \ref{thmmlocext}. Spatial models
of $\mathit{SIS}$ and $\mathit{SIR}S$ lead to measure-valued processes with different
qualitative behavior. One such model that has been studied in some
detail is the \emph{long-range contact process}; cf. \citet{BDS89},
\citet{MT94}, \citet{DP99}. In this model, only one individual, who
can be either infected or susceptible, inhabits each lattice point,
but infectious contact is allowed at distances up to $L$ (usually the
$\ell_{\infty}$ metric is used). Scaling limits were obtained for the
limiting regime $L \rightarrow\infty$; see \citet{MT95} and
\citet{DP99} for details. \citet{BDS89} determined the long-range
functional dependence of the critical value $\lambda_c(L)$ on $L$ (but
not the precise constants): in dimension $d=1$, they showed that for
large $L$,
%
%
\begin{equation}
\label{LRC} 0<cL^{-2/3}\le\lambda_c(L)-1\le
CL^{-2/3}.
\end{equation}
The long-conjectured (but still unresolved) link between the discrete
and continuum settings in $d=1$ is that
%
%
\begin{equation}
\label{LRconj1} \lambda_c(L) -1 \sim\theta_cL^{-2/3}.
\end{equation}

\citet{DP99} established weak convergence of long-range contact
processes to a super-Brownian motion in dimensions\vadjust{\goodbreak} $d\geq2$, while for
$d=1$ \citet{MT95} showed that the scaling limit of the long-range
contact process is governed by the
stochastic PDE
%
%
\begin{equation}
\label{MTC}\frac{\partial u}{\partial t}=\frac
{u''}{6}+\theta u-u^2+
\sqrt{2u} \dot W.
\end{equation}
In this equation the local time density in the third integral of
\eqref{eqDWmgSIR} is replaced by the density, $u_t$, of
$X_{t}$. This reflects the fact that in the contact (and other $\mathit{SIS}$)
processes, the susceptible population is depleted
locally by the current size of the infected set. The results of
\citet{DP99} show this effect induces a killing term with a known
constant rate.
\citet{MT94} showed that
there is a phase transition in equation \eqref{MTC} in that there is
positive probability of survival for $\theta$ above a
critical $\theta_c>0$ and a.s. extinction below it. By contrast, the
martingale problems \eqref{eqDWmgSIR} have solutions in up to $3$
dimensions, whereas
\eqref{MTC} only makes sense in one spatial dimension (since
super-Brownian motion has the property that the mass distributions
$X_{t}$ at positive times are absolutely continuous only in dimension
$1$).

The discrete $\mathit{SIR}$ models underlying our continuous models are
described in Section~\ref{secdiscreteSIR} below. The analogue to
(1.6) in this
discrete setting is also described in Section~\ref{secdiscreteSIR}.

\subsection{Plan of the paper}\label{ssecplan} The remainder of the
paper is devoted to the proofs of Theorems
\ref{thmmmain}--\ref{thmmlocext}. The plan is as follows. In
Section~\ref{secpreliminariesepi} we discuss existence and
uniqueness of solutions to a class of martingale problems including
\eqref{eqDWmgSIR}, weak convergence of certain discrete processes
to spatial epidemics, and basic comparison principles. In
Section~\ref{secpreliminariessbm} we discuss some regularity
properties of (supercritical) super-Brownian motions and their local
time densities. In Sections~\ref{sseccvuniversal} and
\ref{sec1dext} we prove that the critical values $\theta_{c}$ in
dimensions $2$ and $3$ do not depend on the initial mass
distributions, and we prove that spatial epidemics in $\zz{R}^{1}$ die
out almost surely at all values of the transmission rate $\theta$. In
Section~\ref{secsurv} we prove that spatial epidemics in dimensions
$2$ and $3$ can survive if the transmission rate $\theta$ is
sufficiently high; and in Section~\ref{secext} we prove that at low
values of $\theta$ extinction is certain. We prove a weak form of
local extinction in Section~\ref{ssecweakLocalExtinction} and
finally, in Section~\ref{seclocalExtinction}, we prove
Theorem \ref{thmmlocext}.

\textit{Standing notation}.
For any $a\geq0$, $[a]$ stands for its integer part.
For any Borel subset $D\subseteq\zz{R}^{d}$, let $\mathcal{M} (D)$ be
the space of finite
Borel measures on $D$, equipped with the weak topology, and let
$\mathcal{M}_{c} (D)$ be the subset consisting of all measures with
compact support in $D$. These spaces are partially ordered in a natural
way: we write $\mu\leq\nu$ to mean that for all nonnegative, bounded
functions~$\varphi$,
\[
\int\varphi \,d\mu\leq\int\varphi \,d\nu.
\]
For a measure $\mu\in\mathcal{M} (D)$ and a nonnegative measurable
function $f\dvtx D \rightarrow\zz{R}_{+}$, we shall continue to use the shorthand
notation $\mu(f)$ or $\langle\mu,f \rangle$ to denote the integral
of $f$
against $\mu$ and also write $|\mu|$ for $\mu(1)$, the total mass
of $\mu$.
Let $C_b(\zz{R}^d)$
be the space of bounded and continuous functions on $\zz{R}^d$,
endowed with the sup-norm topology, and let $C_c(\zz{R}^d)$
be the space of compactly supported continuous functions on $\zz{R}^d$.
Furthermore, for any $x=(x_1,\ldots,x_d)\in\zz{R}^d$ and any $r>0$, let
$Q_r(x)=[x_1-r/2,x_1+r/2) \times\cdots\times[x_d-r/2,x_d+r/2)$ be the
(half-closed, half-open) cube of side length $r$ centered at $x$, and,
for notational ease,
$Q(x):=Q_{1}(x)$. Finally, let $C_p(\zz{R}^d,\zz{R}_+)$ be the space of
nonnegative piecewise constant functions on $\zz{R}^d$ satisfying the
following conditions: each such function is supported by $\bigcup_i
Q(x_i)$ for finitely many $x_i\in\zz{Z}^d$ and is constant on each
cube.

\textit{Conventions}. Throughout
the paper, $C, c, C_1$, \mbox{etc.} denote generic constants whose
values may change from line to line. The notation
$Y_{n}=o_{P} (f (n))$ means that $Y_{n}/f (n) \rightarrow0$ in
probability; and $Y_{n}=O_{P} (f (n))$ means that the sequence
$|Y_{n}|/f (n)$ is tight. Also, for any $a,b\in\zz{R}$,
$a\wedge b:=\min(a,b)$ and $a\vee b:=\max(a,b)$. Finally, we use
a ``local scoping rule'' for notation: any notation introduced in
a proof is local to the proof, unless otherwise indicated.

\section{Preliminaries on the epidemic processes}
\label{secpreliminariesepi}

\subsection{Dawson's Girsanov theorem; existence and
uniqueness}\label{secgirsanov}

Existence and uniqueness of solutions (in the weak sense) to a class
of martingale problems similar to \eqref{eqDWmgSIRgen} was
established in \citet{MT11} using Dawson's Girsanov theorem. Existence
in the special case $K\equiv0$, $\theta=0$ was also proved in
\citet{lalley09} and \citet{lz10} by weak convergence methods, which
extend trivially to the general case. Nevertheless, since Dawson's
Girsanov formula will be of crucial importance in many of the
arguments to follow, we begin by reviewing the essential facts. We
first state a variant of Dawson's Girsanov theorem [Theorem IV.1.6 in
\citet{PerkinsSFnotes}] tailored to our needs.

Let $\Omega=D([0,\infty);\mathcal{M}_{c}(\zz{R}^d))$ be the canonical
path space for compactly supported measure-valued processes, with
coordinate maps
$X_{t}\dvtx \Omega\rightarrow\mathcal{M}_{c} (\zz{R}^{d})$ and associated
filtration $\zz{F}=(\mathcal{F}^{X}_{t})_{t\geq0}$. Fix a
probability measure $P$ on $(\Omega,\mathcal{F}^{X}_{\infty})$, and
suppose that there is a linear mapping $\psi\mapsto( M_{t} (\psi)
)_{t\geq0}$ from the space $C^{2}_{c} (\zz{R}^{d})$ to the space of
$\zz{F}$-adapted, continuous martingales such that $M_{0} (\psi)=0$
and such that $M (\psi)$ has quadratic variation $[M
(\psi)]_{t}=\int_{0}^{t} \langle X_{s},\psi^{2} \rangle \,ds$. This mapping
extends to an \emph{orthogonal martingale measure} $dM (s,x)$; see
\citet{Walsh86}. For any previsible$\times$Borel process $B\dvtx \zz
{R}_{+}\times
\Omega\times\zz{R}^{d} \rightarrow\zz{R}$, we say that $B$ is
$L^{2}$-\emph{admissible} if
%
%
\begin{equation}
\label{eqB-admissible} \int_{0}^{t} \bigl
\langle X_{s},B_{s}^{2} \bigr\rangle \,ds <
\infty \qquad\mbox{for all }t\ge0\ P\mbox{-almost surely}.
\end{equation}
If $B$ is $L^{2}$-admissible, then the stochastic integrals
%
%
\begin{equation}
\label{eqWalshIntegral} \int_0^t\int
_{\zz{R}^d} B_{s} (x) \,dM (s,x)
\end{equation}
exist and constitute a continuous, $\zz{F}$-adapted local martingale
with quadratic variation process $\int_{0}^{t}\langle X_{s},B^{2}_{s}
\rangle
\,ds$.
Consequently, for each $\gamma>0$, the process
%
%
\begin{equation}
\label{eqGirsanovstd2} \aligned\mathcal{E}^B_t= \exp
\biggl( \frac{1}{\sqrt\gamma}\int_0^t\int
_{\zz{R}^d} B_{s} (x) \,dM(s,x) - \frac{1}{2\gamma}\int
_0^t \bigl\langle X_{s},B^{2}_{s}
\bigr\rangle \,ds \biggr) \endaligned
\end{equation}
is a continuous local martingale.

\begin{lemma}[(Dawson's Girsanov theorem)]\label{lemmaGirsanov}
Let $P$ be a probability measure on
$(\Omega,\mathcal{F}_{\infty})$ such that under $P$
the coordinate process $(X_{t})_{t\geq0}$
satisfies the following martingale problem: for some
uniformly bounded $L^{2}$-admissible integrand $A$, for all $\psi\in
C^{2}_{c} (\zz{R}^{d})$,
%
%
\begin{equation}
\label{eqGirX}\qquad X_t(\psi) = X_0(\psi) +
\frac{\alpha}{2}\int_0^t \langle
X_{s}, \Delta\psi \rangle \,ds + \int_0^t
\langle X_{s},A_{s} \psi \rangle \,ds+ \sqrt\gamma
M_t^P(\phi),
\end{equation}
where $M_t^P(\psi)$ is a continuous $\mathcal{F}^X_t$-martingale
with quadratic variation\break  $ [M^P(\psi)]_t=\int_0^t X_s(\psi^2) \,ds$.
\begin{longlist}[(a)]
\item[(a)] Suppose that $Q$ is another probability measure
on $(\Omega,\mathcal{F}^X_{\infty})$ such that under $Q$ the coordinate
process $(X_{t})_{t\geq0}$ satisfies the martingale problem
%
%
\begin{eqnarray}
\label{eqGirY} X_t(\psi)& =& X_0(\psi) +
\frac{\alpha}{2}\int_0^t\langle
X_{s}, \Delta\psi \rangle \,ds + \int_0^t
\bigl\langle X_{s},(A_{s}+B_{s}) \psi \bigr
\rangle \,ds
\nonumber
\\[-8pt]
\\[-8pt]
\nonumber
&&{} + \sqrt \gamma M_t^Q(\psi)
\end{eqnarray}
for all $\psi\in C^{2}_{c} (\zz{R}^{d})$, where
$B$ is a uniformly bounded $L^{2}$-admissible integrand,
and $M_t^Q(\psi)$ is a
continuous martingale (under $Q$) with quadratic variation $
[M^Q(\psi)]_t=\int_0^t X_s(\psi^2) \,ds$.\vspace*{1pt} Suppose also that the
restrictions of $P$ and $Q$ to the $\sigma$-algebra
$\mathcal{F}^{X}_{0}$ are equal. Then for each $t<\infty$ the measures
$P$ and $Q$ on $\mathcal{F}^X_{t}$ are mutually absolutely continuous,
with likelihood ratio
%
%
\begin{equation}
\label{eqGirsanovstd} \frac{dQ}{dP} \bigg\rrvert _{\mathcal{F}^X_t} =
\mathcal{E}^B_t.
\end{equation}
In particular, $Q$ is uniquely determined on $\mathcal{F}^X_{\infty}$
by the martingale
problem~\eqref{eqGirY}.

\item[(b)] Conversely, if $Q$ is the probability measure
determined by the likelihood ratios \eqref{eqGirsanovstd}, then
under $Q$ the process $X$ satisfies the martingale problem
\eqref{eqGirY}.
\end{longlist}
\end{lemma}

We next apply the above to prove that the martingale problem
\eqref{eqDWmgSIRgen} is well-posed. Recall that for each $x\in\zz
{R}^d$, $Q_r(x)$ stands for the cube of side length $r$ centered at\vadjust{\goodbreak}
$x$, and $Q(x)=Q_{1}(x)$. For any continuous path $X_{t}$ valued in
$\mathcal{M}_{c} (\zz{R}^{d})$, define
\[
L(t,X,x)=L_t(X,x)=L^{X}_{t} (x) = \limsup
_{\varepsilon\downarrow0}\frac{\int_0^t
X_s(Q_\varepsilon(x)) \,ds}{\varepsilon^d}.
\]
When there is no confusion, we shall suppress the dependence on $X$
and abbreviate $L_t(X,x)$ as $L_t(x)$. If $X_{t}$ is an adapted
process on the filtered space $(\Omega, \zz{F}) $, then $L (t,X,x)$ is
nonnegative, nondecreasing in $t$, and $\mathcal{P}\times
\mathcal{B}^{d}$-measurable, where $\mathcal{B}^{d}$ is the Borel
$\sigma$-field on $\zz{R}^d$, and $\mathcal{P}$ is the previsible
$\sigma$-field. {If $X$ has a local time density, $L(t,X,x)$ will be a
jointly measurable version of it.}

\begin{thmm}\label{thmmexistuniq}
Let $\mu\in\mathcal{M}_{c} (\zz{R}^{d})$ satisfy
Assumption \ref{asmtpinireg1}, and let $K\in
C_p(\zz{R}^d,\zz{R}_+)$. For any fixed $\theta\in\zz{R}$ and
$\gamma>0$, denote by $P_{\mu}=P^{\theta,0,\gamma}_{\mu,0}$ the
law of
a super-Brownian motion with initial mass distribution $\mu$, drift
$\theta$ and branching rate $\gamma$.
\begin{longlist}[(a)]
\item[(a)] If $X$ solves the martingale problem
\eqref{eqDWmgSIRgen} with initial value $X_{0}=\mu$,
then the law $P_{\mu,K}:=P^{\theta,\beta,\gamma}_{\mu,K}$ of $X$
on the
canonical path space is unique and given by
%
%
\begin{equation}
\label{eqRNformulaLT}\quad  \frac{dP_{\mu,K}}{dP_{\mu}}\bigg\rrvert _{\mathcal
{F}_t^X} =
\mathcal{E}^{B}_{t}, \qquad\mbox{where } B (s,\omega,x)= - \bigl(K
(x)+\beta L (s,X,x) \bigr),
\end{equation}
and $dM(s,x)$ is the orthogonal martingale measure under $P_{\mu}$.
Conversely, if $P_{\mu,K}$ is the probability measure specified by
\eqref{eqRNformulaLT}, then under $P_{\mu,K}$ the coordinate process
$X_{t}$ satisfies the martingale problem \eqref{eqDWmgSIRgen}.

\item[(b)] The mapping $(\mu,K)\mapsto P_{\mu,K}$ is jointly measurable
with respect to the appropriate Borel fields.

\item[(c)] Under $P_{\mu,K}$ the local time process $L_{t} (x)$ is jointly
continuous in $(t,x)$ and almost surely is the density of the
occupation measure $R_{t}=\int_{0}^{t}X_{s} \,ds$.

\item[(d)] Under the measure $P_{\mu,K}$ the process $(X,L)$ is strong
Markov, that is, for any $\mathcal{F}^X_t$-stopping time $\tau$,
\begin{eqnarray}
P_{\mu,K}(X_{\tau+\cdot}\in A|\mathcal{F}_\tau)
=P_{X_\tau,K+\beta L_\tau}(A) \nonumber\\
\eqntext{\mbox{almost surely on }\{ \tau <\infty\} \mbox{ for all }A\in\mathcal{F}^X_{\infty}.}
\end{eqnarray}

\item[(e)] For any pair $K,K'$ of suppression rate functions, the probability
measures $P_{\mu,K}$ and $P_{\mu,K'}$ are mutually absolutely
continuous on $\mathcal{F}^{X}_{t}$, with likelihood ratio
%
%
\begin{eqnarray}
\label{eqRNformulaLKK} \frac{dP_{\mu,K'}}{dP_{\mu,K}}\bigg\rrvert _{\mathcal{F}_t^X}
& =&
\exp \biggl\{\frac{1}{\sqrt\gamma}\int_0^t\int
\bigl(K(x)-K'(x) \bigr)\,dM_K(s,x)
\nonumber
\\[-8pt]
\\[-8pt]
\nonumber
&&\hspace*{19pt}{} -\frac{1}{2\gamma}\int_0^t\int
\bigl(K(x)-K'(x) \bigr)^2 X_s(dx)\,ds
\biggr\},
\end{eqnarray}
where $dM_K(s,x)$ is the orthogonal martingale measure under
$P_{\mu,K}$.\vadjust{\goodbreak}
\end{longlist}
\end{thmm}

\begin{remark}\label{remarkrandomInitialConditions}
Assertion (b) guarantees that if $X_{0}$ and $K_{0}$ are random and
$\mathcal{F}_{0}^X$-measurable, then the random probability measure
$P_{X_{0},K_{0}}$ is $\mathcal{F}_{0}^X$-measurable. Similarly, if
$X_{\tau}$ and $K_{\tau}$ are $\mathcal{F}_{\tau}^X$-measurable, then
$P_{X_{\tau},K_{\tau}}$ is $\mathcal{F}_{\tau}^X$-measurable. Moreover,
since $P_{X_{0},K_{0}}$ is a regular conditional distribution on the
canonical path space given $\mathcal{F}_{0}^X$, it follows from (d) that
the strong Markov property holds when the initial condition $X_{0}$
and the suppression rate function $K_{0}$ are random.
\end{remark}

\begin{remark}\label{remarkadmissibility}
Since the local time density $L_{t}$ is not uniformly bounded on finite time
intervals, the exponential process $\mathcal{E}^{B}_{t}$
is not a priori a martingale.
Part of the assertion of the
theorem is that in fact $\mathcal{E}^{B}_{t}$ \emph{is} a martingale,
and hence that \eqref{eqRNformulaLT} defines a probability measure
on $\mathcal{F}^{X}_{t}$.
\end{remark}

\begin{pf*}{Proof of Theorem \ref{thmmexistuniq}}
(a) First we claim that any solution $X$ to
martingale problem \eqref{eqDWmgSIRgen} has the property that its
local time density $L^{X}_{t} (x,\omega)$ is bounded in $(t,x)$ for
$t$ in finite intervals and for every $t<\infty$ has
compact support in $x$ for almost every $\omega$. This follows because
on some probability space a version of $X_{t}$ can be coupled with a
super-Brownian motion $\ol{X}_{t}$ with drift $\theta$ and branching
rate $\gamma$ such that $\ol{X}_t\ge X_t$ for all $t\ge0$
\mbox{almost surely}.
See Proposition IV.1.4 in
\citet{PerkinsSFnotes} which we apply with $D=0$, $C_t(\phi)=\int_0^t
X_s(L^X_s\phi) \,ds$, and
only to the first coordinate of the pair of processes considered there.
To apply the above result we need to show that $t\to C_t$ is a
continuous $\sM(\bR^d)$-valued process. For $\phi\in C_c^2(\bR^d)$,
$C_t(\phi)$ is continuous by the martingale problem. It is easy to
extend the martingale problem to $\phi=1$ by taking limits and the
continuity of $t\to C_t(1)$ follows. This establishes the required
continuity.
Since super-Brownian motion has a continuous
local time process with compact support in any finite time interval,
by Sugitani's theorem, it follows that the process $X$ also has
a local time density $L^{X}_t(x)$ with the advertised properties.

Unfortunately,
we cannot directly apply the previous lemma to conclude \eqref{eqRNformulaLT},
because $L_{t}^{X}$ is not
\emph{uniformly} bounded in $\omega$. To circumvent this
problem we use a localization argument. Fix $0<b<\infty$, and consider
the exponential process $\mathcal{E}^{B}_{t\wedge\tau(b)}$, where
\[
\tau(b)=\inf \Bigl\{t \dvtx \max_x\bigl |B_{t}(x)\bigr|
\geq b \Bigr\}.
\]
By Lemma \ref{lemmaGirsanov}, the process $\mathcal{E}^{B}_{t\wedge
\tau(b)}$ is a martingale, and so
under the probability measure $Q^{b}$
specified by equation
\eqref{eqRNformulaLT}
(with the stopped
exponential martingale as the likelihood ratio), the process $X$
satisfies the martingale problem~\eqref{eqDWmgSIRgen} with $K
(x)+\beta L (s,X,x)$ replaced by its stopped value. But the
preceding paragraph implies that for each $t$,
$Q^{b}( \tau(b)\leq t) \rightarrow0$ as $b \rightarrow
\infty$, that is, $\lim_{b \rightarrow\infty}E_{\mu} (\mathcal
{E}^{B}_{t\wedge\tau(b)}\indic_{ \tau(b)\leq t})=0$. Therefore,
\begin{eqnarray*}
E_{\mu} \bigl(\mathcal{E}^{B}_{t} \bigr) &\geq&
E_{\mu} \bigl(\mathcal{E}^{B}_{t}
\indic_{ \tau(b)> t} \bigr)=E_{\mu} \bigl(\mathcal{E}^{B}_{t\wedge
\tau(b)}
\indic_{ \tau(b)> t} \bigr)
\\
&=& E_{\mu} \bigl(\mathcal{E}^{B}_{t\wedge\tau(b)} \bigr) -
E_{\mu} \bigl(\mathcal{E}^{B}_{t\wedge\tau(b)}
\indic_{ \tau(b)\leq t} \bigr) \to1.
\end{eqnarray*}
On the other hand, by Fatou's lemma, $E_{\mu} (\mathcal
{E}^{B}_{t})\leq
1$, and so $E_{\mu} (\mathcal{E}^{B}_{t})=1$.
It follows that $\mathcal{E}^{B}_{t}$ is a martingale under $P_{\mu}$
and that under the probability measure defined by
\eqref{eqRNformulaLT} the process $X$ satisfies the martingale
problem~\eqref{eqDWmgSIRgen}.

(b, c) These are easy consequences of (a), the continuity
of $\mu\to P_{\mu}$, and Sugitani's theorem.

(d) It suffices to consider a
finite-valued $\tau$. By (c), the local time $L_{t}$ is the occupation
density of $X$ under $P_{\mu,K}$, so it follows that
\[
L_{\tau+t}(X,x)=L_\tau(X,x)+L_t(X_{\tau+\cdot},x)
\qquad\mbox{for all }(t,x) \mbox{ almost surely}.
\]
If $Q(\omega)$ is a regular conditional probability for
$X_{\tau+\cdot}$ given $\mathcal{F}_\tau$, then it follows easily from
this that \mbox{almost surely} under $Q(\omega)$ the coordinate
process satisfies the martingale problem \eqref{eqDWmgSIRgen} with
$K$ replaced by $K+L_{\tau}$. Therefore, by the uniqueness in law of
solutions, $Q(\omega)=P_{X_\tau(\omega),K(\omega)+\beta
L_\tau(\omega)}$ almost surely. The strong Markov property now
follows.

(e) This follows immediately from (a).
\end{pf*}

In the course of proving (a) we have also established the following:
%
\begin{prop}\label{propBRWenvelope}
Let $X$ be a solution of the martingale problem
\eqref{eqDWmgSIRgen} where $\mu$ and $K$ are as in Theorem \ref
{thmmexistuniq}. Then on some probability space, a~version of
$X$ can be coupled with a dominating super-Brownian motion $\overline
X$, with the same initial mass
distribution $\mu$,
and drift $\theta$, so that
$\overline{X}_{t}\geq X_{t}$ for all $t\ge0$ a.s. We will call
$\overline X$ the
\emph{super-Brownian motion envelope}.
\end{prop}

\begin{remark}
Lemma \ref{lemmaGirsanov} holds equally well on the
larger space of continuous $\mathcal{M}(\zz{R}^{d})$-valued paths
[as in
\citet{PerkinsSFnotes}]. The proof of Theorem \ref{thmmexistuniq}
also holds on this larger space if one starts with compactly
supported initial conditions. That is, the solutions necessarily
have compact supports for all $t$ by the domination in (a). This
slightly strengthens the uniqueness part and may be used
implicitly below without further comment. The main reason for
restricting to compactly supported measures is the use of
Proposition \ref{propSIRconv} below in the proof of our main result
Theorem \ref{thmmmain}.
\end{remark}

\subsection{Discrete epidemic models}\label{secdiscreteSIR}

Measure-valued processes that satisfy the martingale problem
\eqref{eqDWmgSIR} arise naturally as weak limits of
discrete, finite-population stochastic models of spatial epidemics.
Here we describe one such class of models, following \citet{lalley09}
and \citet{lz10}. Several of the couplings we shall develop later in
the paper involving measure-valued epidemics will be constructed by first
building corresponding couplings for discrete epidemics, then using
the weak convergence of the discrete to the measure-valued processes
to prove that they extend to the measure-valued setting.

The discrete \emph{SIR}-$d$ epidemic models take place in populations
of size $N$ located at each of the sites of the integer lattice
$\zz{Z}^{d}$. We shall call $N$ the \emph{village size}. Each of the
$N$ individuals (or \emph{particles}) at a site $x\in\zz{Z}^{d}$ may
at any time be either \emph{susceptible, infected}, \emph{recovered}
or \emph{removed}. Infected individuals remain infected for one unit
of time, and then recover, after which they are immune to further
infection. The rules governing the transmission of infection are as
follows: at each time $i=1,2,\ldots,$ for each pair $(i_x, s_y)$ of
an infected individual located at $x$ and a susceptible individual at
$y$, $i_x$ infects $s_y$ with probability $p_N(x,y)$, where
%
%
\begin{eqnarray}
\label{eqpN} p_N(x; y)=p_N^\theta(x; y)
 &=&
\frac{1+\theta/N^{\alpha}}{(2d+1)N}\qquad
\mbox{if } |y- x|\leq1 \mbox{ and }
\nonumber
\\[-8pt]
\\[-8pt]
\nonumber
&=& 0\qquad\mbox{otherwise,}
\end{eqnarray}
where $|z|$ is the Euclidean norm of $z$ and
\[
\alpha=\alpha(d) = 2/(6-d)
\]
is the critical exponent; see Theorem 1 in \citet{lalley09} and
Theorem 2 in \citet{lz10}. For the \emph{SIR}-$d$ model with
village size $N$, define
\begin{eqnarray*}
\zz{X}_i^N(x)&:=&\mbox{set of
infected particles at } x \mbox{ at time } i;\qquad X_i^N(x):=
\bigl|\zz{X}_i^N(x)\bigr|;
\\
\zz{K}^N(x)&:=&\mbox{set of removed particles at } x
(\mbox{at time } 0);\qquad K^N(x):=\bigl|\zz{K}^N(x)\bigr|;
\\
\zz{R}_n^N(x)&:=&\mbox{set of recovered
particles at } x \mbox{ at time } n; \qquad R_n^N(x):=\bigl|
\zz{R}_n^N(x)\bigr|;
\\
\zz{X}_i^N&:=&\bigcup_x
\zz{X}_i^N(x),\qquad \zz{K}^N:= \bigcup_x
\zz{K}^N(x)\quad \mbox{and}\quad \zz{R}_n^N:=
\bigcup_x \zz{R}_n^N(x).
\end{eqnarray*}

Theorem \ref{thmmmain} and Proposition \ref{propSIRconv} below
suggest, after an interchange of
limits, that the critical infection probability $p_c(N)$ for the SIR
model satisfy
%
%
\begin{eqnarray}
\label{SIRCRIT}0<cN^{-\alpha}\le(2d+1)N\cdot p_c(N) - 1 \le C
N^{-\alpha}
\nonumber
\\[-8pt]
\\[-8pt]
\eqntext{\mbox{for large }N \mbox{ and }d=2,3.}
\end{eqnarray}
This would be consistent with the
result \eqref{LRC} for the long-range contact
process.
Whether or not there is a stronger relation [as in \eqref{LRconj1}]
involving the exact constants~$\theta_c$ in Theorem \ref{thmmmain} is
another interesting open question.

\textit{The standard construction}.
We now describe a way to construct this process using a
\emph{percolation structure}. Connections between SIR epidemics
and bond percolation go back at least to \citet{mo77} (see page 322) in
the continuous setting and were
used extensively by \citet{cd88}, again in the continuous time setting.
The construction we use is a modification of the
constructions in \citet{lalley09} and \citet{lz10}. We shall call this
the \emph{standard construction}. The percolation structure is a
random graph with vertex set $\zz{Z}^{d}\times\{1,2,\ldots,N\}$; the vertex
$(x,i)$ represents the $i$th individual (or particle) in the
``village'' $\mathcal{V}_{x}$ situated at location $x\in\zz{Z}^{d}$.
For each pair $(x,i)$ and $(y,j)$ of vertices whose spatial locations
differ by at most $1$ (i.e., $|x-y|\leq1$), a $p_{N}$-coin toss
determines whether or not there is an edge between $(x,i)$ and
$(y,j)$. (As is often the case in such constructions, it is useful,
for comparison purposes, to assume that these coin tosses are realized
using independent $\operatorname{uniform}[0,1]$ random variables.) Thus, the
resulting random
graph $\mathcal{G}=\mathcal{G}^{N}$ has edges only between vertices in
the same or neighboring villages.

The spatial epidemic is defined by a deterministic algorithm on the
random graph $\mathcal{G}$. Since the village size $N$ is fixed in
this algorithm, we shall omit all superscripts $N$ in the
specification of the algorithm. The colors \emph{green}, \emph{blue},
\emph{red} and \emph{black} will be used to denote susceptible,
infected, recovered and removed vertices in each generation. For the
$0$th generation, designate $K (x)$ vertices at location $x$ as
\emph{black}; the set of black vertices will not change during the
course of the epidemic. Next, color $X_{0}(x)$ vertices in
$\mathcal{V}_{x}$ \emph{blue}, and all remaining vertices
\emph{green}. (Thus, in generation~$0$ there are no red vertices.)
Now define a time evolution as follows. In generation $n+1$, the set
$\zz{X}_{n+1}$ of blue vertices will consist of all vertices that
were green in generation $n$ \emph{and} were connected by edges of the
random graph to blue vertices (i.e., vertices in
$\zz{X}_{n}$). Finally, all vertices that were blue in
generation $n$ become red in generation $n+1$, and remain red in all
subsequent generations (i.e., $\zz{R}_{n+1}=\zz{R}_{n}\cup
\zz{X}_{n}$).

The virtue of this construction is that all quantities of interest can
easily be described in terms of the geometry of the random graph
$\mathcal{G}^{\prime}=\mathcal{G}\setminus\zz{K}$
obtained by deleting all black vertices from $\mathcal{G}$. The set
$\zz{X}_{n}$ consists of all vertices at distance $n$
in the graph $\mathcal{G}^{\prime}$ from the set of vertices that were
colored blue in
generation $0$. Similarly, the set $\zz{R}_{n}$ consists of
all vertices at distance $<n$ from the blue vertices in
generation $0$. The set $\zz{R}_{\infty}$ of vertices
that are \emph{ever} infected during the course of the epidemic is the
union of the connected clusters of the blue vertices of generation~$0$
in $\mathcal{G}^{\prime}$. It is immediately obvious from this that the
recovered sets $\zz{R}_{n}$ are nonincreasing in the initial
condition $\zz{K}$, and nondecreasing in $\zz{X}_{0}$ and the
transmission parameter $\theta$.

Denote by
%
%
\begin{equation}
\label{eqPn} \zz{P}^{n} = \bigl(P_{n} (x,y)
\bigr)_{x,y\in\zz{Z}^{d}}= \bigl(P_{n} (y-x) \bigr)_{x,y\in
\zz{Z}^{d}}
\end{equation}
the transition probability kernel of the simple random walk on
$\zz{Z}^{d}$, that is, $\zz{P}^{n}=\zz{P}*\zz{P}^{n-1}$ is the
$n$th convolution power of the one-step transition probability
kernel given by
%
%
\begin{equation}
\label{eqsrw} P_{1} (x,y)= 1/ (2d+1) \qquad\mbox{for } |y-x|\leq1,
\end{equation}
and let $\sigma^{2}=2/ (2d+1)$ be the variance of the distribution
$P_{1} (0,\cdot)$. Let $G_{n} (x,y)$ be the associated Green's
function
\[
G_n(x, y ):=\sum_{1\le i< n}
P_i(x, y),\qquad G_n(x):= G_n(0,x),
\]
and for any finite measure $\mu$ on $\zz{Z}^{d}$ denote by $\mu G_{n}
(x)=(\mu*G_{n}) (x)$ the convolution of $\mu$ with $G_{n}$.

Next, we explain the re-scaling of the discrete epidemics that gives
weak convergence to the measure-valued epidemics determined by the
martingale problem~\eqref{eqDWmgSIRgen}.

\begin{dfn}\label{definitionrescaling}
The \emph{Feller--Watanabe scaling operator} $\mathcal{F}_{N}$ scales
mass by $1/N^{\alpha}$ and space by
$1/\sqrt{N^{\alpha}\sigma^{2}}$,
that is, for any finite Borel measure $\mu$ on $\zz{R}^{d}$ and
any test function $\phi$,
%
%
\begin{equation}
\label{eqFellerScaling} \langle\phi,\mathcal{F}_{N} \mu \rangle=
N^{-\alpha}\int\phi \bigl(x/\sqrt{N^{\alpha} \sigma^2}
\bigr) \mu(dx).
\end{equation}
\end{dfn}

\begin{dfn}\label{definitionSugitaniScaling}
The \emph{Sugitani scaling operator} $\mathcal{S}_{N}$ scales mass
by\break
$1/N^{\alpha(2-d/2)} $ and space by
$1/\sqrt{N^{\alpha}\sigma^{2}}$,
that is, for any function $f$,
%
%
\begin{equation}
\label{eqSugitaniScaling} (\mathcal{S}_{N}f) (x)= \frac{f (\sqrt {N^{\alpha} \sigma^2} x)}{N^{\alpha(2-d/2)}}.
\end{equation}
When the function $f$ is only defined for $x\in\zz{Z}^d$, define
$(\mathcal{S}_{N}f) (x)$ for $x\in\zz{Z}^d/[\sqrt{N^{\alpha}
\sigma
^2}]$ as above, and extend it to a continuous function on $\zz{R}^{d}$
by a suitable piecewise linear interpolation.
\end{dfn}

The following weak convergence theorem is a slight variant of the
main results in \citet{lalley09} and \citet{lz10}.

\begin{prop}\label{propSIRconv} Assume that $d\leq3$, and suppose
that the initial configurations
$\mu^N:=X_0^N$ and $K^N$
are both supported by finitely many integer sites and are
such that for some measure $\mu$
satisfying Assumption \ref{asmtpinireg2} below and some $K\in
C_p(\zz
{R}^d,\zz{R}_+)$, the following conditions are satisfied, where
$\Longrightarrow$ denotes weak convergence on the respective spaces:
\begin{longlist}[(a)]
\item[(a)] if $d=1$, then
${\mu^N(\sqrt{N^\alpha\sigma^2} \cdot)}/\sqrt{N^\alpha}$
are supported by a common compact interval, and (after linear
interpolation to be continuous functions on $\zz{R}$)
%
%
\begin{equation}
\label{hyp1dini} \frac{\mu^N(\sqrt{N^\alpha\sigma^2} x)}{\sqrt {N^\alpha}}\quad \Longrightarrow \quad X_0(x)\in
C_c(\zz{R});
\end{equation}
\item[(b)] if $d=2$ or $3$, then
%
%
\begin{eqnarray}
\label{hypinitialConfig}  \mathcal{F}_{N}\mu^{N}\quad
&\Longrightarrow&\quad\mu,
\\
\label{hypsmoothness}  \mathcal{S}_{N} \bigl(\mu^{N}
*G_{[N^{\alpha}t]} \bigr)\quad&\Longrightarrow&\quad\mu*q_{t}\in C_b
\bigl([0,\infty )\times\zz{R}^{d}\bigr),
\end{eqnarray}
where the second convergence is in $D([0,\infty);C_b(\zz{R}^{d}))$;

\item[(c)] in all dimensions,
%
%
\begin{equation}
\label{hypK} K^{N}(x) = \bigl[N^{\alpha(2-d/2)}\cdot K \bigl(x/ \bigl[
\sqrt{N^{\alpha} \sigma^2} \bigr] \bigr) \bigr]\qquad \mbox{for all
} x\in\zz{Z}^d.
\end{equation}
\end{longlist}
Then we have the following weak convergence:
%
%
\begin{equation}
\label{eqSIR-convergence} \bigl(\mathcal{F}_{N}X^{N}_{[N^{\alpha}t]},
\mathcal{S}_{N}R^{N}_{[N^{\alpha}t]} \bigr)\quad \Longrightarrow\quad
\bigl( X_{t}, L_{t} (x) \bigr)
\end{equation}
in $D([0,\infty);\mathcal{M}_c(\zz{R}^d))\times D([0,\infty
);C_b(\zz{R}^d))$,
where the limit process $X$ has initial configuration $X_{0}=\mu$,
solves \eqref{eqDWmgSIRgen} with $\gamma=1$, $\beta=1$, $\theta
$ as
in \eqref{eqpN}, and suppression rate $K$, and $L_t(x)$ is its
local time density process.
\end{prop}

\begin{asmptn}\label{asmtpinireg2} The finite measure $\mu$ has
compact support. When $d=1$, $\mu$ has a density $X_0(x)\in C_c(\zz
{R})$, and for $d=2,3$ for some $C_\mu>0$, $\mu$ satisfies
%
%
\begin{eqnarray}
\label{conc} &&\sup_{x\in\zz{R}^d}\mu \bigl(B(x,r) \bigr)
\nonumber
\\[-8pt]
\\[-8pt]
\nonumber
&&\qquad\le
\cases{C_{\mu}(\log1/r)^{-3},& \quad$\mbox{if }d=2,$\vspace*{2pt}
\cr
C_\mu r (\log1/r)^{-2},&\quad$\mbox{if }d=3$} \qquad\mbox{for
all $r\in(0,1]$}.
\end{eqnarray}
\end{asmptn}

\begin{remark}\label{rmktwoasmptn}
It is easy to see that Assumption \ref{asmtpinireg2} implies
Assumption~\ref{asmtpinireg1}. Take the case $d=3$, for example. For
any $(t_n,x_n)\to(t,x)$, we want to show that $\int_y q_{t_n}(y-x_n) \,d\mu(y) \to\int_y q_{t}(y-x) \,d\mu(y)$. Since $\mu(\{x\})=0$, we have
\[
q_{t_n}(y-x_n)\to q_t(y-x)\qquad\mbox{for } \mu
\mbox{-a.a. }y,
\]
and hence it suffices to show that $\{q_{t_n}(y-x_n)\}$ is uniformly
integrable with respect to $\mu$, which in turn reduces to show
\[
\lim_{\delta\to0} \sup_n \int
_{|y-x_n|\leq\delta} q_{t_n}(y-x_n) \,d\mu(y) =0.
\]
To see this, let $M(r)=\mu(B(x_n,r))$ for $r\ge0$.
The elementary bound $q_t(z)\le C|z|^{-1}$ and an integration by parts
lead to
%
%
\begin{eqnarray}
\label{eqintparts} \int_{|y-x_n|\leq\delta} q_{t_n}(y-x_n)
\,d\mu(y) &\leq& C\int_{|y-x_n|\leq\delta} |y-x_n|^{-1}
\,d\mu(y)
\nonumber\\
&=&C\int_0^\delta r^{-1} \,dM(r)
\nonumber
\\[-8pt]
\\[-8pt]
\nonumber
&=&Cr^{-1}M(r)|_0^\delta+ C\int
_0^\delta r^{-2} M(r) \,dr
\\
&\leq& C \bigl(\log(1/\delta) \bigr)^{-2} + C\int_0^\delta
r^{-1} \bigl(\log(1/r) \bigr)^{-2} \,dr,\nonumber
\end{eqnarray}
which goes to $0$ as $\delta\to0$.
A similar argument applies for $d=2$.
\end{remark}

\begin{remark}\label{rmkasmptn2trueforsbm}
For any $\mu\in\sM_c(\zz{R}^d)$ and any fixed $\theta\in\zz{R},
\gamma
>0$, by Theorems III.4.2 and III.3.4. in \citet{PerkinsSFnotes}, if $X$
is a super-Brownian motion with initial mass distribution $\mu$, drift
$\theta$ and branching rate $\gamma$, then Assumption \ref
{asmtpinireg2} is satisfied by $X_t$ for all $t>0$ almost surely.
Furthermore, for any $K\in C_p(\zz{R}^d,\zz{R}_+)$ and $\beta>0$, by
the absolute continuity between the laws $P^{\theta,0,\gamma}_{\mu,0}$
and $P^{\theta,\beta,\gamma}_{\mu,K}$, the same is true for a spatial
epidemic with initial mass distribution $\mu$, local suppression rate
$K$, transmission rate $\theta$, branching rate $\gamma$ and inhibition
parameter $\beta$.
\end{remark}

\begin{remark}\label{remarkinitialApproximation}
For $\mu$ satisfying Assumption \ref{asmtpinireg2} there are
rescaled counting measure $\mu^N$'s satisfying the hypotheses of the
above theorem, and hence
Proposition \ref{propSIRconv} implies, among other things, that for
any suppression rate
function $K\in C_p(\zz{R}^d,\zz{R}_+)$, the measure-valued epidemic
process $X$ satisfying \eqref{eqDWmgSIRgen} is a weak limit of
appropriately scaled discrete \emph{SIR} epidemics. When $d=1$, for
each $x\in\zz{Z}/\sqrt{N^\alpha\sigma^2}$, let
$\mu^{N}(x\sqrt{N^\alpha\sigma^2})=[\sqrt{N^\alpha}\cdot
X_0(x)]$. Then
\eqref{hyp1dini} is obvious.
When $d=2$ or 3, the required sequence $\{\mu^N\}$ satisfying \eqref
{hypinitialConfig} and \eqref{hypsmoothness} can be built as follows.
Recall that for each $x\in\zz{Z}^d$, $Q(x)$ stands for the
(half-closed, half-open) unit cube centered at $x$. $\zz{R}^d$ can
hence be decomposed as a nonoverlapping union of $Q(x)$'s for $x\in
\zz
{Z}^d$, and so for any $y\in\zz{R}^d$, we can find a unique $x\in\zz
{Z}^d$ such that $y\in Q(x)$, and with a slight abuse of notation,
denote such an $x$ by $[y]$. Next, let $\{X_i\}$ be a sequence of \mbox
{i.i.d.} random variables with probability distribution $\mu/|\mu|$,
and let
\[
\mu^N = \sum_{i=1}^{[N^\alpha\cdot|\mu|]}
\delta_{[X_i\sqrt{N^\alpha
\sigma^2}]}.
\]
(Note that $\alpha<1$ and on each integer site there are $N$ vertices,
so for all $N$ large enough, $\mu^N$ can be realized as a counting
measure on the graph $\zz{Z}^{d}\times\{1,2,\ldots,N\}$.)
One can then show that $\{\mu^N\}$ satisfies
\eqref{hypinitialConfig} and \eqref{hypsmoothness} almost surely. In
fact, \eqref{hypinitialConfig} holds trivially by the strong law of
large numbers (SLLN), the uniform continuity of test functions and the
simple bound
%
%
\begin{equation}
\label{eqrounddist} \biggl\llvert\frac{[X_i \sqrt{N^\alpha
\sigma
^2}]}{\sqrt{N^\alpha\sigma^2}} - X_i \biggr
\rrvert\leq\frac{1}{\sqrt{N^\alpha\sigma^2}} \qquad\mbox {for all }i.
\end{equation}
The verification of \eqref{hypsmoothness} is given in the \hyperref[app]{Appendix}.
\end{remark}

\begin{remark}\label{remarkjointConvergenceSIR}
The arguments of \citet{lalley09} and \citet{lz10} are based on the fact
that each of the discrete \emph{SIR} epidemics has law absolutely
continuous with respect to the law of a critical branching random walk
with the same initial condition. The Radon--Nikodym derivatives can be
written explicitly as products, and these can be shown to converge to
exponentials of the form $\mathcal{E}^{B}_{t}$ appearing in
\eqref{eqRNformulaLT}. Since branching random walks, after rescaling,
converge to super-Brownian motions, it follows that the rescaled discrete
\emph{SIR} epidemics converge to processes related to super-Brownian
motion by \eqref{eqRNformulaLT}, that is, processes that solve the
martingale problem \eqref{eqDWmgSIR}.

Routine modifications of these arguments can be used to establish weak
convergence for a variety of discrete processes similar to or related
to the discrete \emph{SIR} epidemics constructed above. In particular,
the convergence \eqref{eqSIR-convergence} can be extended to
\emph{joint} weak convergence for \emph{coupled} \emph{SIR} epidemics
with suitable initial conditions. For example, let
$\mu^{N,A},\mu^{N,B}$ be initial conditions satisfying the hypotheses
\eqref{hyp1dini}--\eqref{hypsmoothness}, and let
$\zz{X}^{N,A},\zz{X}^{N,B},\zz{X}^{N}$ be discrete \emph{SIR}
epidemics all constructed using the same percolation structure
$\mathcal{G}^{N}$, with the same initially removed sets
$\zz{K}^{N} (x)$, in such a way that the sets $\zz{X}^{N,A}_{0}
(x)$ and $\zz{X}^{N,B}_{0} (x)$ are nonoverlapping, with
cardinalities $\mu^{N,A} (x)$ and $\mu^{N,B}(x)$, and such that
%
%
\begin{equation}
\label{eqdisjointUnion} \zz{X}^{N}_{0} (x) =
\zz{X}^{N,A}_{0} (x)\cup\zz{X}^{N,B}_{0}
(x).
\end{equation}
Then after rescaling, the processes
$X^{N,A},X^{N,B} $ and $X^{N}$ converge \emph{jointly} in law to
(dependent) measure-valued epidemics $X^{A}_{t}, X^{B}_{t}$ and
$X_{t}$, with initial mass distributions $\mu^{A},\mu^{B}$ and
$\mu^{A}+\mu^{B}$, respectively, whose local time densities satisfy
%
%
\begin{equation}
\label{eqcoupledLocalTimes} L^{A}_{t}\vee
L^{B}_{t}\leq L_{t}\leq L^{A}_{t}+L^{B}_{t}.
\end{equation}

(The arguments that follow will not rely in an essential way on this
joint convergence, however. All that is needed is that
\emph{subsequences} converge jointly, as this is enough to guarantee
the existence of coupled measure-valued processes satisfying the same
monotonicity properties [such as \eqref{eqcoupledLocalTimes}]. Joint
convergence along subsequences follows trivially from the weak
convergence of marginals, since this implies joint tightness.)
\end{remark}

\subsection{Comparison lemmas}\label{seccomparisons}
The construction of the measure-valued spatial epidemic process as the
weak limit of discrete epidemics and the Girsanov formulas~\eqref{eqRNformulaLT}--\eqref{eqRNformulaLKK} lead to a
number of basic comparison principles that will be used in the proof
of Theorem \ref{thmmmain}. We formulate these as
\emph{couplings}, in which two epidemic processes (or super-Brownian
motions) are constructed on a common probability space in such a way
that various functionals of the processes [e.g., the limiting local
time densities $L_{\infty} (x)$] are ordered.

\begin{lemma}\label{lemmasurvchar}
Suppose that $X_t$ has law $P^{\theta,1,\gamma}_{\mu,K}$ for some
$K\in C_{p} (\zz{R}^{d},\zz{R}_{+})$ and some initial condition $\mu$
that satisfies Assumption \ref{asmtpinireg1}. Then
\[
P(X \mbox{ survives}) = P \biggl(\lim_{t\to\infty}\int
_{\zz{R}^d} L_t(x) \,dx =\infty \biggr)=P \Bigl(\lim
_{t\to\infty}|X_t|=\infty \Bigr).
\]
\end{lemma}

\begin{pf}
This uses the existence of a coupling between the measure-valued epidemic
$X$ and its super-Brownian motion envelope (Proposition~\ref
{propBRWenvelope}).
If $Z_s:=|X_s|$ is the total mass at time $s$, then
$\lim_{t\to\infty}\int_{\zz{R}^d} L_t(x) \,dx = \int_0^\infty Z_s
\,ds$. Because $Z_s$ is continuous, and 0 is an absorbing
state (\mbox{e.g.}, by the strong Markov property in
Theorem \ref{thmmexistuniq}), if $\int_0^\infty Z_s \,ds
=\infty$, then $X$ must survive. On the other hand, if
$\liminf_{t\to\infty}Z_t <\infty$, then there exists $M\in\zz{N}$ and
an infinite sequence of stopping times $\tau_n\to\infty$ such that
%
%
\begin{equation}
\label{eqsmallmass} \tau_{n+1}\geq\tau_n + 1\quad \mbox{and}\quad Z_{\tau_n}\leq M.
\end{equation}
Consider the time period $[\tau_n,\tau_n+1]$. By the strong Markov
property and the existence of a monotone coupling between a spatial
epidemic and its super-Brownian motion envelope, the process
$Z_{t+\tau_n}$ is dominated by
Feller diffusion with drift $\theta$ and initial total mass less
than $M$. This dies out in the next one unit of time with
positive probability, independent of $n$, hence so does $X_{t+\tau_n}$
for $t\leq1$.
It follows that with probability 1, if $X$ survives, then
$\lim_{t\to\infty}Z_t=\infty$. As the latter trivially implies
$\int_0^\infty Z_s \,ds=\infty$, the proof is complete.
\end{pf}

\begin{lemma}\label{lemmacomparisontheta}
Fix $\theta< {\theta}^{*}$ and $\gamma>0$. For any initial mass
distribution $\mu$ that satisfies
Assumption \ref{asmtpinireg2} and any $K\in C_p(\zz{R}^d,\zz{R}_+)$,
there exist on some probability space
epidemic processes $(X,X^{*})\in
D([0,\infty);\mathcal{M}_c(\zz{R}^d))^2$ with laws
$P^{\theta,1,\gamma}_{\mu,K}$ and $P^{{\theta}^{*},1,\gamma}_{\mu,K}$
and local time densities $L_{t},L^{*}_{t}$, respectively, such
that almost surely, for every $t\geq0$,
%
%
\begin{equation}
\label{eqdomination} L_{t}\leq L^{*}_{t}.
\end{equation}
\end{lemma}

\begin{pf}
This follows from the weak convergence result
\eqref{eqSIR-convergence} and the standard construction. Recall that
in the standard construction of the discrete \emph{SIR} epidemics, the
evolution is determined by the random graph
$\mathcal{G}^{\prime}
$
in which edges are present with
probabilities $p_{N} (x;y)=p^{\theta}_{N} (x;y)$ given by
\eqref{eqpN}. These probabilities are increasing in
$\theta$. Consequently, it is possible (using auxiliary $\operatorname{uniform}[0,1]$
random variables) to simultaneously construct random graphs
$\mathcal{G}^{\prime}$ and $\mathcal{G}^{\prime}_*$ with percolation
probabilities
$p^{\theta}_N (x;y)$ and $p^{\theta^{*}}_N(x;y)$, respectively, in
such a way
that the edge set of $\mathcal{G}^{\prime}$ is contained in that of
$\mathcal{G}^{\prime}_*$. This forces
\[
\zz{R}_{n} (x)\subseteq\zz{R}^{*}_{n} (x)\qquad
\mbox{for all } n\geq0 \mbox{ and } x\in\zz{Z}^{d},
\]
and hence also $R_{n} (x)\leq R^{*}_{n} (x)$. This inequality will be
preserved upon taking weak limits, so we obtain \eqref{eqdomination}.
\end{pf}

\begin{remark}\label{remarksurvivalByLT}
It follows immediately from \eqref{eqdomination} and Lemma \ref
{lemmasurvchar} that
%
%
\begin{equation}
\label{eqK-survival-probability-comparison} P (X \mbox{ survives}) \leq P
\bigl(X^{*} \mbox{ survives} \bigr).
\end{equation}
\end{remark}

\begin{lemma}\label{lemmacompK}
Let $K,K^{*}\in C_{p} (\zz{R}^{d};\zz{R}_{+})$ be suppression rate
functions such that $K \leq K^{*}$. Then for any $\mu$ satisfying Assumption
\ref{asmtpinireg2}, there exist spatial epidemics $(X,X^{*})\in
D([0,\infty);\mathcal{M}_c(\zz{R}^d))^2$ with marginal laws
$P^{\theta,1,\gamma}_{\mu,K}$ and $P^{{\theta},1,\gamma}_{\mu, K^*}$
and local time densities $L_{t},L^{*}_{t}$, respectively, such that
almost surely,
%
%
\begin{equation}
\label{eqcompK} L_{t} \geq L^{*}_{t} \qquad\mbox{for
all } t\geq0,
\end{equation}
and
$P (X \mbox{ survives}) \geq P (X^{*} \mbox{ survives})$.
\end{lemma}

\begin{pf}
The existence of the coupling follows directly from the weak convergence
\eqref{eqSIR-convergence} and the standard construction, because in
this construction, increasing the removed sets $\zz{K} $
decreases the sizes of the connected components. The assertion
about
survival probabilities follows from \eqref{eqcompK}, by Lemma
\ref{lemmasurvchar}.
\end{pf}

\begin{lemma}\label{lemmacompmu} Let $\mu_{0}, \nu_{0}$ be initial
mass distributions satisfying
Assumption~\ref{asmtpinireg2}, and $\mu=\mu_{0}+\nu_{0}$. Then on
some probability space there exist epidemic
processes $(X, X^1, X^2)\in
D([0,\infty);\mathcal{M}_c(\zz{R}^d))^3$, with initial conditions
$\mu,\mu_{0}$, and $\nu_{0}$ and marginal laws $P^{\theta,1,\gamma
}_{\mu,0}$, $P^{\theta,1,\gamma}_{\mu_{0},0}$ and $P^{\theta,1,\gamma}_{\nu_{0},0}$, respectively, such that
%
%
\begin{eqnarray}
\label{eqsubadditive} \max \bigl(L^1_t(x),
L^2_t(x) \bigr)\leq L_t(x) \leq
L^1_t(x) + L^2_t(x)
\nonumber
\\[-8pt]
\\[-8pt]
 \eqntext{\mbox{for
all } t\geq0\mbox{ and }x\in\zz{R}^d.}
\end{eqnarray}
Consequently,
%
%
\begin{eqnarray}
\label{eqsurvsubadditive} P(X \mbox{ survives}) &\leq& P \bigl(X^1
\mbox{ survives} \bigr) + P \bigl(X^2 \mbox{ survives} \bigr)\quad
\mbox{and}
\\
\label{eqsurvmono} P(X \mbox{ survives}) &\geq& P \bigl(X^i \mbox{
survives} \bigr)\qquad \mbox{for each } i=1,2.
\end{eqnarray}
\end{lemma}

\begin{pf}
This follows by the same argument as the preceding lemma; see
Remark \ref{remarkjointConvergenceSIR}.
\end{pf}

Lemma \ref{lemmacompmu} describes the effect of adding
infected mass at time $0$. The next lemma concerns the effect of
introducing additional infected mass at a time $t_{*}>0$ after the epidemic
has already begun. Let $\mu, \mu_{0}, \nu_{0}$ be initial mass
distributions satisfying the hypotheses of
Lemma \ref{lemmacompmu}. Say that $X^{*}$ is a
\emph{measure-valued epidemic with immigration} at time $t_{*}$ if it
satisfies the following martingale problem: for all $\varphi\in
C^{2}_{c} (\zz{R}^{d})$,
%
%
\begin{eqnarray}
\label{eqEpidemicWithImmigration}\quad  X^{*}_t(\phi) &=&
\mu_{0} (\varphi)+\mathbf{1}_{[t_{*},\infty)} (t)\nu_{0} (
\varphi)
\nonumber
\\[-8pt]
\\[-8pt]
\nonumber
&&{}+\int_0^t X^{*}_s
\bigl( \Delta\phi/2 +\theta\phi-K\phi-\beta L^{{*}}_s \phi
\bigr) \,ds+\sqrt{\gamma} M^{*}_t(\phi),
\end{eqnarray}
where $M^{*}_{t}$ is a continuous martingale with quadratic variation
$[M^{*} (\varphi)]_{t}=\int_0^t X^{*}_s(\phi^2) \,ds$, and
$L^{*}=L^{X^{*}}$ is the
local time density of $X^{*}$. Existence and uniqueness of solutions to
\eqref{eqEpidemicWithImmigration} follows from
Theorem \ref{thmmexistuniq} and the Markov property.

\begin{lemma}\label{lemmaEpidemicWithImmigration} Let $\mu_{0}, \nu
_{0}$ be initial mass distributions satisfying
Assumption~\ref{asmtpinireg2}, and $\mu=\mu_{0}+\nu_{0}$.
On some probability space there exists a solution $X$ to
\eqref{eqDWmgSIRgen} with initial mass distribution
$\mu$ and a solution $X^{*}_{t}$ to
\eqref{eqEpidemicWithImmigration} such that
%
%
\begin{equation}
\label{eqImmigrationComparisona} L^{X}_{t}\geq
L^{X^{*}}_{t}\qquad \forall t\geq0\quad \mbox{and}\quad L^{X}_{\infty}=
L^{X^{*}}_{\infty}.
\end{equation}
\end{lemma}

\begin{pf}
This is by discrete approximation, using the standard construction of
the discrete \emph{SIR} epidemics. On each percolation structure
$\mathcal{G}=\mathcal{G}^{N}$, we construct a pair of epidemics. The
first, denoted by $\zz{X}=\zz{X}^{N}$, is constructed using initially
infected sets $\zz{X}_{0}=\zz{X}^{N}_{0}$ such that
\eqref{hyp1dini}--\eqref{hypsmoothness} hold. The second,
denoted by $\zz{Y}=\zz{Y}^{N}$, has initially infected sets
$\zz{Y}_{0}\subseteq\zz{X}_{0}$ such that after
Feller--Watanabe rescaling the initial mass distributions converge to
$\mu_{0}$; see Remark \ref{remarkinitialApproximation} in
Section~\ref{secdiscreteSIR}. This second epidemic $\zz{Y}$ has
spontaneous new infections at time $[N^{\alpha}t_{*}]$: in
particular, all individuals in the sets\looseness=1
\[
\zz{X}_{0} \setminus\zz{Y}_{0}:=\bigcup_x
\bigl( \zz{X}_{0} (x) \setminus\zz{Y}_{0} (x) \bigr)
\]\looseness=0
who are not yet recovered become infected. Thus, the time evolution of
the epidemic $\zz{Y}_{n}$ is determined by the random graph
$
\mathcal{G}^{\prime}:=\mathcal{G}\setminus\zz{K}$ as
follows: (1) For $n<[N^{\alpha}t_{*}]$, the recovered set $\zz
{R}^{Y}_{n}$ consists
of all vertices at graph distance $< n$ from the initially infected set
$\zz{Y}_{0}$. (2) For $n\geq[N^{\alpha}t_{*}]$, the set
$\zz{R}^{Y}_{n}$ consists of all vertices $v$ such that either the graph
distance of $v$ from $\zz{Y}_{0} $ is $<n$, or the
graph distance of $v$ from $\zz{X}_{0} \setminus\zz{Y}_{0}$
is $< n- [N^{\alpha}t_{*}]$.

From the construction above and the standard construction described
earlier, it is clear that
%
%
\begin{equation}
\label{XYinclusion}\mbox{for all $n\geq0$}\qquad \zz{R}^X_n
\supseteq\zz{R}^Y_n\quad\mbox{and}\quad\zz{R}^X_n
\subseteq\zz{R}^Y_{n+[N^{\alpha}t_{*}]}.
\end{equation}
Set
\[
Y^{N}_{t} = \mathcal{F}_{N}|
\zz{Y}_{[N^{\alpha}t]}| \quad\mbox{and}\quad R^{ Y, N}_{t} =
\mathcal{S}_{N}\bigl|\zz{R}^Y_{[N^{\alpha}t]}\bigr|,
\]
where $\mathcal{F}_{N}$ and $\mathcal{S}_{N}$ are the Feller--Watanabe
and Sugitani rescaling operators.

\begin{claim}\label{claimdelay}
The vector-valued process $(Y^{N},R^{Y, N})$ converges weakly to
a process $(X^{*},L^{*})$ such that $X^{*}$ solves the
martingale problem \eqref{eqEpidemicWithImmigration}, and
$L^{*}=L^{X^{*}}$ is the local time density of $X^{*}$.
\end{claim}
By passing to a subsequence, if necessary, it follows from
Proposition \ref{propSIRconv}, the above claim and \eqref
{XYinclusion} that
\[
L_t^{X^*}\le L_t^Y\le
L^{X^*}_{t+t_*}\qquad\mbox{for all }t\ge0 \mbox{ a.s.}
\]
This clearly implies \eqref{eqImmigrationComparisona}.

\begin{pf*}{Proof of the claim (sketch)} This is done by following the likelihood
ratio strategy described in
Remark \ref{remarkjointConvergenceSIR}. As used in \citet{lalley09}
and \citet{lz10}, this strategy was based on the fact that each
discrete \emph{SIR} epidemic considered had law absolutely continuous
with respect to the law of a critical branching random walk with the
same initial condition. The bulk of the proof consisted of showing
that the likelihood ratios converged in law, \emph{under the branching
random walk measure}, to the Radon--Nikodym derivative of a
measure-valued epidemic relative to the law of super-Brownian
motion. For the processes considered in this claim, the appropriate
comparison processes are not standard branching random walks, but
rather \emph{branching random walks with immigration} in which new
particles are introduced at times $[N^{\alpha}t_{*}]$ in such a way
that after Feller--Watanabe rescaling the mass distributions of these
new particles converge to $\nu_{0}$. The laws of these processes
converge, after rescaling, to the law of super-Brownian motion with
immigration at time $t_{*}$, that is, a process $Y^{*}$ satisfying the
martingale problem~\eqref{eqEpidemicWithImmigration} with $\beta=0$
and $K=0$. (This follows easily from the standard convergence
theorem for critical branching random walks because the effect of
the immigration is simply to superimpose an independent branching random
walk shifted in time by $[N^{\alpha}t_{*}]$.)

Consider the likelihood ratios for the law of the epidemic process
$Y^{N}$ relative to that of the corresponding branching random
walk with immigration. These are products of factors indexed by
(discrete) times $t$ and lattice sites $x\in\zz{Z}^{d}$ [see
\citet{lalley09}, equation (53)]. For $t\leq[N^{\alpha}t_{*}]$ the factors
are exactly the same as in the case where there is no
immigration. Beginning with time $t=[N^{\alpha}t_{*}]$, new factors
are introduced; these indicate the relative likelihood ratios for the
newly introduced immigrants and their offspring. {Under the law of the
branching random walks with immigration} the immigrants and their
offspring evolve \emph{independently} of the progeny of the original
(time $0$) particles. Using this fact, one can show, in much the same
manner as in \citet{lalley09} and \citet{lz10}, that the likelihood
ratios converge weakly (under the branching random walk with
immigration laws) to the
Radon--Nikodym derivative of the process $X^{*}$ relative to
super-Brownian motion with immigration.
In carrying out this final step, the main hurdle is showing that in
the epidemics with immigration, the numbers of individuals in the sets
$\zz{X}_{0}\setminus\zz{Y}_{0}$ who are infected prior to time
$[N^{\alpha}t_{*}]$ is of order $O_{P} (1)$. Here is a brief synopsis
of the argument: since the local time densities, after rescaling,
converge,
the maximum number of recovered individuals
at time $[N^{\alpha}t_{*}]$
at any site is of order $O_{P}
(N^{\alpha(2-d/2)})$. Consequently, because $\zz{X}_{0}\setminus
\zz{Y}_{0}$ has cardinality on the order $N^{\alpha}$, the number of
individuals in $\zz{X}_{0}\setminus\zz{Y}_{0}$ infected prior to time
$[N^{\alpha}t_{*}]$ is of order
%
%
\begin{equation}
\label{purplebnd} O \bigl(N^{\alpha} \bigr)\times\frac
{O_P(N^{\alpha
(2-d/2)})}{N} =
O_P(1),
\end{equation}
since $\alpha=2/(6-d)$.
\end{pf*}
This completes the proof of Lemma \ref{lemmaEpidemicWithImmigration}.
\end{pf}
In the proof of Theorem \ref{thmmmain} it will be necessary to compare
the evolution of a measure-valued epidemic $X$ with a coupled process
in which additional infected mass is introduced at a \emph{random} time.
Say that $X^{*}_{t}$ is a \emph{measure-valued epidemic with
immigration} at time $\tau$ if it satisfies the following martingale
problem: for all $\varphi\in
C^{2}_{c} (\zz{R}^{d})$,
%
%
\begin{eqnarray}
\label{eqtauEpidemicWithImmigration}\quad  X^{*}_t(\phi)& =&
\mu_{0} (\varphi)+\mathbf{1}_{[\tau,\infty)} (t)\nu_{0} (
\varphi)
\nonumber
\\[-8pt]
\\[-8pt]
\nonumber
&&{}+\int_0^t X^{*}_s
\bigl( \Delta\phi/2 +\theta\phi-K\phi-\beta L^{*}_s \phi
\bigr) \,ds+\sqrt{\gamma} M^*_t(\phi),
\end{eqnarray}
where $\tau$ is a finite stopping time relative to the
filtration $\sF^{X^*}$, $M^{*}_{t}$ is a continuous martingale
with quadratic variation $[M^{*} (\varphi)]_{t}=\int_0^t
X^{*}_s(\phi^2) \,ds$ and $L^*=L^{X^{*}}$ is the local time density.

\begin{lemma}\label{lemmatauImmigration}
Let $\mu_{0}, \nu_{0}$ be initial mass distributions satisfying
Assumption~\ref{asmtpinireg2}, and $\mu=\mu_{0}+\nu_{0}$. Then on
some probability space
there exist epidemics $(X_t, X_t^*) \in
D([0,\infty);\mathcal{M}_c(\zz{R}^d))^2$ such that (1) $X$ solves
the martingale problem~\eqref{eqDWmgSIRgen} with initial value
$X_{0}=\mu$; (2) $X^{*}$ solves the martingale problem~\eqref{eqtauEpidemicWithImmigration}; (3)
%
%
\begin{equation}
\label{eqImmigrationComparison} L^{X}_{t}\geq
L^{X^{*}}_{t}\qquad \mbox{for all } t\geq0 \quad\mbox{and}\quad
L^{X}_{\infty}= L^{X^{*}}_{\infty}.
\end{equation}
\end{lemma}

\begin{pf}
By the usual continuity (weak convergence) arguments, it suffices to
prove this for stopping times $\tau$ that take values in a finite
set. By a routine induction on the cardinality of this finite set, it
suffices to consider stopping times that take values in a two-element
set $\{0,t_{*} \}$. For such stopping times, the result follows from
Lemma \ref{lemmaEpidemicWithImmigration}, since this can be applied
conditionally on $\mathcal{F}_{0}$.
\end{pf}

Now the results of Lemmas \ref{lemmacompK}, \ref{lemmacompmu} and
\ref{lemmatauImmigration} can be combined, allowing us to couple the
measure-valued epidemic $X$ with a measure-valued process
$X^{*}$ in which the infected mass is decreased and the
suppression rate increased at a random time $\tau$. Here we will use
the strong Markov
property [Theorem \ref{thmmexistuniq}(c)] and Remark \ref
{remarkrandomInitialConditions}. The process
$X^{*}$ will satisfy the following martingale problem: for every
$\varphi\in C^{2}_{c} (\zz{R}^{d})$,
%
%
\begin{eqnarray}
\label{eqrubeGoldbergConstruction} &&X^{*}_{t} (\varphi)=
\nonumber
\\[-4pt]
\\[-12pt]
\nonumber
&&\qquad\cases{\displaystyle \mu_0(\phi)+\int_0^t
X^{*}_s \bigl( \Delta\phi/2 +\theta\phi-K\phi-\beta
L_s^{X^{*}}\phi \bigr) \,ds\vspace*{2pt}\cr
\qquad{}+\sqrt{\gamma}
M_t^{*}( \phi), & \quad$\mbox{for all } t< \tau$, \vspace*{2pt}
\cr
\displaystyle Y_{\tau} (\varphi) + \int_{\tau}^t
X^{*}_s \bigl( \Delta\phi/2 +\theta\phi-K^{*}_{\tau}
\phi-\beta L_s^{X^{*}}\phi \bigr) \,ds\vspace*{2pt}\cr
\qquad{}+\sqrt{\gamma}
\bigl(M^{*}_t(\phi)-M^{*}_{\tau}(\phi)
\bigr),&\quad $\mbox{for all } t\geq\tau,$}\hspace*{-35pt}
\end{eqnarray}
where $L^{X^{*}}_{s}$ is the local time density and:
\begin{longlist}[(iii)]
\item[(i)]$\tau$ is a finite $\sF^{X^*}$-stopping time;
\item[(ii)]$Y_{\tau}$ is an $\sF^{X^*}_{\tau-}$-measurable random
measure satisfying $Y_{\tau}\leq X^{*}_{\tau-}+\nu_{0}$;
\item[(iii)]$K^{*}_{\tau}$ is an $\sF^{X^*}_{\tau}$-measurable random
element of $C_{p} (\zz{R}^{d},\zz{R}_{+})$
satisfying \mbox{$K^{*}_{\tau}\geq K$};
\item[(iv)]$M^{*}_{t} (\varphi)$ is an $\sF^{X^*}$-continuous martingale with
quadratic variation $[M^{*} (\varphi)]_{t}=\int_0^t X^{*}_s(\phi^2) \,ds$.
\end{longlist}
%

\begin{prop}\label{propcompimmigration}
Let $\mu_{0}, \nu_{0}$ be initial mass distributions satisfying
Assumption \ref{asmtpinireg2}, and $\mu=\mu_{0}+\nu_{0}$.
Then there exist epidemics $(X_t, X_t^*) \in
D([0,\infty); \mathcal{M}_c(\zz{R}^d))^2$ such that:
\begin{longlist}[(iii)]
\item[(i)]$X$ solves the martingale problem \eqref{eqDWmgSIRgen} with
initial value $X_{0}=\mu$;
\item[(ii)]$X^{*}$ satisfies the martingale
problem \eqref{eqrubeGoldbergConstruction};
\item[(iii)] the local time densities of $X$ and $X^{*}$ satisfy
\[
L^{X}_{t}\geq L^{X^{*}}_{t} \qquad\mbox{for
all } t\geq0.
\]
\end{longlist}
\end{prop}

This proposition can be used iteratively, using the strong Markov
property and standard arguments, so as to allow immigration
and increases in the suppression rate $K^{*}$ at increasing stopping times
$0=\tau_{0}\leq\tau_{1}\leq\tau_{2}\leq\cdots<\infty$. The
associated martingale problem is as follows: for every $\varphi\in
C^{2}_{c} (\zz{R}^{d})$, and for all $\tau_i\leq t< \tau_{i+1}$,
%
%
\begin{eqnarray}
\label{eqDWSIRImmiSeries}\qquad X_t^*(\varphi) &= &\mu_{i}(
\varphi)
+ \int_{\tau_i}^t X^{*}_s
\bigl( \Delta\varphi/2 +\theta\varphi-K^{*}_{i}\phi-\beta
\bigl(L_s \bigl(X^{*} \bigr)-L_{\tau_i}
\bigl(X^{*} \bigr) \bigr)\varphi \bigr) \,ds
\nonumber
\\[-4pt]
\\[-12pt]
\nonumber
&&{}+\sqrt{\gamma}
\bigl(M^{*}_t(\varphi)-M^{*}_{\tau_i}(
\varphi) \bigr),
\end{eqnarray}
where:
\begin{longlist}[(iii)]
\item[(i)] for each $i=1,2,\ldots,$ $\mu_{i}$ and $\nu_{i}$ are
$\sF^{X^*}_{\tau_{i}-}$-measurable random measures
satisfying Assumption \ref{asmtpinireg2} and
such that
\[
\mu_{i} + \nu_{i} \leq X^{*}_{\tau_i-}+
\nu_{i-1};
\]
\item[(ii)]$K^{*}_{0}\equiv0$, and $K^{*}_{i}\in C_{p}
(\zz{R}^{d},\zz{R}_{+})$ is, for each $i$, an
$\sF^{X^*}_{\tau_i}$-measurable random function such that
\[
K^{*}_{i}\geq K^{*}_{i-1} + \beta
\bigl(L_{\tau_i} \bigl(X^{*} \bigr)-L_{\tau_{i-1}}
\bigl(X^{*} \bigr) \bigr);
\]
\item[(iii)]$M^{*}_{t}$ is a continuous $\sF^{X^*}$-martingale with
quadratic variation\break  $[M^{*} (\varphi)]_{t}=\int_0^t X^{*}_s(\phi^2) \,ds$.
\end{longlist}
%
The existence of a solution to this martingale problem follows from
the strong Markov property [Theorem \ref{thmmexistuniq}(d)].
Proposition \ref{propcompimmigration} and a standard induction
argument now yield the following comparison result.

\begin{prop}\label{propcompimmigrationseries}
Let $\mu_{0}, \nu_{0}$ be initial mass distributions satisfying
Assumption \ref{asmtpinireg2}, and $\mu=\mu_{0}+\nu_{0}$.
Then on some probability space there exist measure-valued processes
$X$ and $X^{*}$ such that \textup{(i)}
$X$ solves the martingale problem $(\mathrm{MP})_{\mu,0}^{\theta,\beta,\gamma}$
specified in \eqref{eqDWmgSIRgen}; \textup{(ii)} $X^{*}$ solves the
martingale problem \eqref{eqDWSIRImmiSeries}; \textup{(iii)} the
corresponding local time processes satisfy
\[
L^{X}_{t}\geq L^{X^{*}}_{t} \qquad\mbox{for
all } t\geq0.
\]
\end{prop}

\subsection{The sandwich lemma}
The discrete \emph{SIR}-$d$ process $X^N$ is naturally associated
with a \emph{branching envelope}. This is a nearest-neighbor branching
random walk $\overline{X}_n^N$ with initial condition
$\overline{X}^{N}_{0}=X^{N}_{0}$ and offspring distribution
Bin$((2d+1)N,p_N^\theta(0,0))$ that dominates $X^{N}_{n}$, that is,
such that for each $n\geq0$ and $x\in\zz{Z}^{d}$,
\[
X^{N}_{n} (x)\leq\overline{X}^{N}_{n}
(x).
\]
See Section~1.6 of \citet{lalley09} for details concerning the
construction. Since branching random walks, after Feller--Watanabe
rescaling, converge weakly to super-Brownian motions, the
vector-valued processes $(X^{N}, \overline{X}^{N})$, similarly
rescaled, have marginals that converge weakly. It follows that after
rescaling the laws of the vector-valued processes $( X^{N},
\overline{X}^{N})$ are tight. Hence, any subsequence has a weakly
convergent subsequence, and the limit process $(X, \overline{X})$ must
satisfy $X_{t}\leq\overline{X}_{t}$. The component processes $X$
and $\ol{X}$ of any such weak limit must be a measure-valued
epidemic [i.e., a solution of the martingale problem
\eqref{eqDWmgSIRgen} with $\gamma=1$] and a super-Brownian motion
with drift $\theta$, respectively. This gives another proof of
Proposition \ref{propBRWenvelope}.
Next is a result which also gives a \emph{lower bound} on the epidemic process.

\begin{lemma}\label{lemmasandwich0}
For any measure $\mu\in\mathcal{M}_{c} (\zz{R}^{d})$ satisfying
Assumption \ref{asmtpinireg2}, any $\kappa>0$,
$\theta\in\zz{R}$, and any
function $K\in C_{p} (\zz{R}^{d},\zz{R_{+}})$ there exist,\vspace*{1pt}
on some probability space, measure-valued processes $X, \ol
{X},\ul{X}$ with common initial state $X_{0}=\ol{X}=\ul{X}=\mu$
such that
\[
\underline{X}_t \le X_t\le\overline{X}_t\qquad
\mbox{for all } t\leq\tau,
\]
where
\[
\tau=\inf \Bigl\{t\geq0\dvtx \Bigl(\max_x K(x) \Bigr) +
\Bigl( \max_x L^{X}_t(x) \Bigr) \geq
\kappa \Bigr\},
\]
with the following laws:
\begin{longlist}[(iii)]
\item[(i)]$X$ is a spatial epidemic with local suppression rate $K$,
transmission rate~$\theta$, branching rate $\gamma=1$ and inhibition parameter $\beta
=1$;
\item[(ii)]$\ol{X}$ is a super-Brownian motion with drift $\theta$;
\item[(iii)]$\ul{X}$ is a super-Brownian motion with drift
$\theta-\kappa$.
\end{longlist}
%
\end{lemma}

The proof will once again be based on discrete approximations. We
shall build approximating discrete epidemic processes that satisfy the
analogous sandwich relationship. The construction makes use of the
following lemma. First, observe that in a discrete SIR epidemic, when
two
infected individuals simultaneously attempt to infect
the same susceptible individual, all but one of the attempts fail;
call such an occurrence a \emph{collision}. [Hence, e.g., when three
infected individuals simultaneously attempt to infect
the same susceptible individual, then the number of collisions would be
${3\choose2} = 3$.]

By slightly modifying the proof of Lemma 9 in \citet{lz10}, in
particular, by noticing that the statement right above equation (62)
therein also holds for the way that we count the number of collisions
here, we get the following:
%
\begin{lemma}[{[A slight variant of Lemma 9 and equations (62)--(64) in
\citet{lz10}]}]\label{lemmacollisions}
For each pair $(n,x)\in\zz{N}\times\zz{Z}^{d}$, let
$\Gamma^{N}_{n} (x)$ be the number of collisions at site $x$ and
time $n$ in the \emph{SIR}-$d$ epidemic with village size $N$.
Assume that the hypotheses
\eqref{hypinitialConfig}--\eqref{hypsmoothness} of Proposition
\ref{propSIRconv} are satisfied.
Then for any fixed $T\geq0$,
%
%
\begin{equation}
\label{eqcollisionEstimate} E \sum_{n\leq N^\alpha T}\sum
_{x} \Gamma^{N}_{n} (x) =o
\bigl(N^{\alpha} \bigr).
\end{equation}
\end{lemma}

A direct consequence of the previous lemma is that if we
define a \textit{Modified SIR} process in the following way:

\textit{Modified SIR process}. At any site/time $(x,t)$,
each particle produces\break  Bin$((N-K^N(y)-R_t^N(y)),
p_N^\theta(x,y))$ number of offspring at neighboring sites $y$, where
$R_k^N(y)=\sum_{i<k}X^N_i(y)$,
then the Modified SIR process can be constructed together
with the original SIR process, in much the same way as for the
branching envelope with the original SIR process, such that: (1)~the
Modified SIR process always dominates the original SIR process; and (2)~the
discrepancy $D_{t} (x)\ge0$ between them at site $x$ and time $t$
satisfies that for any $T>0$,
%
%
\begin{equation}
\label{eqdiscrepancy} \max_{t\leq N^\alpha T}\sum
_x D_{t} (x)=o_{P}
\bigl(N^{\alpha} \bigr).
\end{equation}
Therefore after the Feller--Watanabe scaling as in Proposition
\ref{propSIRconv}, the modified SIR process will
converge to the same limit as in Proposition \ref{propSIRconv}.

[To show \eqref{eqdiscrepancy}, observe that if we let $D_n:=\sum_x
D_{n} (x)$, and $\wt{D}_n = D_n/(1+\theta/N^\alpha)^n$, then $\wt{D}_n$
is a sub-martingale: $E(\wt{D}_{n+1}|\sF_n) \geq\wt{D}_n$, with the
inequality due to collisions at generation $n+1$. Further note that for
any $T>0$,
\begin{eqnarray*}
E(\wt{D}_{[N^\alpha T]})&\leq&\frac{E \sum_{n\leq N^\alpha T}
((1+\theta/N^\alpha)^{[N^\alpha T]-n}\cdot\sum_{x} \Gamma
^{N}_{n} (x) )}{(1+\theta/N^\alpha)^{[N^\alpha T]}}
\\
&=&O \biggl(E \biggl(\sum_{n\leq N^\alpha T}\sum
_{x} \Gamma^{N}_{n} (x) \biggr)
\biggr)= o \bigl(N^{\alpha} \bigr).
\end{eqnarray*}
Equation~\eqref{eqdiscrepancy} then follows from the Doob's martingale
inequality.]

We now prove Lemma \ref{lemmasandwich0}.
\begin{pf*}{Proof of Lemma \ref{lemmasandwich0}}
We shall build approximating particle systems that satisfy the
analogous sandwich relationship. Choose $X_0^N$ and $K^N$ such
that \eqref{hypinitialConfig}--\eqref{hypK} are satisfied.
The super-solution $\ol{X}^N$ is a nearest-neighbor branching random
walk with
initial configuration $X_0^N$ and such that at any site/time
$(x,t)$, each particle at site $x$ produces Bin$(N,
p_N^\theta(x,y))$ number of offspring at neighboring sites $y$. By
Watanabe's theorem, $\ol{X}^N$ converges to the desired
$\ol{X}$. As noted above, the {Modified SIR} ${X}^N$ will
approximate ${X}$.
Define the stopping time
\[
\tau^N=\min \Bigl\{t\geq0\dvtx \Bigl(\max_xK^N(x)
\Bigr) + \Bigl( \max_xR^{N}_t(x)
\Bigr) \geq\kappa{N^{(\alpha(2-d/2))}} \Bigr\}.
\]
We may assume that $\kappa>\sup_xK(x)$ (or the result is trivial).
The sub-solution $\ul{X}^N$ is a nearest-neighbor branching random walk
with initial configuration $X_0^N$ and
such that at any site/time $(x,t)$, each particle at $x$ produces
Bin$([N- \kappa N^{\alpha(2-d/2)}], p_N^\theta(x,y))$ number of
offspring at neighboring sites $y$. By Watanabe's theorem $\ul{X}^N$
converges weakly to the super-Brownian motion $\ul{X}$. It is clear
that before time $\tau^N$, the three processes $\ol{X}^N$, ${X}^N$ and
$\ul{X}^N$ can be built on a common probability space such that
\[
\ul{X}^N_t \le{X}^N_t\le
\ol{X}^N_t \qquad\mbox{for all } t\leq\tau^N.
\]
By Skorokhod's representation theorem and Proposition \ref
{propSIRconv} we
may assume $\liminf_N\tau^N\ge\tau$ a.s. Here we use $\max_xK^N(x)/N^{(\alpha(2-d/2))}\rightarrow\max_x K(x)$ and the fact that
there is a greater than or equal to sign in the definition of $\tau$.
By taking limits in the above, along a subsequence if necessary to get
joint convergence, we complete the proof.
\end{pf*}

\subsection{Scaling}
It will be necessary, in some of the arguments to follow, to rescale
time and/or space. When a super-Brownian motion, or more generally a
solution to the martingale problem \eqref{eqDWmgSIRgen} is
rescaled, its diffusion rate may change, that is, the Laplacian in
\eqref{eqDWmgSIRgen} may be multiplied by a constant $\alpha
$. The resulting martingale problem is as follows:
%
%
\begin{eqnarray}
\label{eqsirpar} X_t(\phi)&=& X_0(\phi) +
\frac{\alpha}{2}\int_0^t X_s(
\Delta\phi) \,ds + \theta\int_0^t
X_s(\phi) \,ds
\nonumber
\\[-8pt]
\\[-8pt]
\nonumber
&&{}- \int_0^t
\bigl(X_s(K\phi)+\beta X_s (L_s\phi)
\bigr) \,ds + \sqrt{\gamma} M_t(\phi),
\end{eqnarray}
where $\alpha,\beta,\gamma>0$ and $\theta\in\zz{R}$ are
constants, $K\in C_p(\zz{R}^d, \zz{R}^+)$, and $M_t(\phi)$ is a
continuous martingale with
quadratic variation $[M(\phi)]_t=\int_0^t X_s(\phi^2) \,ds$. As usual,
$L_{t}$ is the local time density of the process $X$. We shall
refer to this martingale problem as $(\mathrm{MP})^{\theta,\beta,\gamma,\alpha}_{\mu,K}$ and continue to write $(\mathrm{MP})^{\theta,\beta,\gamma
}_{\mu,K}$ if
$\alpha=1$.

\begin{lemma}\label{lemmascaling}
Let $X$ solve the martingale problem $(\mathrm{MP})^{\theta,\beta,\gamma
,\alpha
}_{\mu,K}$. For any constants $a,b,c> 0$, define a new
measure-valued process $U$ by
%
%
\begin{equation}
\label{eqfirstRescaling} \qquad U_t(\psi) = c \int_x
\psi(bx) X_{at}(dx)\qquad \mbox{for all bounded measurable } \psi\mbox{ on }
\zz{R}^d.
\end{equation}
Then $U_{t}$ solves the martingale problem $(\mathrm{MP})^{\theta',\beta
',\gamma
',\alpha'}_{\mu',K'}$ with parameters
\begin{eqnarray*}
\theta'&=&a\theta, \qquad\beta'=\frac{a^{2}b^{d}\beta}{c},\qquad
\gamma'=ac\gamma, \qquad\alpha'=ab^{2}\alpha,\\
K'(x)&=&aK(x/b),
\end{eqnarray*}
and initial measure defined by
$
\int\psi(x) \mu'(dx)= c\int\psi(bx) \mu(dx).
$
The local time densities $L=L^{X}$ and $L^{U}$ are related by
%
%
\begin{equation}
\label{eqltscaling} L^U_s(x) = \frac{c}{ab^d}L_{as}
\biggl(\frac{x}{b} \biggr)\qquad \mbox{ for all } x, t.
\end{equation}
\end{lemma}
\begin{pf}
This is by routine calculations.
\end{pf}

\begin{remark}
Based on the above result, one can show that by choosing $a, b$ and $c$
appropriately, the scaling as in \eqref{eqfirstRescaling} would
transform the martingale problem $(\mathrm{MP})^{\theta,\beta,\gamma,\alpha
}_{\mu,0}$ into $(\mathrm{MP})^{\theta',1,1,1}_{\mu',0}$; in other words,
the model is a
one-parameter model.
\end{remark}

\section{Preliminaries on (supercritical) super-Brownian motions}
\label{secpreliminariessbm}

In this section we present some regularity results for super-Brownian
motions and
their local times. The results are only of interest, and only will be
used for $d>1$, and
so we assume $d=2$ or $3$ throughout this section.

\subsection{Uniform regularity of super-Brownian motions}

In this and the following subsection, let $Y=Y^\mu$ be a (driftless)
super-Brownian
motion with initial state $Y_{0}=\mu$, and let $P_\mu= P_{\mu
,0}^{0,0,1}$ be its law. Denote by
$B_r(y)$ the open Euclidean ball in $\zz{R}^d$ centered at $y$
of radius $r$, and for any measure $\mu\in\mathcal{M}(\zz{R}^d)$,
define
%
%
\begin{equation}
\label{dfnD} D(\mu,r)=\sup \bigl\{\mu \bigl(B_r(y) \bigr)\dvtx y\in
\zz{R}^d \bigr\}.
\end{equation}
For any function $\phi$ and any measure $\mu\in\mathcal{M}(\zz
{R}^{d})$, set $\mu\phi= \mu*\phi$, where $*$ denotes convolution. In
particular,
for any $t\geq0$,
\[
(\mu p_t) (x)=\int p_{t}(x-y)\mu(dy)\quad \mbox{and}\quad (\mu
q_t) (x)=\int_0^t\int
p_{s}(x-y)\mu(dy) \,ds,
\]
where $p_{t}$ and $q_t$ are the Gauss kernel and the integrated Gauss kernel
in \eqref{eqheatKernel}, respectively.
For $r\in(0,1]$ let
\[
h(r) = \sqrt{r \log\frac{1}{r}}\quad \mbox{and}\quad \phi(r)=r^2
\biggl(1+\log\frac{1}{r} \biggr)^2.
\]
Finally, for $T,r_0,C_i>0$ introduce the event
\[
G_T(r_0; C_1, C_2)= \bigl\{
D(Y_t,r) \leq C_1 \bigl(D(\mu p_t,
C_2 r) + \phi(r) \bigr) \mbox{ for all } r\leq r_0\mbox{
and } t\leq T \bigr\}.
\]

The following lemma is an easy consequence of the proof of Theorem
4.7 in \citet{BEP91}.

\begin{lemma}\label{lemmasbmreg}
If $K\ge1$ there are constants $C_1,C_2>0$ (depending on $K$),
and for any $T>0$ there is an $r_0(K,T)\in(0,1]$ such that for all
$\lambda\ge1$ and any $\mu$ with $|\mu|=\lambda$,
%
%
\begin{equation}
\label{eqsbmreg} P_\mu \bigl( G_T \bigl(r_0
e^{-\lambda}; C_1, C_2 \bigr) \bigr)
\geq1-e^{-K\lambda}.
\end{equation}
\end{lemma}

\begin{pf}
This is a quantitative version of Theorem 4.7 of \citet{BEP91}. The
proof of that result shows for $K\ge1$ there are constants
$C_1,C_2,C_3\ge1$ such that for all $\lambda\ge1$, $T>0$,
$n\in\zz{N}$ and $\mu$ with $|\mu|=\lambda$,
%
%
\begin{eqnarray}
\label{BEPbnd} P_\mu \bigl(G_T \bigl(h
\bigl(2^{-n} \bigr);C_1,C_2
\bigr)^c \bigr)&\le& C_3(T+1) (\lambda+1)2^{-Kn}+C_3
\lambda2^{-Kn}
\nonumber
\\[-8pt]
\\[-8pt]
\nonumber
& \le& C_3(2T+3)\lambda2^{-Kn}.
\end{eqnarray}
Here one has to chase constants a bit to check that the constant
$c_{2.2}$ in the proof of the above theorem in \citet{BEP91} may be
taken to be as large as you like at the cost of our $C_2$ and
their $c_{4.2}$ being large. The latter can then be handled in the
key bound in the proof of Theorem 4.7 in \citet{BEP91} by taking
our $C_1$ large enough. Now choose $n_0\ge2$ in $\zz{N}$ so that
%
%
\begin{equation}
\label{n0defn} C_3(2T+3)2^{-Kn_0}\le e^{-K\lambda}
\lambda^{-1}<C_3(2T+3)2^{-Kn_0+K}.
\end{equation}
The above definition implies
\begin{eqnarray*}
h \bigl(2^{-n_0} \bigr) &=&2^{-n_0/2}(n_0
\log2)^{1/2} \ge \biggl(\frac{e^{-K\lambda}\lambda^{-1}}{2^KC_3(2T+3)} \biggr)^{1/(2K)} [2\log2
]^{1/2}\\
 &\ge& r_0(K,T)e^{-\lambda},
\end{eqnarray*}
where in the last inequality we used the simple fact that for all
$\lambda, K\geq1$, $\lambda^{-1/(2K)}\geq\lambda^{-1}\geq
\exp(-\lambda/2)/2$.
Therefore \eqref{BEPbnd} and \eqref{n0defn} imply that
\[
P_\mu \bigl(G_T \bigl(r_0e^{-\lambda};C_1,C_2
\bigr)^c \bigr)\le e^{-K\lambda}.
\]
\upqed\end{pf}

To formulate the next result we introduce the following:
%
\begin{dfn}\label{dfnmassreg}
For any positive constants $A, \lambda$ and $r_0$, with $r_0\le1$, and
any measure $\mu\in\mathcal{M}(\zz{R}^d)$, we say that $\mu$ is
$(A,\lambda,r_0)$-\emph{admissible} if
\[
D(\mu,r) \leq A r^2 \biggl(\lambda r^{d-2}+ \biggl(1+\log
\frac{1}{r} \biggr)^2 \biggr)\equiv\psi(r) \qquad\mbox{for all } r
\leq r_0 e^{-\lambda}.
\]
\end{dfn}

\begin{cor}\label{corsbmreg}
For any fixed $K\ge1$ and $T>0$, there exist positive constants
$A=A(K,T)$ and $r_0=r_0(K,T)\le1$, such that for all $\lambda\ge1$
and $\mu$ with $|\mu|=\lambda$,
%
%
\begin{equation}
\label{eqsbmreg2} P_\mu \bigl( Y_T \mbox{ is } (A,
\lambda,r_0)\mbox{-admissible} \bigr)\geq1-e^{-K\lambda}.
\end{equation}
\end{cor}
\begin{pf}This follows from Lemma \ref{lemmasbmreg} by noticing that
\[
D(\mu p_T, C_2 r)\leq C_4(T)\lambda
r^{d}.
\]
\upqed\end{pf}

\subsection{Local time densities of super-Brownian motions}
Recall [equation \eqref{eqheatKernel}] that $p_{t} (x)$ and $q_{t}
(x)$ are the Gauss kernel and the integrated Gauss kernel, respectively.

\begin{lemma}\label{lemmaKt}Suppose that $\mu\in\mathcal{M}(\zz{R}^d)$
satisfies $|\mu|=\lambda$ and
is $(A,\lambda,r_0)$-admissible for some constants $A$ and $r_0$.
For any $0<\beta<2-d/2$ and any fixed $T>0$, define
%
%
\begin{eqnarray}
\label{dfnXii} \Xi_1(T)&:=&\max_x (\mu
q_{T} ) (x)\quad \mbox{and}
\nonumber
\\[-8pt]
\\[-8pt]
\nonumber
 \Xi_2(T)&:=&\max
_x \int_0^T \int
s^{-\beta/2} p_{s}(x-y) \mu(dy) \,ds.
\end{eqnarray}
Then there exists a constant $A'=A'(A, r_0, T,\beta)>0$ such that for
both $i=1,2$ and
for all $\lambda\ge1$,
%
%
\begin{equation}
\label{eqltexpmax} \Xi_i(T)\leq \kappa_{d}(\lambda):=
\cases{ 
A' \lambda^2, &\quad $
\mbox{when } d=2,$\vspace*{2pt}
\cr
A' \lambda^2e^{\lambda},
&\quad $\mbox{when } d=3.$}
\end{equation}
\end{lemma}
\begin{pf}
We shall only prove the result for $\Xi_1(T)$; the proof for
$\Xi_2(T)$ is similar. Let $\widetilde{r}_0= r_0e^{-\lambda}$.
We first deal with the integral for $t\in[\widetilde{r}_0^{8/3},T]$:
to do so, for any fixed $x\in\zz{R}^d$ we cover $\zz{R}^d$ with
balls $B_i$ of radius $\widetilde{r}_0$ and with center of distance
$d_i=k \widetilde{r}_0$ to $x$ for some $k\in\zz{Z}_{\geq0}$.
Then for any $t>0$,
\begin{eqnarray*}
(2\pi)^{d/2}(\mu p_{t}) (x)&=&\int t^{-d/2} \exp
\biggl(-\frac
{|x-y|^2}{2t} \biggr)\mu(dy)
\\
&\leq&\sum_{i} \int_{y\in B_i}
t^{-d/2} \exp \biggl(-\frac{\min
(0,d_i-\widetilde{r}_0)^2}{2t} \biggr) \mu(dy).
\end{eqnarray*}
The balls can be chosen in such a way that for any $k\geq3$ there
are at most $C(k-2)^{d-1}$ balls with center of distance
$k\widetilde{r}_0$ to $x$. It is then easy to see that there exist
constants $C_i$ such that
\begin{eqnarray*}
&&\int t^{-d/2} \exp \biggl(-\frac{|x-y|^2}{2t} \biggr) \mu(dy)
\\
&&\qquad\leq C_1 t^{-d/2} \psi(\widetilde{r}_0) +
C_2\sum_{k=3}^{\infty}
t^{-d/2} \exp \biggl(-\frac{(k-1)^2\widetilde{r}_0^2}{2t} \biggr) \cdot
(k-2)^{d-1} \psi(\widetilde{r}_0)
\\
&&\qquad\leq C_1 t^{-d/2} \psi(\widetilde{r}_0) +
C_2\psi(\widetilde{r}_0)\int_1^{
\infty} t^{-d/2}a^{d-1} \exp \biggl(-\frac{a^2 \widetilde{r}_0^2}{2t} \biggr)
\,da
\\
&&\qquad\leq C_1 \psi(\widetilde{r}_0)t^{-d/2} +
C_3 \psi(\widetilde{r}_0)\widetilde{r}_0^{-d}
\\
&&\qquad\leq\cases{ C_1 \psi(\widetilde{r}_0)t^{-1}
+ C_3 \bigl(\lambda+ \bigl(1+\log(1/\widetilde{r}_0)
\bigr)^2 \bigr), &\quad$\mbox{when } d=2,$\vspace*{2pt}
\cr
C_1
\psi(\widetilde{r}_0)t^{-3/2} +C_3 \bigl(
\lambda+ \bigl(1+\log(1/ \widetilde{r}_0) \bigr)^2/
\widetilde{r}_0 \bigr), &\quad$\mbox{when } d=3.$}
\end{eqnarray*}
Therefore
\begin{eqnarray*}
&&\int_{\widetilde{r}_0^{8/3}}^T \int t^{-d/2} \exp
\biggl(-\frac
{|x-y|^2}{2t} \biggr) \mu(dy) \,dt
\\
&&\quad\leq\cases{ 
C_4 \psi(\widetilde{r}_0)
\bigl(\log T +\log(1/ \widetilde{r}_0) \bigr) + C_3 T
\bigl(\lambda+ \bigl(1+\log(1/\widetilde{r}_0) \bigr)^2
\bigr), \vspace*{2pt}\cr
\hspace*{242pt}\mbox{when } d=2,\vspace*{2pt}
\cr
C_4 \psi(
\widetilde{r}_0)\widetilde{r}_0^{ (-4/3)}
+C_3 T \bigl(\lambda+ \bigl(1+\log(1/\widetilde{r}_0)
\bigr)^2/\widetilde{r}_0 \bigr), \qquad \mbox{when } d=3,}
\end{eqnarray*}
which can be bounded by $\kappa_d(\lambda)$ for all $\lambda\ge1$
for an appropriate choice of $A'$. Now we deal with the
integral for $t\in[0,\widetilde{r}_0^{8/3}]$,
\begin{eqnarray*}
&&\int_0^{\widetilde{r}_0^{8/3}}\int t^{-d/2} \exp
\biggl(-\frac
{|x-y|^2}{2t} \biggr) \mu(dy) \,dt
\\
&&\qquad=\int_0^{\widetilde{r}_0^{8/3}} \biggl(\int_{|x-y|\leq t^{3/8}}
+\int_{|x-y|> t^{3/8}} \biggr) t^{-d/2} \exp \biggl(-
\frac{|x-y|^2}{2t} \biggr) \mu(dy) \,dt
\\
&&\qquad\leq \int_0^{\widetilde{r}_0^{8/3}}t^{-d/2}\psi
\bigl(t^{3/8} \bigr) \,dt + \lambda\int_0^{\widetilde{r}_0^{8/3}}
t^{-d/2} \exp \biggl(-\frac{1}{2t^{1/4}} \biggr) \,dt
\\
&&\qquad\leq C_5+ C_6\lambda.
\end{eqnarray*}
\upqed\end{pf}

The following lemma is implicit in \citet{Sugitani89}.
%
\begin{lemma}\label{lemmaltreg}
Suppose that $\mu\in\mathcal{M}(\zz{R}^d)$ and for all $t>0$,
$\Xi_1(t)<\infty$,
and for some $0<\beta< 2-d/2$ and all $t>0$,
$\Xi_2(t)<\infty$, where $\Xi_1$ and $\Xi_2$ are defined in~\eqref
{dfnXii}.
Define
\[
Z_t(x)=L_t(x) - (\mu q_t) (x).
\]
Then for any $T>0$, there exist
constants $\eta_0 = \eta_0(\beta,T)>0, C_i=C_i(\beta,T)>0$ such that
for all $0\leq\eta<\eta_0$
and $t\leq T$,
%
%
\begin{eqnarray}
\label{eqltexpdiff} E_\mu\exp \biggl( \frac{\eta(Z_t(a) - Z_t(b))}{|a-b|^\beta} \biggr)
\leq\exp \bigl(C_1 \Xi_2(2t) \cdot\eta \bigr)
\nonumber
\\[-8pt]
\\[-8pt]
\eqntext{\mbox{for
all } 0<|a-b|\leq2}
\end{eqnarray}
and
%
%
\begin{equation}
\label{eqltexpldp} E_\mu\exp \bigl( \eta Z_t(a) \bigr)
\leq\exp \bigl(C_2 \Xi_1(2t)\cdot\eta \bigr)\qquad\mbox{for
all } a\in\zz{R}^d.
\end{equation}
\end{lemma}

\begin{pf}
The second claim \eqref{eqltexpldp} follows from Lemma 3.4 in
\citet
{Sugitani89}.

To prove \eqref{eqltexpdiff}, following (3.34) in \citet
{Sugitani89}, for a random variable $X$ we say that
\[
E\exp(\eta X) = \exp \Biggl(\sum_{n=1}^\infty
c_n\eta^n \Biggr)
\]
holds formally if for all $k\geq1$, $E|X|^k<\infty$ and
\[
EX^k= \biggl(\frac{d^k (\exp(\sum_{n=1}^k c_n\eta^n
) )}{d\eta^k} \biggr)\bigg\rrvert _{\eta=0}.
\]
By (3.38), (3.45) and (3.48) in \citet{Sugitani89}, we have
formally
%
%
\begin{equation}
\label{eqmgf} E_\mu\exp \bigl(\eta \bigl(Z_t(a) -
Z_t(b) \bigr) \bigr) = \exp \Biggl(2\sum
_{n=2}^\infty \biggl(\frac{\eta}{2}
\biggr)^n \bigl\la\mu, \nu_n(t, \cdot)\bigr\ra \Biggr),
\end{equation}
where for $n\geq2$, and $x\in\zz{R}^d$,
%
%
\begin{eqnarray}
\label{nundefn} \qquad&&\bigl|\nu_n(t,x)\bigr|
\nonumber
\\[-8pt]
\\[-8pt]
\nonumber
&&\qquad\leq b_n
\cdot|a-b|^{n\beta} t^{2-(d+\beta)/2} \int_0^{2t}
s^{-\beta/2} \bigl(p_s(a-x) + p_s(b-x) \bigr)
\,ds,
\end{eqnarray}
and $\{b_n\}$ are defined inductively as follows:
\[
b_1=C_4>0, \qquad b_n = C_5\sum
_{k=1}^{n-1} b_k
b_{n-k}.
\]
Using the proof of Lemma 3.4 in \citet{Sugitani89}, if we let
$f(\eta)=\sum_{n=1}^\infty b_n\eta^n$, then for some $\delta>0$,
%
%
\begin{eqnarray}
\label{fbounds} f(\eta)-C_4\eta= C_5 f(
\eta)^2,\qquad f(\eta)=\frac{1-\sqrt{1-4C_4C_5\eta
}}{2C_5}\leq C\eta
\nonumber
\\[-8pt]
\\[-8pt]
 \eqntext{\mbox{for } 0\leq\eta
\leq\delta.}
\end{eqnarray}
This shows that $\sum_n b_n\eta^n$ has a positive radius of
convergence, and the formal equation \eqref{eqmgf} is indeed an
equation when $\eta$ is sufficiently close to 0 because the Taylor
series for the analytic function on the right-hand side is given by the
left-hand side.
Relation \eqref{eqltexpdiff} then follows easily from the upper
bounds \eqref{nundefn} and~\eqref{fbounds}.
\end{pf}

\begin{cor}\label{corltlrp}
Under the assumptions of the previous lemma, for any fixed $T>0$, there
exist constants $\eta_0=\eta_0(\beta, T)>0,
C_i=C_i(\beta, T)>0$ such that for all $0<\eta<\eta_0$,
%
%
\begin{eqnarray}
\label{eqltdiff} E_\mu\exp \biggl(\frac{\eta|L_T(a) -
L_T(b)|}{|a-b|^\beta} \biggr) \leq2
\exp \bigl(C_1 \Xi_2(2T) \eta \bigr)
\nonumber
\\[-8pt]
\\[-8pt]
\eqntext{\mbox{for all }
0<|a-b|\leq2}
\end{eqnarray}
and
%
%
\begin{equation}
\label{eqltldp} E_\mu\exp \bigl(\eta L_T(a) \bigr) \leq
\exp \bigl(C_2 \Xi_1(2T) \eta \bigr)\qquad \mbox{for all } a\in
\zz{R}^d.
\end{equation}
\end{cor}
\begin{pf} Relation \eqref{eqltldp} follows easily from \eqref
{eqltexpldp}. As for
\eqref{eqltdiff}, by \eqref{eqltexpdiff} and the elementary
inequality $e^{|x|}\le e^x+e^{-x}$,
\begin{eqnarray*}
&& E_\mu\exp \biggl(\frac{\eta|L_T(a) - L_T(b)|}{|a-b|^\beta} \biggr) \\
&&\qquad\leq 2 \exp
\bigl(C_1 \Xi_2(2T) \eta \bigr)\cdot\exp \biggl(
\frac{\eta| (\mu q_T)(a) - (\mu q_T)(b)|}{|a-b|^\beta} \biggr).
\end{eqnarray*}
By (3.44) in \citet{Sugitani89} we have for all $x,y$ and $t>0$,
\[
\bigl|p_t(x)-p_t(y)\bigr|\le c(\beta)t^{-\beta/2}|x-y|^\beta
\bigl(p_{2t}(x)+p_{2t}(y) \bigr),
\]
and so
\begin{eqnarray*}
&&\bigl|(\mu q_T) (a) - (\mu q_T) (b) \bigr| \\
&&\qquad\leq  C
|a-b|^\beta\int_0^T \int
_x t^{-\beta/2} \bigl(p_{2t}(a-x) +
p_{2t}(b-x) \bigr) \mu(dx) \,ds \\
 &&\qquad \leq C|a-b|^\beta
\Xi_2(2T).
\end{eqnarray*}
\upqed\end{pf}

\begin{lemma}\label{lemmafieldmax} Suppose that $\Upsilon(x)$ is an
\mbox{almost surely} continuous
random field on $\zz{R}^d$
such that for some $\eta>0$ and $\beta>0$,
%
%
\begin{equation}
\label{eqreg} \qquad\cases{\displaystyle E\exp \biggl(\frac{\eta|\Upsilon(a) -
\Upsilon
(b)|}{|a-b|^\beta} \biggr) \leq
C_1,&\quad $\mbox{for all } 0<|a-b|\leq\sqrt{d}; \mbox{and}$\vspace*{2pt}
\cr
 E\exp \bigl(\eta\Upsilon(a) \bigr) \leq C_2, & \quad $\mbox{for all } a
\in\zz{R}^d.$}
\end{equation}
Let $M_L=\max_{a\in Q_L(0)} \Upsilon(a)$ [$Q_L(0)$ is the cube of side
length $L$ centered at 0]. Then for all $L\in\zz{N}$ and $m\geq0$,
%
%
\begin{equation}
\label{eqltmaxldp} P(M_L\geq m) \leq \bigl(C_1e^{2d/\beta}+C_2
\bigr)L^d \exp \biggl(-\frac{\eta
m}{1+\gamma} \biggr),
\end{equation}
where $\gamma=8d^{\beta/2}$. In
particular, for any $K>0$, there exists $C>0$, depending only on $K$,
$C_1$, $C_2$ and $\beta$, such that
for all $L\geq1$ and $\lambda\geq1$,
%
%
\begin{equation}
\label{eqltmax} P \biggl(M_L\geq C \frac{\lambda+\log L}{\eta} \biggr) \leq
e^{-K\lambda}.
\end{equation}
\end{lemma}
\begin{pf}
Inequality \eqref{eqltmax} follows by plugging $C\frac{\lambda
+\log
L}{\eta}$ as $m$ into \eqref{eqltmaxldp} and noticing that
when $C$ is large enough, the factor $\exp(-\frac{C\lambda
}{2(1+\gamma)}-\frac{C\log L}{1+\gamma} )$ would be smaller than
$[(C_1e^{2d/\beta}+C_2)L^d]^{-1}$ for all $\lambda,L\geq1$.
Furthermore, it suffices to prove \eqref{eqltmaxldp} for $L=1$ as
the results then follow trivially by dividing $Q_L(0)$ into unit cubes.
We apply Lemma 1 of \citet{Garsia72} with
$p(u)=u^{\beta}$, $\Psi(u)=\exp(\frac{\eta
|u|}{d^{\beta/2}} ), Q_1=Q_1(0)$ and
\[
 B=\int_{Q_1}\int_{Q_1}\exp
\biggl(\frac{\eta|\Upsilon(x)-\Upsilon
(y)|}{|x-y|^\beta} \biggr) \,dx\,dy.
\]
It is easy to check that this $B$ satisfies the hypothesis of
Lemma 1 in the above reference. That result,
or more precisely (10) in the proof,
implies
\begin{eqnarray*}
M_1 &\le&\Upsilon(0)+8\int_0^1
\Psi^{-1} \biggl(\frac{B}{u^{2d}} \biggr) \,d \bigl(u^{\beta}
\bigr)
\\
&\le&\Upsilon(0)+\frac{\gamma}{\eta} \biggl[\log(B)+2d\int_0^1
\log(1/u) \,d \bigl(u^{\beta} \bigr) \biggr]
\\
&=&\Upsilon(0)+\frac{\gamma}{\eta} \bigl[\log(B)+(2d/\beta) \bigr].
\end{eqnarray*}
Therefore for $x\ge0$,
\begin{eqnarray*}
P \bigl(M_1\ge(1+\gamma)x \bigr)&\le &C_2\exp(-\eta x)+P
\bigl(\log(B)\ge\eta x-(2d/\beta) \bigr)
\\
&\le& C_2\exp(-\eta x)+E(B)\exp \bigl((2d/\beta)-\eta x \bigr)
\\
&=& \bigl(C_2+C_1e^{2d/\beta} \bigr)e^{-\eta x},
\end{eqnarray*}
which is \eqref{eqltmaxldp} for $L=1$, and where \eqref{eqreg} is
used to see that $E(B)\le C_1$.
\end{pf}

Combining Lemma \ref{lemmaKt}, Corollary \ref{corltlrp} and Lemma
\ref{lemmafieldmax},
we obtain the following for the local time $L$ of the super-Brownian
motion $Y$.

\begin{prop}\label{propregmax}
For any $T>0$, $M>0$, $K>0$, $A>0$ and $r_0>0$, there exists a constant
$A''$, depending only on $(T, M, K, A, r_0)$, so that for all $\lambda
\ge1$ and all $(A,\lambda,r_0)$-admissible $\mu\in
\mathcal{M}(\zz{R}^d)$ satisfying $|\mu|=\lambda$,
the local time, $L_T(x)$, of the super-Brownian motion
$Y^\mu$ satisfies
\[
P_{\mu} \Bigl(\max_{|x|\leq Me^\lambda} L_T(x) \geq
A''\lambda \kappa_d(\lambda) \Bigr) \leq
e^{-K\lambda},
\]
where $\kappa_d(\lambda)$ is defined in \eqref{eqltexpmax}.
\end{prop}
\begin{pf}
By Lemma \ref{lemmaKt} and Corollary \ref{corltlrp}, for any
fixed $0\leq\eta<\eta_0$, if we let $\eta(\lambda)=\eta/\kappa
_d(\lambda)$,
then for some \emph{fixed} $C_1$ and $C_2$, for all $\lambda\geq1$,
the assumptions~\eqref{eqreg} hold for the random field $L_T(x)$ and
$0<\beta<2-d/2$ by
replacing $\eta$ with $\eta(\lambda)$. The conclusion then follows
from \eqref{eqltmax}.
\end{pf}

\subsection{Local time densities of supercritical super-Brownian motions}

In this and the following subsection, $Y=Y^\mu$ is a super-Brownian
motion with drift one starting at an initial state $\mu$,
let $P_\mu^1 = P_{\mu,0}^{1,0,1}$ be its law.\vspace*{1pt} Further denote by
$P_\mu
^0 = P_{\mu,0}^{0,0,1}$
the law of a (driftless) super-Brownian motion
starting at $\mu$. By Lemma \ref{lemmaGirsanov} we have
\[
\frac{dP_\mu^1}{dP_\mu^0}\bigg\rrvert_{\mathcal{F}^Y_t}:=\Phi_t=\exp
\biggl(M_t^0(1)-\frac{1}{2}\int
_0^t|Y_s| \,ds \biggr),
\]
where $M^i$ denote the martingale measure under $P_\mu^i$, $i=0,1$,
and therefore
%
%
\begin{equation}
\label{marts} M_t^0(1)=M_t^1(1)+
\int_0^t|Y_s| \,ds.
\end{equation}

\begin{lemma}\label{lemmaLR} For any $K\ge1$, $T>0$, $\lambda\ge1$
and $\mu\in\mathcal{M}(\zz{R}^d)$
with $|\mu|=\lambda$,
%
%
\begin{equation}
\label{eqLRbd} E_\mu^0 ( \Phi_T\cdot
\indic_{(\Phi_T\geq e^{K\lambda})} ) \leq\frac{5e^T}{K}.
\end{equation}
\end{lemma}
\begin{pf} Using \eqref{marts} we see the above expectation equals
\begin{eqnarray*}
E_\mu^0 (\Phi_T\cdot\indic_{ (M_T^0-
({1}/{2})\int_0^t|Y_s|\,ds\geq K\lambda)}
) &=& P_\mu^1 \biggl(M^1_T(1)+
\frac{1}{2}\int_0^t|Y_s|
\,ds \geq K\lambda \biggr)
\\
&\leq& \frac{4E_\mu^1 ((M^1_T(1))^2 )}{K^{2}\lambda^{2}} +\frac
{E_\mu^1
(\int_0^T|Y_s| \,ds )}{K\lambda}
\\
&=& \frac{4\lambda(e^T-1)}{K^2\lambda^2}+\frac{\lam(e^T-1)}{K\lam
}\le \frac{5e^T}{K}.
\end{eqnarray*}
\upqed\end{pf}

\begin{prop}\label{propsupersbmreg}
For any fixed $T>0$ and any $\vep>0$, there exist constants
$A=A(T,\vep)> 0$ and $r_0=r_0(T,\vep)\in(0,1]$ such that for all
$\lambda\ge1$ and all $\mu$ with $|\mu|=\lambda$,
%
%
\begin{equation}
\label{eqsupsbmreg} P_\mu^1 \bigl(Y_T
\mbox{ is } (A,\lambda,r_0)\mbox{-admissible} \bigr)\geq1-\vep.
\end{equation}
\end{prop}
\begin{pf} Let $G_{T,\lambda}$ denote the event in \eqref{eqsupsbmreg}.
For $T,\vep$ as above choose $K=K(\vep,T)\ge2$ so that
\[
\frac{10e^T}{K}+e^{-K/2}<\vep,
\]
and then choose $A$ and $r_0$ as in Corollary \ref{corsbmreg} for
this choice of $K$ and $T$, so that they depend ultimately on $T$
and $\vep$. Then the previous lemma and Corollary
\ref{corsbmreg} imply that
\begin{eqnarray*}
P_\mu^1 \bigl(G^c_{T,\lambda} \bigr)
&=&E_\mu^0(\Phi_T\cdot\indic_{G^c_{T,\lambda}})
\\
&\le& E_\mu^0(\Phi_T\cdot
\indic_{(\Phi_T\ge e^{\lambda
K/2})})+e^{\lambda K/2}P_\mu^0
\bigl(G^c_{T,\lambda} \bigr)
\\
&\le&10e^T/K+e^{-\lambda K/2}<\vep,
\end{eqnarray*}
where the choice of $K$ is used in the last inequality.
\end{pf}

The same reasoning, but now using Proposition \ref{propregmax}
in place of Corollary \ref{corsbmreg}, gives the following
proposition.
%
\begin{prop}\label{propsupersbmmax}
For any positive constants $T$, $M$ and $\vep$,
$A>0$ and $r_0>0$ there exists a constant
$A''$, depending only on $(T, M, \vep, A, r_0)$, so that for all
$\lambda\ge1$ and all $(A,\lambda,r_0)$-admissible\vadjust{\goodbreak} $\mu\in
\mathcal{M}(\zz{R}^d)$ satisfying $|\mu|=\lambda$,
the local time, $L_T(x)$, of
$Y^\mu$ satisfies
\[
P_\mu^1 \Bigl(\max_{|x|\leq Me^\lambda}
L_T(x) \leq A''\lambda \kappa
_d(\lambda) \Bigr) \geq1-\vep.
\]
\end{prop}

\subsection{Propagation of supercritical super-Brownian motions}

We continue to let $Y=Y^\mu$ be a super-Brownian motion with drift one
starting at $\mu$, and let
$P_\mu^1$ denote its law. Recall that for $x\in\zz{Z}^d$, $Q_r(x)$
denotes the cube of side length $r$ centered at~$x$, and
$Q(x):=Q_1(x)$.

\begin{lemma}\label{lemmapropspeed}
For any $T\geq1$ and $\vep>0$, there exists a constant
$M=M(T,\vep)>0$ such that for any $\lambda\geq e$ and any $\mu$ satisfying
$|\mu|=\lambda$ and\break  $\operatorname{ Supp}(\mu)\subseteq Q(0)$, we have
\[
P_\mu^1 \bigl( \operatorname{ Supp} \bigl(Y[0,T] \bigr)\subseteq
Q_{M\sqrt{\log
\lambda}}(0) \bigr)\geq1-\vep.
\]
\end{lemma}
\begin{pf}
This is a direct consequence of Theorem A in \citet{Pinsky95}.~%
\end{pf}

Write $\mathcal{N}(0)=\{x\in\zz{Z}^d\dvtx \Vert x\Vert _1=1\}$ for the nearest
neighbors of the origin in $\zz{Z}^d$. Fix a $T$ sufficiently large
such that
%
%
\begin{equation}
\label{eqT} \min_{x\in\mathcal{N}(0)} \min_{y\in Q(0)}
e^{T} ( \mathbf{1}_{Q(x)}*p_{T} ) (y) \geq2,
\end{equation}
where $p_{t} (x)$ is the Gauss kernel in \eqref{eqheatKernel}, and $*$
denotes convolution.

\begin{lemma}\label{lemmaregeneration}
For any $\vep>0$ and $T$ as above, there exists $\lambda_0=\lambda
_0(T,\vep)>0$ such that for any $\mu$ satisfying
$\operatorname{ Supp}(\mu)\subseteq Q(0)$ and $|\mu|=\lambda\geq\lambda_0$,
\[
P_\mu^1 \bigl(Y_T \bigl(Q(x) \bigr)\geq
\lambda\mbox{ for all } x\in\mathcal{N}(0) \bigr) \geq1-\vep.
\]
\end{lemma}
\begin{pf}
By a well-known moment formula [see, e.g., Exercise II.5.2 in \citet
{PerkinsSFnotes}],
together with the assumption that $\operatorname{ Supp}(\mu)\subseteq Q(0)$ and
\eqref{eqT}, for
any $x\in\mathcal{N}(0)$,
\[
EY_T \bigl(Q(x) \bigr) =e^{T}\langle\mu,
\mathbf{1}_{Q(x)}*p_{T} \rangle\ge2|\mu|=2\lambda
\]
and
\[
\var \bigl(Y_T \bigl(Q(x) \bigr) \bigr)\leq e^{2T}
\biggl\langle\mu, \int_0^{T} (
\mathbf{1}_{Q(x)}*p_{(T-s)} )^2* p_{s}
\,ds \biggr\rangle.
\]
Consequently, by the Chebyshev inequality,
\begin{eqnarray*}
P \bigl(Y_T \bigl(Q(x) \bigr)\leq\lambda \bigr)& \leq& P \biggl(
\bigl\llvert Y_T \bigl(Q(x) \bigr) - EY_T \bigl(Q(x)
\bigr) \bigr\rrvert\geq\frac{1}{2}EY_T \bigl(Q(x) \bigr)
\biggr)
\\
&\leq&\frac{4 e^{2T}
\langle\mu, \int_0^{T} (\mathbf{1}_{Q(x)}*p_{(T-s)} )^2* p_{s}  \,ds
\rangle} {
(2\lambda)^2}
\\
&\le &C_T\lambda^{-1}.
\end{eqnarray*}
The conclusion follows.\vadjust{\goodbreak}
\end{pf}

\section{A weak form of local extinction and its consequences}
\label{secweakLocalExtinction}

\subsection{A weak form of local extinction}\label{ssecweakLocalExtinction}

Let
$V_d=\pi^{d/2}/\Gamma(1+d/2)$ be the volume of a unit $d$-dimensional ball.

\begin{prop}\label{propLocextL1}
There exists $\kappa<\infty$
such that for any $\theta\in\zz{R}, \gamma>0$ and $K\in
C_p(\zz{R}^d, \zz{R}_+)$, if $X$ solves
$(\mathrm{MP})^{\theta,1,\gamma}_{\mu,K}$ and $\mu$ satisfies
Assumption {\ref{asmtpinireg1}}, then for any $N\geq1$,
%
%
\begin{equation}
\label{eqL1} E\langle L_{\infty},\mathbf{1}_{B_{N} (0)} \rangle\le
\frac{2|\mu|}{\kappa+ 2\theta^+}+V_d \bigl(\kappa+2\theta^+ \bigr)
(N+1)^d
\end{equation}
and
%
%
\begin{equation}
\label{eqL2} E \bigl\langle L^{2}_{\infty},
\mathbf{1}_{B_{N} (0)} \bigr\rangle \leq4|\mu| + V_d \bigl(
\kappa+2\theta^+ \bigr)^2 (N+1)^d.
\end{equation}
\end{prop}

\begin{pf}
First, observe that there exists $\kappa>0$ such that for any $N\geq
1$ there
exists a function $\varphi=\varphi_{N}\in C^{2}$ such that
%
%
\begin{equation}
\label{Lapbound} |\Delta\varphi|\leq\kappa\sqrt{\varphi} \quad\mbox{and}\quad
\mathbf{1}_{B_{N} (0)}\leq\varphi\leq\mathbf{1}_{B_{N+1} (0)}.
\end{equation}
For example, set $\varphi(x)=\psi(|x|)$ where $\psi=\psi_{N}\dvtx \zz{R}
\rightarrow[0,1]$ is a smooth, even function bounded above and below
by the indicators of $[-N-1,N+1]$ and $[-N,N]$, monotone on
$[-N-1,-N]$ (and therefore also on $[N,N+1]$), and such that (e.g.) $\psi(x)= (x+N+1)^{4}$ for $x\in[-N-1,-N-1/2]$.
Because $\varphi\in C^{2}_{c} (\zz{R}^{d})$, the martingale identity
\eqref{eqDWmgSIRgen} applies (with $\beta=1$), so after taking
expectations, we obtain
\[
E\langle X_{t},\varphi \rangle=\langle\mu,\varphi \rangle+
\frac
{1}{2} E \langle L_{t},\Delta\varphi \rangle +\theta E
\langle L_{t},\varphi \rangle -E \langle L_t, K \varphi
\rangle - E\int_{0}^{t} \langle
X_{s},L_{s}\varphi \rangle \,ds.
\]
A routine integration by parts shows that
\[
\int_{0}^{t} \langle X_{s},L_{s}
\varphi \rangle \,ds= \frac{1}{2} \bigl\langle L_{t}^{2},
\varphi \bigr\rangle.
\]
Since $|\Delta\varphi|\leq\kappa\sqrt{\varphi}$ and
$\langle X_{t},\varphi \rangle \geq0$, it follows that
%
%
\begin{eqnarray}
\label{eqL2rel} -2\langle\mu,\varphi \rangle&\leq&\kappa E \langle
L_{t},\sqrt{ \varphi} \rangle +2\theta^+ E\langle L_{t},
\varphi \rangle -E \bigl\langle L_{t}^{2},\varphi \bigr
\rangle
\nonumber\\
&\leq& \bigl(\kappa+2\theta^+ \bigr) E\langle L_{t},\sqrt{
\varphi } \rangle- E \bigl\langle L_{t}^{2},\varphi \bigr
\rangle
\\
\nonumber
&\leq& \bigl(\kappa+2\theta^+ \bigr) E\langle L_{t},\sqrt{
\varphi } \rangle- \bigl(V_d^{-1} (N+1)^{-d}
\bigr) \bigl(E \langle L_{t},\sqrt{ \varphi} \rangle
\bigr)^{2},
\end{eqnarray}
the last by Cauchy--Schwarz and the fact that $\varphi$ has support
contained in $B_{N+1} (0)$. This clearly gives an upper bound on
$E\langle L_t,\sqrt\phi\rangle$ that is independent of $t$. In
fact,
\eqref{eqL2rel} implies that
\begin{eqnarray*}
&&\biggl(E\langle L_{t},\sqrt{ \varphi} \rangle- \frac{1}{2}V_d
(N+1)^{d} \bigl(\kappa+2\theta^+ \bigr) \biggr)^2 \\
&&\qquad\leq
\frac{1}{4} \bigl(V_d (N+1)^{d} \bigr)^2
\bigl(\kappa+2\theta^+ \bigr)^2 + 2\langle\mu,\varphi \rangle
V_d (N+1)^{d}
\\
&& \qquad\leq\biggl(\frac{1}{2}V_d (N+1)^{d} \bigl(
\kappa+2 \theta^+ \bigr) + \frac{2\langle\mu,\varphi \rangle}{(\kappa
+2\theta^+ )} \biggr)^2,
\end{eqnarray*}
and hence
\begin{eqnarray*}
&&E\langle L_{t},\sqrt{ \varphi} \rangle\\
&&\qquad\leq\frac{1}{2}V_d
(N+1)^{d} \bigl(\kappa+2\theta^+ \bigr) + \biggl(\frac
{1}{2}V_d
(N+1)^{d} \bigl(\kappa+2\theta^+ \bigr) + \frac{2\langle\mu
,\varphi \rangle}{(\kappa+2\theta^+ )} \biggr)
\\
&&\qquad=V_d \bigl(\kappa+2\theta^+ \bigr) (N+1)^d+
\frac{2\mu(\phi)}{\kappa+2\theta^+}.
\end{eqnarray*}
Letting $t\rightarrow\infty$ yields
%
%
\begin{equation}
\label{eqL10} E\langle L_{\infty},\sqrt{\varphi} \rangle \le
V_d \bigl(\kappa+2\theta^+ \bigr) (N+1)^d+
\frac{2\mu(\phi)}{\kappa+2\theta^+}.
\end{equation}
Relation \eqref{eqL1} follows, since $\sqrt{\varphi}$ bounds the indicator
function of $B_{N} (0)$. Finally, by the second inequality in \eqref
{eqL2rel},
\[
E \bigl\langle L_{\infty}^{2},\varphi \bigr\rangle
\leq2|\mu| + \bigl( \kappa+ 2\theta^+ \bigr) E\langle L_\infty,\sqrt{
\phi}\rangle.
\]
Relation \eqref{eqL2} follows from \eqref{eqL10}.
\end{pf}

\begin{remark}
The above proposition easily shows that each of the terms on the
right-hand side of \eqref{eqDWmgSIR} converges a.s. as $t\to\infty$.
Therefore
$X_t(\phi)$ converges \mbox{a.s.} as $t\to\infty$ and clearly the
limit must
be $0$ by the above. This shows that $X_t(K)$ approaches $0$ as
$t\to\infty$ for all compact sets $K$ \mbox{a.s.} Our Theorem \ref{thmmlocext}
asserts a
much stronger result, namely that $X_t(K)=0$ for large enough $t$ a.s.
\end{remark}

\subsection{\texorpdfstring{Universality of the critical values $\theta_c$}
{Universality of the critical values theta c}}\label
{sseccvuniversal}

For any $\mu$ satisfying Assumption~\ref{asmtpinireg2}, if $X$
solves \eqref{eqDWmgSIR} with $\theta\le0$, then $P(X \mbox{
survives})=0$ because $X$ is dominated by a critical super-Brownian
motion (by Proposition \ref{propBRWenvelope}), which goes extinct
almost surely [see, e.g., equation (5.7) in \citet{Feller51} or
(II.5.12) in \citet{PerkinsSFnotes}].
Lemma \ref{lemmacomparisontheta} and Remark \ref{remarksurvivalByLT}
therefore imply that for any such
$\mu$ and any function $K \in C_{p} (\zz{R}^{d}; \zz{R}_+)$, there
is a
critical value
$\theta_c(\mu,K)\in[0,\infty]$ so that a spatial epidemic $X$ with
suppression rate $K$ and transmission parameter $\theta$ [see~\eqref{eqDWmgSIRgen}] survives with positive probability if
$\theta>\theta_{c} (\mu,K)$ and with zero probability if $\theta
<\theta_{c} (\mu,K)$.

\begin{prop}\label{propthetaccommon} The critical value
$\theta_c(\mu, K)$ depends only on the dimension $d$ and not on the
choice of $0\neq\mu$ satisfying Assumption \ref{asmtpinireg1} or
$K\in C_{p} (\zz{R}^{d}; \zz{R}_+)$.
\end{prop}

\begin{pf} In this argument $\theta$ will be fixed and $\gamma=1$,
so we suppress the dependence of the laws $P^{\theta,1,\gamma}_{\mu
,K}$ on $\theta$ and $\gamma$. By Theorem \ref{thmmexistuniq}, for any
measure $\mu
\in\mathcal{M}_{c} (\zz{R}^{d})$ satisfying
Assumption \ref{asmtpinireg1} and any two suppression rate
functions $K,K'\in C_{p} (\zz{R}^{d}; \zz{R}_+)$, the laws $P_{\mu,K}$
and $P_{\mu,K'}$ are mutually
absolutely continuous on $\mathcal{F}^{X}_{t}$, with Radon--Nikodym
derivative \eqref{eqRNformulaLKK}. Since $K$ and $K'$ both have
compact support,
inequality \eqref{eqL2} of Proposition \ref{propLocextL1} implies
that the integrals in the likelihood ratio converge, so that
\[
\lim_{t \rightarrow\infty} \biggl(\frac{dP_{\mu,K'}}{dP_{\mu
,K}} \biggr)_{\mathcal{F}_{t}}
:=Y
\]
exists and is positive $P_{\mu,K}$-almost surely. Hence, by Fatou's lemma,
\[
P_{\mu,K'} (X \mbox{ survives})= \lim_{t \rightarrow\infty}P_{\mu,K'}
\bigl(|X_{t}|>0 \bigr) \geq E_{\mu,K} (Y \mathbf{1}_{\{X\ \mathrm{survives}\}} ).
\]
It follows that if $X$ survives with positive $P_{\mu,K}$
probability, then it also survives with positive $P_{\mu,K'}$
probability. Reversing the roles of $K$ and $K'$ shows that the reverse
is also true. Therefore, the critical value $\theta(\mu,K)$ does not
depend on $K$.

To complete the proof, it suffices, in view of the preceding
paragraph, to prove that if $X$ survives with positive probability
under $P_{\mu,0}$, then it survives with positive probability under
$P_{\nu,0}$ for $\nu\neq0$. By the Markov property [Theorem \ref
{thmmexistuniq}(d)],
\[
P_{\mu,0} (X\mbox{ survives})=E_{\mu,0} \bigl(P_{X_{1},L_{1}}
(X \mbox{ survives}) \bigr),
\]
and similarly for $P_{\nu,0}$. By the argument of the preceding paragraph,
\[
P_{X_{1},L_{1}} (X\mbox{ survives})>0 \quad\Longleftrightarrow\quad P_{X_{1},0}(X
\mbox{ survives})>0,
\]
so for both $\omega=\mu$ and $\omega=\nu$,
\[
P_{\omega,0} (X\mbox{ survives})>0\quad \Longleftrightarrow \quad E_{\omega,0}
\bigl(P_{X_{1},0}(X\mbox{ survives}) \bigr)>0.
\]
But the laws of $X_{1}$ under
$P_{\mu,0}$ and $P_{\nu,0}$ are mutually absolutely continuous. (This
can be seen as follows. First, by the absolute continuity results in
\citet{EP91}
[see, e.g., Theorem III.2.2 in \citet{PerkinsSFnotes}], if $P_{\mu}$
and $P_{\nu}$ are the laws of
super-Brownian motions with initial conditions $\mu$ and $\nu$, then
the distributions of $X_{1}$ under $P_{\mu}$ and $P_{\nu}$ are
mutually absolutely continuous. Second, by
Theorem \ref{thmmexistuniq}(a), for any initial measure $\omega$ the
measures $P_{\omega}$ and $P_{\omega,0}$ are mutually absolutely
continuous.) Therefore,
\[
P_{\mu,0} (X\mbox{ survives})>0 \quad\Longleftrightarrow\quad P_{\nu,0} (X
\mbox{ survives})>0.
\]
\upqed\end{pf}

Note that the above arguments do not require $\mu$ to satisfy the
stronger Assumption \ref{asmtpinireg2}, instead just the original
Assumption \ref{asmtpinireg1}.

\subsection{Extinction in dimension one}\label{sec1dext}

\begin{prop}\label{proposition1DExtinction}
If $d=1$, then for {every} $\theta\in\zz{R}$ and every initial measure
$\mu$ that satisfies Assumption \ref{asmtpinireg1}, the solution
$X$ of the martingale problem \eqref{eqDWmgSIR} dies out almost surely.
\end{prop}

\begin{pf}
First, by Proposition \ref{propBRWenvelope}, on some probability
space there
is a version of the process $X$ and a super-Brownian motion
$\ol{X}$ with drift $\theta$ such that
$X_{0}=\ol{X}_{0}=\mu$ and $X_{t}\leq\ol{X}_{t}$ for all
$t\geq0$.

By a result of \citet{Pinsky95}, there is a positive constant
$C=C_{\theta}<\infty$ such that almost surely the support of the
random measure $\ol{X}$ is eventually contained in the interval
$[-Ct,Ct]$. Since $\ol{X}$ dominates $X$, the same is true for $X$.
Now by Lemma \ref{lemmasurvchar}, on the event that $X$ survives, the
total mass of
the measure $X_{t}$ must diverge. Because this mass is
(eventually) contained in $[-Ct,Ct]$, it follows from L'Hospital's
rule that on the event of survival, the
occupation density process $L_{t} (x)$ must satisfy
\[
\frac{1}{t}\int_{-Ct}^{Ct}
L_{t} (x) \,dx = \frac{1}{t}\int_0^t
|X_{u}| \,du \longrightarrow\infty.
\]
Hence, if there is positive probability
of survival, then
\[
\frac{1}{t}E\int_{-Ct}^{Ct}
L_{t} (x) \,dx \longrightarrow\infty.
\]
But this contradicts \eqref{eqL1} in Proposition \ref{propLocextL1}.
\end{pf}

\section{Proof of survival when $d=2$ or $3$}\label{secsurv}

In this section we prove that in dimensions $2$ and $3$, for all
sufficiently large values of the transmission rate $\theta$,
spatial epidemics---that is, solutions of the martingale problem
\eqref{eqDWmgSIR}---survive with positive probability. By
Proposition \ref{propthetaccommon}, the critical value $\theta_c$
for survival in dimensions $d=2,3$ does not depend on the initial
mass distribution $\mu$; hence it suffices to prove that for \emph
{some} finite measure $\mu$,
there is positive probability of survival. The proof
will make use of an auxiliary $3$-dependent \emph{site percolation process}:
this will be constructed in such a way that if
percolation occurs with positive probability, then the epidemic must
survive with positive probability. We will show that by taking $\theta$
sufficiently large, we can make the density of the
site percolation arbitrarily close to $1$. Since percolation occurs
with positive probability in a site percolation process when the
density is near $1$ [see, e.g., Theorem 4.1 of \citet{DurrettLect}], it
will follow that for large values of $\theta$ the epidemic process
will survive with positive probability. We refer the reader to
Chapter~4 of \citet{DurrettLect} for terminology and a general framework
for such comparison arguments.

\subsection{Scaled process}

We assume $d=2$ or 3 throughout this section. Let $X$ be a spatial
epidemic process with transmission rate $\theta$ and initial mass
distribution $\mu$, that is, a solution to the martingale\vadjust{\goodbreak} problem
\eqref{eqDWmgSIR}. It will be convenient to work with a rescaled
version of the spatial epidemic defined as follows: for any
$\theta>0$,
\[
U_t(\psi) = \theta X_{t/\theta} \bigl(\psi(\sqrt{\theta} \cdot)
\bigr)\qquad \mbox{for all } \psi\in C_c^2 \bigl(
\zz{R}^d \bigr).
\]
The effect of this rescaling is described by Lemma
\ref{lemmascaling}: in particular, $U$ satisfies the martingale
problem $(\mathrm{MP})^{1,\beta,1}_{\tilde{\mu},0}$ with
$\beta=\theta^{(d-6)/2}$ and $\tilde{\mu}$ defined by $\int\psi(x)
\,d\tilde{\mu} (x) = \theta\int\psi(\sqrt{\theta}x) \,d\mu(x)$.
For notational ease, we will use the notation
%
%
\begin{equation}
\label{eqbeta} \beta=\beta(\theta)=\theta^{(d-6)/2}
\end{equation}
in this section, and we will
drop the tilde on the initial measure $\mu$. We will show that when
$\theta$ is sufficiently large, for a suitable initial condition
$\mu$, the process $U$ survives with positive probability.

\subsection{Sandwich lemma}\label{secsandwichrescaled}

By Lemma \ref{lemmasandwich0}, a spatial epidemic process can be
bounded below and above by super-Brownian motions with different drift
terms up to the time that its local time density exceeds some
threshold. We now explain how the result of
Lemma \ref{lemmasandwich0} translates to the rescaled processes.

For any $\mu\in\mathcal{M}_c(\zz{R}^d)$ satisfying Assumption
\ref{asmtpinireg2}, and any function $K\in C_p(\zz{R}^d,\zz{R}_+)$,
let $U$ be a solution of the martingale problem
$(\mathrm{MP})_{\mu, K}^{1,\beta,1}$, that is, the spatial epidemic with
transmission rate $1$, branching rate $1$, inhibition parameter
$\beta$, local suppression rate $K$ and initial mass
distribution $\mu$. In addition, for any fixed constant $\kappa>0$,
let $\overline{U}$ and $\underline{U}$ be super-Brownian motions with
drift 1 and
drift $1-\beta\cdot\kappa$, respectively.
Denote by $\ol{M}$ and
$\ul{M}$ the orthogonal martingale measures associated with
$\ol{U}$ and $\ul{U}$, respectively.\looseness=1

\begin{lemma}\label{lemmasandwich} Versions of the processes
$U$, $\overline{U}$ and $\underline{U}$, all with
the same initial condition $\mu$, can be built
on a common probability space in such a way that
\[
\underline{U}_t \le U_t\le\overline{U}_t
\qquad\mbox{for all } t\leq\tau,
\]
where
\[
\tau=\inf \Bigl\{t\geq0\dvtx \Bigl(\max_x K(x) \Bigr) +
\Bigl( \max_x\beta L_t(U,x) \Bigr) \geq\beta
\kappa \Bigr\}.
\]
\end{lemma}
\begin{pf} This follows by first rescaling $U$, $\ol U$ and $\ul U$ as
in Lemma \ref{lemmascaling} so that the $\beta$ parameter becomes $1$,
and the drift parameters (or transmission rates) of $\ul U$ and $\ol U$ become
$\theta(1-\beta\kappa)$ and $\theta$, respectively.
Then one may apply Lemma \ref{lemmasandwich0} to the rescaled
process. Finally undoing the scaling leads to the required conclusion.
\end{pf}

By Lemma \ref{lemmaGirsanov}(a), the law of $\underline{U}$ is
absolutely continuous with respect to that of~$\overline{U}$, and the
likelihood ratio on $\sF_t$ is
%
%
\begin{equation}
\label{eqLRlUuU} \operatorname{ LR}^{\kappa}_t= \exp \biggl\{-\beta
\kappa\ol{M} _t (1) - \frac{\beta^2\kappa
^2}{2}\int_0^t
|\overline{U}_s| \,ds \biggr\}.
\end{equation}

\subsection{Percolation probability estimates}

Recall 
that $Q_r(x)$ denotes
the cube of side length $r$ centered at $x$, and, as before, we abbreviate
$Q(x):=Q_{1}(x)$.
The auxiliary site percolation processes will be constructed by
partitioning the space $\zz{R}^{d}$ into cubes $Q (x)$ of side
length $1$ centered at lattice points $x\in\zz{Z}^{d}$, and then
using the
behavior of the superprocesses in the cube $Q (x)$ to determine
whether the site $x$ will be occupied or not in the auxiliary
percolation process. Roughly, a site $x$ will be occupied if, within
a certain fixed amount of time $T<\infty$, the measure-valued process
$U$ started from a certain initial mass distribution supported by
$Q(x)$ manages to generate a sufficiently large total mass in each of
the adjacent cubes $Q(y)$ while simultaneously not accumulating too
much local time.
The objective of this section is to develop estimates that will
allow us to conclude that if $\theta$ is large (hence $\beta$ is
small), then $x$ is occupied
with high probability.

Define the grid $\Gamma$ to be $\zz{Z}^2_{+}$ when
$d=2$ and $\zz{Z}^2_{+}\times\{0\}$ when $d=3$, where
$\zz{Z}^2_{+}:=\{x=(x_1,x_2)\in\zz{Z}^2\dvtx  x_i\geq0, i=1,2\}$. For
$x,y\in\Gamma$, we say that
%
%
\begin{equation}
\label{eqtotalorder} x\prec y \cases{\mbox{if }\Vert x\Vert _1<
\Vert y\Vert _1 \quad\mbox{or } \vspace*{2pt}
\cr
\Vert x\Vert
_1=\Vert y\Vert _1 \mbox{ and } x_{1} <
y_{1}.}
\end{equation}
Here $\Vert x\Vert _1:=\sum_i |x_i|$ is the $\ell_1$-norm. This defines a
\emph{total} order on $\Gamma$, and so the points of the lattice can
be enumerated as $0=x(1) \prec x(2)\prec
\cdots.$ The notation $x\preceq y$ is understood as $x\prec y$ or
$x=y$. Define $\mathcal{A}(x)$ to be the set of $y\in\Gamma$
such that $x\prec y$ and $\Vert x-y\Vert _1=1$. In other words, when $d=2$, for
any $(x_1,y_1)\in\zz{Z}^2_{+}$,
$\mathcal{A}((x_1,y_1))=\{ (x_1,y_1 + 1), (x_1+1,y_1)\}$; and similarly
for $d=3$. We shall call any $y\in
\mathcal{A}(x)$ an ``immediate offspring'' of $x$, and $x$ an
``immediate predecessor'' of $y$.

Fix $T$ so large that \eqref{eqT} holds.
For any $\vep>0$, let $A(\vep)=A(T,\vep)$ and $r_0(\vep)=r_0(T,\vep)$
be the constants specified in Proposition \ref{propsupersbmreg}. For
any meas\-ure-valued process $X$ with
local time density $L_t(X)$, and for any $x\in\Gamma$, $M>0$,
$\chi>0$, $\lambda>0$ and $\vep>0$, define the following events:
%
%
\begin{equation}
\label{eqsreg} \cases{ F^1(M;X,x)= \bigl\{\operatorname{ Supp}
\bigl(L_T(X) \bigr)\subseteq Q_M(x) \bigr\};
\vspace*{2pt}
\cr
\displaystyle F^2(\chi;X)= \Bigl\{\max_y
L_T(X,y) \leq\chi \Bigr\};\vspace*{2pt}
\cr
F^3(X,x)=
\bigl\{X_T \bigl(Q(y) \bigr)\geq|X_0|, \mbox{ for all }
y\in\mathcal{A}(x) \bigr\};\vspace*{2pt}
\cr
F^{4}(\vep;X) = \bigl\{
X_T \mbox{ is } \bigl(A(\vep/4), |X_0|,
r_0(\vep/4) \bigr)\mbox{-admissible} \bigr\}.}
\end{equation}
Observe that these events
depend on the choice of $T$.
For brevity we will write
\begin{eqnarray*}
\ol{F}^1(M)& =& F^1(M;\ol{U},0),\qquad \ol{F}^2(
\chi) = F^2(\chi;\ol{U}),\qquad \ol{F}^3 = F^3(
\ol{U},0)\quad \mbox{and}\\
 \ol{F}^4(\vep)& =& F^4(\vep;
\ol{U}).
\end{eqnarray*}

Define functions
%
%
\begin{equation}
\label{eqfdtheta} f_d(\theta)=\cases{ \theta^{1/2}, &\quad $ \mbox{
when } d= 2$,
\vspace*{2pt}
\cr
\log\theta,&\quad $\mbox{when } d= 3$}
\end{equation}
and
%
%
\begin{eqnarray}
\label{eqchi-def} \widetilde{M}&=&\wt{M}(M,\theta)={ \bigl[M \sqrt {\log
f_d(\theta)}+1 \bigr]} \quad\mbox{and}
\nonumber
\\[-8pt]
\\[-8pt]
\nonumber
\chi&=&\chi \bigl(A'',
\theta \bigr)=A'' f_d(\theta)
\kappa_d \bigl(f_d(\theta) \bigr),
\end{eqnarray}
where $\kappa_d(\cdot)$ is the function defined in
\eqref{eqltexpmax}.

\begin{lemma} \label{lemmapercbrw}
For any $\vep_0>0$, there
exist positive constants $\theta_0, M$ and $A''$, depending only on $T$
and $\vep_0$, such that if $\theta>
\theta_0$, then for any initial measure $\mu$
supported by $Q(0)$,
of total mass $|\mu|=f_d(\theta)$ and $(A(\vep_0/4), f_d(\theta),
r_0(\vep_0/4))$-admissible,
the super-Brownian motion $\ol{U}$ with drift 1 and initial mass
distribution $\mu$ satisfies
%
%
\begin{equation}
\label{eqopenc} P \bigl( \overline{F}^1(\wt{M}) \cap
\overline{F}^2(\chi) \cap\overline{F}^3 \cap
\overline{F}^{4}(\vep_0) \bigr)\geq1-
\vep_0.
\end{equation}
\end{lemma}
\begin{pf}
This is a direct consequence of Lemma \ref{lemmapropspeed},
Proposition \ref{propsupersbmmax}, Lemma \ref{lemmaregeneration},
and Proposition \ref{propsupersbmreg}. More specifically,
by Lemma \ref{lemmapropspeed},
there exists constant $M=M(T,\vep_0)$ such that
$P(\overline{F}^1(\wt{M}))\geq1- \vep_0/4$.
Moreover, by Proposition \ref{propsupersbmmax} with $\vep=\vep_0/4$,
there exists $A''$ such that $P(\overline{F}^2(\chi))\geq1- \vep_0/4$.
[Note that by Proposition \ref{propsupersbmmax}, $A''$ only depends
on $(T,M,\vep_0,A,r_0)$, and in our case the $M$, $A$ and $r_0$ all
only depend on $(T,\vep_0)$, so ultimately $A''$ only depends on
$(T,\vep_0)$.] Next, by Lemma \ref{lemmaregeneration}, there exists
$\theta_0>0$ such that for any $\theta>\theta_0$,
%
%
\begin{equation}
\label{F3bnd} P \bigl( \overline{F}^3 \bigr)\geq1-
\vep_0/4.
\end{equation}

Finally, by Proposition \ref{propsupersbmreg}, as long as $\theta$
is such that $f_d(\theta)\geq1$,
$P(\overline{F}^{4}(\vep_0))\geq1- \vep_0/4$.
\end{pf}

The next result explains the choice of $f_d(\theta)$.
%
\begin{cor}\label{corpercsandwich}
For any positive constants $\theta, M$ and $A''$, let
$\ul{U}=\ul{U}^{\theta,\kappa}$ be the super-Brownian motion with
drift $1-\beta\kappa$ and initial mass distribution $\mu$, where
%
%
\begin{equation}
\label{eqepsilon3} \kappa= \wt{M}^2\chi= \bigl[M \sqrt{ \log
f_d(\theta)} + 1 \bigr]^2\cdot A''
f_d(\theta)\kappa_d \bigl(f_d(\theta) \bigr).
\end{equation}
Then for any $\vep_0\in(0,1)$, there exists
$\theta_0>0$ such that if $\theta> \theta_0$ and if the initial
condition $\mu$
is supported by $Q(0)$ and
of total mass $|\mu|=f_d(\theta)$,
then
%
%
\begin{equation}
\label{eqopenc} P \bigl( F^3(\ul{U},0) \bigr)\geq1-3
\vep_0/2.
\end{equation}
\end{cor}

\begin{pf} Let $F^{5}(\vep_0)=\{\operatorname{LR}^{\kappa}_T\le1+\vep_0\}$ for
the likelihood ratio $\operatorname{LR}^{\kappa}_t$ defined in~\eqref{eqLRlUuU}.
Using the fact that
$\ol{M}_t(1)=\ul{M}_t(1)-\beta\kappa\int_0^t|\ul{U}_s|\,ds$,
we have
%
%
\begin{eqnarray}
\label{LRUI}
&&E^{\ol{U}} \bigl(\operatorname{LR}^{\kappa}_T
\cdot \indic_{(F^{5}(\vep_0))^c} \bigr)
\nonumber\\
&&\qquad=P^{\ul{U}} \biggl(-\beta\kappa\ol{M}_T(1)-
\bigl( \beta^2\kappa^2 /2 \bigr)\int_0^T|
\underline{U}_s|\,ds\ge\log(1+\vep_0) \biggr)
\\
&&\qquad=P^{\ul{U}} \biggl(-\beta\kappa\ul{M}_T(1)+ \bigl(
\beta^2\kappa^2 /2 \bigr)\int_0^T|
\underline{U}_s|\,ds\ge\log(1+\vep_0) \biggr).\nonumber
\end{eqnarray}
Here and below we use $E^{\ol{U}}$ and $P^{\ol{U}}$ ($E^{\ul{U}}$ and
$P^{\ul{U}}$, resp.) to indicate that the expectation and probability
are taken with respect to the law of $\ol{U}$ (\ul{U}, resp.). Since
for any $s\geq0$,
$E^{\ul{U}}(|\underline{U}_s|) = |\underline{U}_0| e^{(1-\beta
\kappa
)s}$ [see, e.g., equation (5.4) in \citet{Feller51}], we have that
\[
E^{\ul{U}}\int_0^{T} |
\underline{U}_s| \,ds\le f_d(\theta)\int
_0^{T} e^{s} \,ds\leq
C_T f_d(\theta).
\]
This and the definitions of $\beta$ [in \eqref{eqbeta}], $\kappa$
and $f_d$
imply that\break  $E^{\ul{U}} (\beta^2\kappa^2 \int_0^T|\ul{U}_s|
\,ds)=o(1)$ as
$\theta$ goes to infinity. Since $\beta\kappa\underline{M}_T(1)$ has
quadratic variation $\beta^2\kappa^2 \int_0^T|\ul{U}_s| \,ds$, we see
from the above that both terms inside the $P^{\ul{U}}$-probability in
\eqref{LRUI} approach $0$ in probability as $\theta\to\infty$. It
follows that there exists $\theta_0>0$ such that
for any $\theta>\theta_0$,
$E^{\ol{U}}(\operatorname{LR}^{\kappa}_T\cdot\indic_{(F^{5}(\vep_0))^c})<\vep_0$.
Therefore by \eqref{F3bnd}, the complement of the event $ F^3(\ul
{U},0)$ has probability bounded above by
\begin{eqnarray*}
\vep_0+E^{\ol{U}} \bigl(\operatorname{LR}^{\kappa}_T
\cdot\indic_{\{\operatorname{LR}^{\kappa}_T
\le1+\vep_0\}} \cdot\indic_{(\overline{F}^3)^c} \bigr) \le
\vep_0+(1+\vep_0)\cdot\frac{\vep_0}{4}\le
\frac{3\vep_0}{2}.
\end{eqnarray*}
\upqed\end{pf}

Combining the sandwich lemma (Lemma \ref{lemmasandwich}) and the
previous two results we obtain:

\begin{prop}\label{propperc} For any $\vep_0>0$ there exist positive
constants $\theta_0, M, A''$ such that,
for any $\theta> \theta_0$, any initial condition $\mu$
satisfying the hypotheses of Lemma \ref{lemmapercbrw}
and any $K\in C_p(\zz{R}^d,\zz{R}_+)$ such that
\[
K(x) + \beta\chi\cdot\indic_{Q_{\wt{M}}(0)}(x) \leq\beta\kappa \qquad\mbox {for all }
x\in\zz{R}^d,
\]
the process $U$ solving $(\mathrm{MP})_{\mu, K}^{1,\beta,1}$ satisfies
%
%
\begin{equation}
\label{eqperc} P \bigl( F^1(\wt{M};U,0) \cap F^2(
\chi;U) \cap F^{3}(U,0) \cap F^{4}(\vep_0;U)
\bigr)\geq1-3\vep_0,
\end{equation}
where the events $F^i\ (i=1,2,3,4)$ are defined as in
\eqref{eqsreg}
by replacing $X$ with $U$ and
$x$ with $0$.
\end{prop}
\begin{pf}
On the event $\overline{F}^1(\wt{M})\cap
\overline{F}^2(\chi)$,
\[
L_T(\ol{U},x) \leq\chi\cdot\indic_{Q_{\wt{M}}(0)}(x).
\]
Therefore by the assumption on $K$ and Lemma \ref{lemmasandwich},
\[
\ul{U}_t\leq U_t\leq\ol{U}_t \qquad\mbox{for all
} t\leq T \mbox{ on } \overline{F}^1(\wt{M})\cap \overline{F}^2(
\chi).
\]
The required bound now follows from Lemma \ref{lemmapercbrw},
Corollary \ref{corpercsandwich} and an elementary argument.
\end{pf}

\subsection{Proof of survival}
%
\begin{prop}\label{propsurvscaled}
For \emph{some} finite measure $\mu$ and \emph{some} $\theta<\infty$,
if $U$ solves $(\mathrm{MP})_{\mu,0}^{1,\beta,1}$ with
$\beta=\theta^{(6-d)/2}$,
then
\[
P(U \mbox{ survives}) > 0.
\]
\end{prop}

\begin{pf}
Fix a $T$ so that \eqref{eqT} holds. Fix $\varepsilon_{0}>0$ small
enough such that any 3-dependent oriented
site percolation process on $\zz{Z}^2_+$ with density at least
$(1-6\vep_0)$ has positive probability of percolation. For this $\vep
_0$, let $\theta
>\theta_{0}$, where $\theta_{0}$ is as in Proposition \ref{propperc}.
Then choose a measure $\mu$ so that it satisfies the hypotheses of
Lemma \ref{lemmapercbrw} with $\vep_0$ specified as above.
Let $L_t(x)$ denote the local time density of~$U$, and let
$L_\infty(x) =\lim_{t\to\infty} L_t(x)$ for all $x\in\zz{R}^d$.
By
Lemma \ref{lemmapropspeed} and (a scaled version of) Proposition
\ref
{propBRWenvelope}, almost surely,
\[
L_\infty\mbox{ is not compactly supported}\quad \Lto\quad U \mbox{ survives}.
\]
It therefore suffices to show that $L_\infty$ is not compactly
supported with positive probability. To \,do so, we will specify an
algorithm that produces a (random) set $\Omega$ consisting of integer
sites such that:
\begin{longlist}[(ii)]
\item[(i)] $L_\infty(Q(x))>0$ for all $x\in\Omega$;
\item[(ii)]$\Omega$ is infinite with positive probability.
\end{longlist}
The set $\Omega$ will be the connected cluster containing the origin
in a $3$-dependent site percolation process with density
$\geq1-6\varepsilon_{0}$.

Let us first give an overview of the algorithm. Recall that the grid
$\Gamma$ is defined to be $\zz{Z}^2_{+}$ when $d=2$ and
$\zz{Z}^2_{+}\times\{0\}$ when $d=3$. Initially all sites
$x\in\Gamma$ are designated \emph{vacant} (i.e., $\Omega
=\varnothing$). Our algorithm relies on the comparison in Proposition
\ref
{propcompimmigrationseries}. Starting from the origin, following the
total order
$0=x (1)\prec x (2) \prec\cdots$ on $\Gamma$ introduced in
\eqref{eqtotalorder}, we shall define stopping times $\tau_{i}$,
random measures $\mu_{i}, \nu_i$ and
suppression rates $K^*_{i}$. Proposition \ref
{propcompimmigrationseries} allows us to couple $U$ with another
process $U^*$, which, on any time interval between two successive
stopping times, is a usual spatial epidemic process. The set $\Omega$
will be determined by $U^*$. Proposition \ref
{propcompimmigrationseries} ensures that $L_t^U\geq L_t^{U^*}$,
which will be used to ensure property (i). Depending on how $U^*$
behaves for $t\in[\tau_{i-1}, \tau_{i}]$, we may change the status of
site $x=x(i)$ from vacant to \emph{occupied}, and add $x$ to the set
$\Omega$. Roughly speaking, this will be done if and only if the
spatial epidemic $U^*_t$ for $t\in[\tau_{i-1}, \tau_{i}]$ succeeds in
(1) putting enough mass in adjacent cubes at time $\tau_{i}$ and (2)
accumulating only a small
amount of local time. On the event that the status of site
$x$ is changed to occupied, for each successor $y\in\sA(x)$, we will
be able to extract a ``nice'' mass distribution $\mu^y$ in such a
way that if a spatial epidemic is initiated by $\mu^{y}$, then
it will have high probability of making events (1) and (2)
occur, in other words, so that site $y$ will also be added to $\Omega$
with high probability. By keeping this probability above the
percolation threshold we will ensure that the random set $\Omega$
consisting of all the \emph{occupied} sites will be infinite with
positive probability.

We now introduce some notation. In addition to $\theta_0$, assume $M,
A''$ are as in Proposition
\ref{propperc}, so that \eqref{eqperc} holds. In the algorithm, we
will repeatedly use stopping rules $\tau=\tau(Y;\ell;R) $ defined
as follows: for a measure-valued process $Y\in
C([0,\infty);\mathcal{M}_c(\zz{R}^d))$ with local time $L_{t}^Y$, a
threshold $\ell>0$, and a region $R\subseteq\zz{R}^d$,
%
%
\begin{equation}
\label{eqstoppingtau} \tau(Y;\ell;R):= \inf \Bigl\{t \dvtx \max
_{x} L_{t}^Y (x)\geq\ell\mbox{ or }
\operatorname{Supp} \bigl(L_{t}^Y \bigr)\nsubseteq R \Bigr\}
\wedge T.
\end{equation}
We will also repeatedly use the notation $F^1,\ldots, F^4$
as introduced in
\eqref{eqsreg}
to define the so-called ``good events.'' For notational ease, for each
$i=1,2,\ldots,$ associated with site $x(i)$ in the above overview, we
write $U^i_t = U^{*}_{t+\tau_{i-1}}$ for $t\geq0$ (i.e., the process
$U^{*}$ shifted and restricted to $t\geq\tau_{i-1}$), and
%
%
\begin{equation}
\label{eqgoodevent} G^{i}=F^1 \bigl(
\widetilde{M};U^i, x(i) \bigr) \cap F^2 \bigl(
\chi;U^i \bigr) \cap F^{3} \bigl( U^i, x(i)
\bigr) \cap F^{4} \bigl(\vep_0; U^i \bigr).
\end{equation}
The event $G^{i}$ will be called a ``good'' event. In plain
language, ignoring the technical restriction
$F^{4}$, on such a good event, before time $T$,
the spatial epidemic $U^i$ has not accumulated local time density more than
$\chi\cdot\indic_{Q_{\wt{M}}(x(i))}$, and in the meanwhile, at time $T$,
it spreads at least $|U^i_0|$ amount of mass in all the
cubes $Q(y)$ for $y\in\mathcal{A}(x(i))$.

Now we describe our algorithm in detail.
In order to apply Proposition \ref{propcompimmigrationseries}, we
need to define four sequences: random measures $\mu_{i}, \nu_{i}$,
suppression rate functions $K^*_{i}$ and stopping times $\tau_i$. The
random measures $\mu_{i}$ and $\nu_{i}$ will be defined through an
auxiliary random measure sequence $w_i$. The suppression rate functions
$K^*_{i}$ will be deterministic functions as follows: $K^*_{0}\equiv
0$, and for $i\geq1$,
%
%
\begin{equation}
\label{eqK^*} K^*_{i}=\beta\chi\cdot\sum
_{j=1}^{i} \indic_{Q_{\wt{M}}(x(j))}.
\end{equation}
Observe that for each $i$, $K^*_{i}$ is a summation of \emph{moving
windows} and is bounded by $\beta\kappa$ everywhere [recall that
$\kappa
$ is defined in \eqref{eqepsilon3} and note each point is covered by
at most $\wt{M}^2$ cubes of the form $Q_{\wt{M}}(x_j)$ for centers in
our 2-dimensional grid].

We start with site $x(1) = 0$. The $\tau_0$, $\mu_{0}$ and $\nu_{0}$
are all deterministic: $\tau_0=0$, $\mu_{0}=\mu$ and $\nu_{0}=0$. Let
$\tau_1 =\tau(U^1;\chi,Q_{\wt{M}}(x(1)))$.
By Proposition \ref{propperc}, the good event $G^{1}$ occurs with
probability $\geq1-3\vep_0$.
Observe also that $\tau_1>0$ almost surely and $\tau_1 = T$ on $G^{1}$.
If the good event $G^1$ occurs, then we change the status of site $0$
to be occupied.
Further define
%
%
\begin{equation}
\label{eqextract1w}\qquad w_1 =\cases{\displaystyle\sum_{z\in\sA(x(1))}
\frac{|\mu|}{U^*_{\tau_1-}(Q(z))}\cdot U^*_{\tau_1-} \bigl(\cdot\cap{Q(z)} \bigr), &\quad$
\mbox{ if } G^1 \mbox{ occurs}$, \vspace*{2pt}
\cr
0,&\quad $ \mbox{
otherwise}.$}
\end{equation}

We now work with site $y=x(i)$ for $i\geq2$. We proceed according to
whether the site $y$ is an
immediate offspring of some occupied site or not.

\textit{Case} I. Site $y$ is an
immediate offspring of some occupied site. Define
%
%
\begin{equation}
\label{eqextractp} (\mu_{i-1}, \nu_{i-1}) =
\bigl(w_{i-1} \bigl( \cdot\cap{Q(y)} \bigr), w_{i-1} \bigl(
\cdot\cap{Q(y)^c} \bigr) \bigr).
\end{equation}
Then $\mu_{{i-1}}$ is a measure supported by $Q(y)$, of total mass
$|\mu
|= f_d(\theta)$,
and $(A(\vep_0/4), f_d(\theta), r_0(\vep_0/4))$-admissible.
Let
$\tau_i =\tau_{i-1} + \tau(U^{i};\chi,Q_{\wt{M}}(y))$.
By Pro\-position \ref{propperc} (with an apparent spatial translation),
the good event $G^{i}$ occurs with probability $\geq1-3\vep_0$.
Observe also that
$\tau_i - \tau_{i-1} = T$ on $G^{i}$.
If the good event $G^i$ occurs, then we change the status of site $y$
to occupied. Moreover, according to whether $G^i$ occurs or not, we
define $w_{i}$ as follows:
%
%
\begin{equation}
\label{eqextractpw}  w_{{i}} = \cases{\displaystyle \nu_{{i-1}} + \sum
_{z\in\wt{\sA}(y)} \frac{|\mu|}{U^*_{\tau
_i-}(Q(z))} \cdot U^*_{\tau_i-}
\bigl(\cdot\cap{Q(z)} \bigr), &\quad$ \mbox{if } G^i \mbox{ occurs}$,
\vspace*{2pt}
\cr
\nu_{i-1}, &\quad $ \mbox{otherwise},$}\hspace*{-35pt}
\end{equation}
where
\[
\wt{\sA}(y)= \bigl\{z \in\sA(y)\dvtx z\notin\sA(u) \mbox{ for }u\mbox{ which is
occupied and } \prec y \bigr\}.
\]

\textit{Case} II. Site $y$ is not an
immediate offspring of any occupied site. Then we set $(\mu_{i-1}, \nu
_{i-1})=(0,w_{i-1})$,
$\tau_i=\tau_{i-1}$ and $w_{i}=w_{{i-1}}$.

In either case at time $\tau_i$ we proceed to site $x(i+1)$.

It is easy to see that such defined $\mu_i, \nu_i, K_i^*$ and $\tau_i$
satisfy the conditions of Proposition \ref
{propcompimmigrationseries}, and therefore the processes $U$ and
$U^*$ can be coupled such that
\[
L_t^U\geq L_t^{U^*} \qquad\mbox{for
all } t \geq0.
\]
Now if we let $\Omega$ be the set of all occupied sites, then by the
algorithm above, for any $x=x(i)\in\Omega$,
\[
L_\infty^{U^*} \bigl(Q(x) \bigr)\geq L_{\tau_i}^{U^*}
\bigl(Q(x) \bigr) - L_{\tau
_{i-1}}^{U^*} \bigl(Q(x) \bigr) >0,
\]
and hence $\Omega$ satisfies condition (i).\vadjust{\goodbreak}

We now show that $\Omega$ is infinite with positive probability. Define
a site percolation on $\Gamma$ as follows: for each $x\in\Gamma$, if
$x$ is occupied, then we let $\xi(x)=1$ if both
$y\in\mathcal{A}(x)$ are occupied, and $=0$ otherwise; if $x$ is
vacant, then we let $\xi(x)$ be a Bernoulli$(1-6\vep_0)$ random
variable that is independent of everything else.

We know that the origin is occupied with positive probability. We claim that
on the event that the origin is occupied, $\Omega$ contains the
collection of sites reachable from the origin.
We may assume that $\xi(0)=1$ since otherwise we are
done. But when 0 is occupied, $\xi(0)=1$ implies that both
$y\in\mathcal{A}(0)$ are occupied. By induction the conclusion
follows.

It remains to show that the above defined site percolation is a
3-dependent site percolation with density at least $(1-6\vep_0)$,
that is, we need to show that for any $n\geq1$ and any $1\leq
i_1<\cdots<i_n$
such that $\Vert x(i_j)-x(i_k)\Vert _1\geq3$,
\[
P \bigl(\xi \bigl(x(i_j) \bigr)=0\mbox{ for all } j =1,\ldots,n
\bigr) \leq(6\vep_0)^n.
\]
Since when a site $x$ is vacant, $\xi(x)$ is a Bernoulli$(1-6\vep_0)$
random variable independent of everything else, we
need only to show
%
%
\begin{equation}
\label{eq3depocc} P \bigl(\xi \bigl(x(i_j) \bigr)=0\mbox{ for all
} j =1,\ldots,n | \mbox{all } x(i_j) \mbox{'s} \mbox{
are occupied} \bigr)\leq(6\vep_0)^n.\hspace*{-35pt}
\end{equation}

Let us first consider the $n=1$ case. When $x:=x(i_1)$ is occupied, by
construction, each $y\in\mathcal{A}(x)$ is occupied with
probability at least $1-3\vep_0$, hence the probability that both
$y\in\mathcal{A}(x)$ are occupied is at least $1-6\vep_0$.
Equation~\eqref{eq3depocc} follows.

In general, for each $m\geq0$, we define $\sG_m$ to be the $\sigma
$-algebra generated by $\{U^{*}_t\dvtx 0 \leq t\leq \tau_m\}$. Then for
each $i\geq1$, the good event $G^{i}$ is measurable with respect to
$\sG_{i}$, and hence the Bernoulli random variable $\xi(x(i))$ is
measurable with respect to $\sG_{\ell}$ where $\ell$ is the index of
the second $y\in\sA(x(i))$. Now since $x(i_j)$'s are at least distance
3 from each other, if we let $\ell_j$ be the index of the second $y\in
\sA(x(i_j))$, then
\[
\ell_j < \ell_n - 2\qquad \mbox{for all } j<n.
\]
Hence by further conditioning on $\sG_{ \ell_{n}- 2}$, \eqref
{eq3depocc} reduces to the $n=1$ case and hence holds.
\end{pf}

\section{Proof of extinction when $d=2$ or $3$} \label{secext}

As the title suggests we shall assume $d=2$ or $3$ throughout this section.
\subsection{Scaled process}
%
\begin{prop}\label{propextscaled} Suppose $U$ is such that for each
$\psi\in C_c^2(\zz{R}^d)$,
%
%
\begin{eqnarray}
\label{eqnUscaled} U_t(\psi) &=& U_0( \psi) +
\frac{\alpha}{2}\int_0^t U_s(
\Delta\psi) \,ds +\vep\int_0^t
U_s(\psi) \,ds
\nonumber
\\[-8pt]
\\[-8pt]
\nonumber
&&{}-\beta\int_0^t
U_s \bigl( L^U_{t}\cdot\psi \bigr) \,ds +
\sqrt{\gamma} M_t(\psi),
\end{eqnarray}
where $M_t(\psi)$ is a martingale with quadratic
variation $[M(\psi)]_t=\int_0^t U_s(\psi^2) \,ds$. There exist
positive constants $\vep_0$\vadjust{\goodbreak} and $\zeta$ such that if the initial
condition $U_0$ belongs to the class
%
%
\begin{equation}
\mathcal{C}:= \bigl\{\mu\mbox{ satisfying 
{Assumption
\ref{asmtpinireg2}}, } 
\operatorname{ Supp}( \mu)\subseteq Q(0),\mbox{ and
} |\mu|= 2 \bigr\},\hspace*{-35pt}
\end{equation}
and the positive parameters $\alpha, \vep,\beta$ and $\gamma$
satisfy Assumption \ref{asmtpscaled} below, then
\[
P(U \mbox{ dies out}) =1.
\]
\end{prop}

\begin{asmptn}\label{asmtpscaled}
\[
\aligned\vep\leq\frac{\beta}{2\cdot3^d}, \qquad\max \biggl(\vep,\frac
{ \alpha
}{\gamma},
\frac{\sqrt\vep}{\gamma} \biggr)\leq\vep_0\quad\mbox{and} \quad\min \biggl(
\frac{\beta}{\vep^2}, \frac{\beta^2}{\vep^3\gamma}, \frac
{1}{\gamma}, \frac{\beta}{\gamma}
\biggr)\geq\zeta. \endaligned
\]
\end{asmptn}

We denote by $P_{\mu}^{\alpha,\vep,\beta,\gamma}$ the law of $U$
satisfying \eqref{eqnUscaled} with $U_0=\mu\in\mathcal{C}$. Then
we can rephrase the conclusion of Proposition \ref{propextscaled} as
\[
p^{\alpha, \vep,\beta,\gamma} = 0,
\]
where
%
%
\begin{equation}
\label{eqsupprobsurv} p^{\alpha,\vep,\beta,\gamma}:=\sup_{\mu\in\mathcal{C}}
P_{\mu}^{\alpha,\vep,\beta,\gamma}(U \mbox { survives}).
\end{equation}
When there is no confusion about the initial configuration $\mu$,
we omit $\mu$ and write $P^{\alpha,\vep,\beta,\gamma}$ and
sometimes just write $P$. Note that $P^{\alpha,\vep,0,\gamma}$
denotes the law of a Dawson--Watanabe process without any local
time killing, and $P^{\alpha,0,0,\gamma}$ the law of driftless
Dawson--Watanabe process. By (a scaled version of) Proposition \ref
{propBRWenvelope} we see
that when $\beta>0$,
%
%
\begin{equation}
\label{eqbrwenvelope} U^{\alpha,\vep,\beta,\gamma} \lesssim U^{\alpha,\vep,0,\gamma},
\end{equation}
where $U^{\alpha,\vep,\beta,\gamma}$ has law
$P_{\mu}^{\alpha,\vep,\beta,\gamma}$, $U^{\alpha,\vep,0,\gamma
}$ has
law $P_{\mu}^{\alpha,\vep,0,\gamma}$ and the above notation means we
can define versions of these processes on the same space with
$U_t^{\alpha,\vep,\beta,\gamma}\le
U_t^{\alpha,\vep,0,\gamma}$ for all $t\ge0$ almost surely. Furthermore,
by Lemma \ref{lemmaGirsanov}, the laws
$P_{\mu}^{\alpha,\vep,0,\gamma}$ and $P_{\mu}^{\alpha,0,0,\gamma
}$ are
related to each other via the likelihood ratio
%
%
\begin{equation}
\label{eqLRdriftstd} \frac{dP_{\mu}^{\alpha,\vep,0,\gamma
}}{dP_{\mu
}^{\alpha
,0,0,\gamma}}(U) \bigg\rrvert_{\mathcal{F}_t} =\exp \biggl(
\frac{\vep}{\sqrt{\gamma}}M_t(1)- \frac{\vep^2}{2\gamma}\int_0^t
|U_s| \,ds \biggr).
\end{equation}

We introduce the following notation:
\[
V_t=|U_t|,\qquad \tau_3=\inf \bigl\{t\dvtx
\operatorname{Supp} \bigl(L^U_t \bigr)\nsubseteq
Q_3(0) \bigr\},
\]
and for any continuous real valued process $X$ and any
$c\in\zz{R}$, we let $T_c(X)$ be the hitting time
\[
T_c(X)=\inf\{t\dvtx X_t=c\}.
\]
Finally, define $\tau$ to be the first time that $V_t$ hits 0 or
$4$ or that $U_t$ exits $Q_3(0)$, that is,
%
%
\begin{equation}
\label{eqtau} \tau=T_0(V)\wedge T_{4}(V)\wedge
\tau_3.\vadjust{\goodbreak}
\end{equation}

\begin{lemma}\label{lemmataufinite} $\tau<\infty$ almost
surely.\vspace*{-2pt}
\end{lemma}

\begin{prop}\label{propextbeatdiffuse} There exist constants $\vep
_0$ and $\zeta$ such that
if the parameters $\alpha, \vep,\beta$ and $\gamma$ satisfy
Assumption \ref{asmtpscaled}, then
%
%
\begin{equation}
\label{eqsubcond} \sup_{\mu\in\mathcal{C}} P_{\mu}^{\alpha,\vep,\beta,\gamma}(V_\tau>0)<
\frac{1}{2\cdot3^d}:=p_c.\vspace*{-2pt}
\end{equation}
\end{prop}

We will prove these results in the next subsection. Proposition
\ref{propextbeatdiffuse} is analogous to Lemma 2.3.1 in
\citet{MT94}. Once we have the proposition, we can prove
Proposition \ref{propextscaled} by constructing a sub-critical branching
process as in \citet{MT94}, or more directly as follows.\vspace*{-2pt}

\begin{pf*}{Proof of Proposition \ref{propextscaled}} Suppose that
the positive parameters $\alpha, \vep,\beta$ and $\gamma$ satisfy
the assumption of Proposition \ref{propextbeatdiffuse}. Let
\[
r_1:=\frac{\sup_{\mu\in\mathcal{C}} P_{\mu}^{\alpha,\vep,\beta
,\gamma
}(V_\tau
>0)}{p_c} <1.
\]
By the definition \eqref{eqsupprobsurv} of
$p^{\alpha,\vep,\beta,\gamma}$, we can find a $\mu\in\mathcal{C}$
such that
%
%
\begin{equation}
\label{eqsurvsupmu} P_{\mu}^{\alpha,\vep,\beta,\gamma}(U \mbox{ survives})\geq
\frac{1+r_1}{2}p^{\alpha,\vep,\beta
,\gamma}.
\end{equation}

Let $U_t$ satisfy \eqref{eqnUscaled} with $U_0=\mu$. For this
$U$, at time $\tau$, on the event that $V_\tau>0$, $U_\tau$ is
contained in $Q_3(0)$ with total mass no greater than $4$. We can
then decompose it into no more than $2\times3^d$ parts as
\[
U_{\tau} = \sum_{i=1}^\ell
U_\tau^i,\qquad \ell\leq2\times3^d,
\]
each of which has support contained in a unit cube, total mass
at most $2$ and satisfies
Assumption \ref{asmtpinireg2}.
To see this last property, the domination in \eqref{eqbrwenvelope},
the absolute continuity in \eqref{eqLRdriftstd},
and the finite propagation speed of the super-Brownian motion [see,
e.g., Theorem III.1.3 in \citet{PerkinsSFnotes}]
show that it suffices to prove that if $U$ is the super-Brownian motion
with law $P_\mu^{\alpha,0,0,\gamma}$, then
$U_{\tau}$ satisfies
Assumption~\ref{asmtpinireg2}
a.s. The last claim follows directly from Theorem III.3.4. in \citet
{PerkinsSFnotes}.

By the Markov property of the joint process
$(U,L^U)$ [see Theorem \ref{thmmexistuniq}(d)], Lemma \ref
{lemmasurvchar}, and
(a scaled version of) Lemma \ref{lemmacompK},
\[
P_\mu^{\alpha,\vep,\beta,\gamma}(U \mbox{ survives}) \leq E \bigl(
\indic_{(V_\tau>0)}\cdot P_{U_\tau}^{\alpha,\vep
,\beta,\gamma}(U \mbox{ survives}) \bigr).
\]
Here we are ``throwing away'' the killing due to $L^U_\tau$. By
Lemma \ref{lemmacompmu} and translation invariance, the
right-hand side is bounded above by
\begin{eqnarray*}
E \Biggl(\indic_{(V_\tau>0)}\cdot\sum_{i=1}^\ell
P_{U_\tau^i}^{\alpha
,\vep,\beta,\gamma}(U \mbox{ survives}) \Biggr)& \leq&
P(V_\tau>0)\cdot E \Biggl(\sum_{i=1}^\ell
p^{\alpha,\vep,\beta,\gamma} \Biggr) \\
&\leq& r_1 p^{\alpha,\vep
,\beta
,\gamma}.\vadjust{\goodbreak}
\end{eqnarray*}
Combining this with the previous inequality and
\eqref{eqsurvsupmu} we get
\[
\frac{1+r_1}{2}p^{\alpha,\vep,\beta,\gamma}\leq r_1 p^{\alpha
,\vep,\beta
,\gamma},
\]
hence $p^{\alpha,\vep,\beta,\gamma}=0$.
\end{pf*}

\subsection{\texorpdfstring{Proof of Lemma \protect\ref{lemmataufinite} and Proposition \protect\ref{propextbeatdiffuse}}
{Proof of Lemma 6.3 and Proposition 6.4}}
In the arguments below, $U$ is a process satisfying \eqref
{eqnUscaled} with a fixed
initial condition $\mu\in\mathcal{C}$. The bounds in Lemmas~\ref{lemmaprobexit}--\ref{lemmaextY} below hold for \emph{all
$\mu\in\mathcal{C}$}, and hence will lead to the uniform bound in Proposition
\ref{propextbeatdiffuse}.

First we note that $V_t=|U_t|$ satisfies the following SDE for some
Brownian motion $W$:
%
%
\begin{equation}
\label{eqV} dV_t=\vep V_t \,dt -\beta U_t
\bigl(L^U_t \bigr) \,dt + \sqrt{\gamma}
\sqrt{V_t} \,dW_t.
\end{equation}
By an integration by parts,
\[
\beta\int_0^t U_s
\bigl(L^U_s \bigr) \,ds=\frac{\beta}{2}\int
\bigl(L^U_t(x) \bigr)^2 \,dx.
\]
When $t\leq\tau_3$, by Cauchy--Schwarz, we get that
%
%
\begin{equation}
\label{eqkillinglb}\quad  \frac{\beta}{2}\int \bigl(L^U_t(x)
\bigr)^2 \,dx\geq\frac{\beta}{2}\cdot\frac
{1}{3^d} \biggl(
\int L^U_t(x) \,dx \biggr)^2 =
p_c \beta \biggl(\int_0^t
V_s \,ds \biggr)^2.
\end{equation}

We now prove Lemma \ref{lemmataufinite}.
\begin{pf*}{Proof of Lemma \ref{lemmataufinite}} Suppose
otherwise $P(\tau=\infty)>0$, in particular,
$P(\tau_3=\infty)> 0$. By \eqref{eqV} and \eqref{eqkillinglb}, on
the event
$\{\tau_3=\infty\}$,
%
%
\begin{eqnarray}
\label{Vbnd}V_t \leq2 +\vep\int_0^t
V_s \,ds -p_c\beta \biggl(\int_0^t
V_s \,ds \biggr)^2 + \sqrt{\gamma}\int
_0^t \sqrt{V_s}
\,dW_s
\nonumber
\\[-8pt]
\\[-8pt]
\eqntext{\mbox{for all } t\geq0.}
\end{eqnarray}
Define a sequence of stopping times $\{r_i\}$ by $r_0=0$ and for
$i\geq0$,
\[
r_{i+1}=\cases{d r_i+1, &\quad $\mbox{if } V_{r_i+1}
\leq2,$\vspace*{2pt}
\cr
\inf\{t\geq r_i + 1\dvtx V_t = 2
\}, & \quad$\mbox{otherwise.}$}
\]
%
\begin{claim}
For all $i$, $r_i<\infty$, almost surely.
\end{claim}
Suppose for some $i$, $r_i<r_{i+1}=\infty$. Then $V_t>
2$ for all $t\ge r_i+1$. Therefore \eqref{Vbnd} shows that on
$\{\tau_3=\infty\}$ the continuous martingale
$\sqrt\gamma\int_0^t\sqrt{ V_s} \,dW_s$ approaches $+\infty$ as
$t\to\infty$, an event of probability zero. This proves the claim.

For each $i$ Proposition \ref{propBRWenvelope} allows us to bound
$V_t$ above on $[r_i,r_i+1]$ by a Feller diffusion with drift
$\vep$ and initial value $2$ which does hit 0 in the next one unit
of time with probability $q>0$. This shows $P(V\mbox{ hits 0 on
}[r_i,r_i+1]|\mathcal{F}_{r_i})\ge q>0$, and we therefore conclude
that $V$ will hit 0 almost surely, again a contradiction to our
supposition.
\end{pf*}

Next we prove Proposition \ref{propextbeatdiffuse}. Define a
continuous random time change
$\eta: [0,\int_0^{T_0(V)}V_s \,ds ]\to[0,T_0(V)]$ by
%
%
\begin{equation}
\label{eqalphas} \eta_t =\inf \biggl\{r\dvtx \int
_0^r V_s \,ds = t \biggr\},
\end{equation}
and let $\widetilde{V}_t=V_{\eta_t}$. Then $\widetilde{V}_t$
satisfies
\[
\widetilde{V}_t = 2 + \vep t -\beta\int_0^{\eta_t}U_s
\bigl(L^U_s \bigr) \,ds +\sqrt{\gamma}
B_t \qquad\mbox{for }t\le\int_0^{T_0(V)}V_s
\,ds,
\]
where $B_t=\int_0^{\eta_t} \sqrt{V_s} \,dW_t$ for $t\le
\int_0^{T_0(V)}V_s \,ds$ and may be extended, if necessary, to a
standard Brownian motion. If
%
%
\begin{equation}
\label{eqY} Y_t=2+\vep t - p_c\beta t^2+
\sqrt{\gamma}B_t,
\end{equation}
then
by \eqref{eqkillinglb}
%
%
\begin{equation}
\label{Yubnd} \widetilde{V}_t\le Y_t \qquad\mbox{for }t\le\int
_0^{\tau_3\wedge
T_0(V)}V_s \,ds,
\end{equation}
since the upper bound on $t$ implies $\eta_t\le\tau_3$.

We want to bound $P(V_\tau>0)$ where $\tau$ is defined in
\eqref{eqtau}. Using the comparison above, noting that by Lemma
\ref{lemmataufinite} $\tau<\infty$ \mbox{almost surely}, we get that
\begin{eqnarray}
\label{eqVtauposdec}
\nonumber
&&P(V_\tau>0)
\nonumber\\
&&\qquad\leq P \bigl(\tau_3< T_0(V),
\tau_3<T_{4}(V) \bigr) + P \bigl(\tau=T_{4}(V)
\bigr)
\nonumber\\
&&\qquad\leq P \bigl(\tau_3<1/(4\vep) \bigr) + P \bigl(1/(4\vep)
\leq \tau_3\leq T_0(V) \bigr)
\nonumber
\\[-8pt]
\\[-8pt]
\nonumber
&&\qquad\quad{}+P \bigl(T_{4}
\bigl(V( \cdot\wedge\tau_3) \bigr)< T_{0}(V) \bigr)
\\
&&\qquad\leq P \bigl(\tau_3<1/(4\vep) \bigr) + P \bigl(T_1(V)
\leq1/(8\vep) \mbox{ and } T_0(V)\ge1/(4\vep) \bigr)
\nonumber\\
\nonumber
&&\qquad\quad{} + P \bigl(T_0(Y)> 1/(8\vep) \bigr) +P
\bigl(T_{4}(Y)<T_{0}(Y) \bigr),
\end{eqnarray}
where in the last line we used that
\[
P \bigl(T_1(V)>1/(8\vep), T_0(V)\geq1/(4\vep) \mbox{ and } \tau_3\geq1/(4\vep) \bigr) \leq P \bigl(T_0(Y)>
1/(8\vep) \bigr).
\]
This holds because $V_0=2$, and hence on the event on the left-hand side,
\[
\int_0^{\tau_3\wedge T_0(V)} V_s \,ds\ge\int
_0^{1/(8\vep)}V_s \,ds\ge1/(8\vep),
\]
which implies $\eta_{1/(8\vep)}\leq1/(8\vep)$, and by
\eqref{Yubnd}, for all $t\leq1/(8\vep)$, $Y_t\geq
\widetilde{V}_t=V_{\eta_t}>0$ [since $ T_0(V)\geq1/(4\vep)$].
Proposition \ref{propextbeatdiffuse} will be proved if we can
show that all four probabilities in \eqref{eqVtauposdec}
are small.

\begin{lemma}\label{lemmaprobexit} There exists a constant $C>0$ such
that
\[
P \bigl(\tau_3\leq1/(4\vep) \bigr)\leq C\sqrt{ \frac{\alpha
}{\gamma}}
\exp \biggl(\frac{\sqrt\vep}{8\gamma} \biggr) +2\sqrt\vep\exp \biggl(\frac{2\sqrt\vep}{\gamma}
\biggr).
\]
\end{lemma}
\begin{pf}
By the domination \eqref{eqbrwenvelope}, it suffices to show the
lemma for $P^{\alpha,\vep,0,\gamma}$, which is then analogous to
Lemma 2.1.9 in \citet{MT94} where the conclusion for the $d=1$ case
is proved. We give here a slightly simpler proof for all $d\le3$.

Following \citet{MT94} and using \eqref{eqLRdriftstd}, we get
that
%
%
\begin{eqnarray}
\label{masterbnd}
 &&P^{\alpha,\vep,0,\gamma} \bigl(\tau_3\leq 1/(4\vep )\bigr)\nonumber\hspace*{-30pt}\\
&&\quad\le E^{\alpha,0,0,\gamma} \biggl(\mathbf{1}_{\{\tau_3\leq
1/(4\vep)\wedge T_{\vep^{-1/2}}(V)\}} \cdot\exp
\biggl(\frac{\vep}{\sqrt{\gamma}}M_{\tau_3}(1) - \frac{\vep^2}{2\gamma}
\int_0^{\tau_3} V_s \,ds \biggr) \biggr)\hspace*{-30pt}\\
\nonumber
&&\qquad{}+P^{\alpha,\vep,0,\gamma}\bigl(T_{\vep^{-1/2}}(V)<1/(4\vep) \bigr)\hspace*{-30pt}\\
\nonumber
&&\quad\leq \sqrt{P^{\alpha,0,0,\gamma} \bigl(\tau_3\leq1/(4\vep)\bigr)}\hspace*{-30pt}\\
&&\qquad{}\times\!\sqrt{\!E^{\alpha,0,0,\gamma} \biggl(\!\exp\!
\biggl(\frac{2\vep}{\sqrt{\gamma}}M_{\tau_3}(1) - \frac{\vep^2}{\gamma}\!
\int_0^{\tau_3}\!V_s\,ds\! \biggr)\cdot\mathbf{1}_{\{\tau_3\leq1/(4\vep)
\wedge T_{\vep^{-1/2}}(V)\}} \!\biggr)\!}\hspace*{-30pt}\\
\nonumber
&&\qquad{}+P^{\alpha,\vep,0,\gamma}\bigl(T_{\vep^{-1/2}}(V)<\infty \bigr).\nonumber\hspace*{-30pt}
\end{eqnarray}

A scale function [see, e.g., Proposition VII.3.2 and Exercise
VII.3.20 in \citet{RevuzYor}] for $V$ when $\beta=0$ is given by
$s(x)=\gamma(1-\exp(-2\vep
x/\gamma))/(2\vep)$ and so
\begin{eqnarray}\label{Vhitbnd}
\qquad P^{\alpha,\vep,0,\gamma} \bigl(T_{\vep^{-1/2}}(V)<\infty \bigr)&=&
\frac
{s(2)-s(0)}{s(\vep^{-1/2})-s(0)}
\nonumber
=\frac{1-\exp(-4\vep/\gamma)}{1-\exp(-2\sqrt\vep/\gamma)}
\nonumber
\\[-8pt]
\\[-8pt]
\nonumber
&\le&\frac{4\vep/\gamma}{(2\sqrt\vep/\gamma)\exp
(-2\sqrt
\vep/\gamma)} =2\sqrt\vep\exp(2\sqrt\vep/\gamma).
\end{eqnarray}

We will use Theorem 1 of \citet{Iscoe88} to bound
$P^{\alpha,0,0,\gamma}(\tau_3\leq1/(4\vep))\leq
P^{\alpha,0,0,\gamma}(\tau_3<\infty)$. To do so, we make another
scaling: let
\[
\widetilde{U}_t(\psi)=\frac{U_t (\psi(\sqrt{{2}/{\alpha}} x ) )}{\gamma}\qquad \mbox{for all }\psi\in
C_c^2 \bigl(\zz{R}^d \bigr).
\]
Then by Lemma \ref{lemmascaling}, $\widetilde{U}$ satisfies the
assumptions of Theorem 1 in
\citet{Iscoe88}, and
\[
U_t \bigl(Q_3^c(0) \bigr)>0
\quad\Longleftrightarrow\quad\widetilde{U}_t \bigl(Q_{3\sqrt{{2}/{\alpha}}}^c(0)
\bigr) >0.
\]
Hence by Theorem 1 in \citet{Iscoe88} and the fact that $U_0\in
\mathcal{C}$,
%
%
\begin{equation}
\label{eqprobdiffuse} P^{\alpha,0,0,\gamma}(\tau_3< \infty) \leq
\frac{\widetilde{U}_0 (u ( (({3}/{2})\sqrt{{2}/{\alpha}} )^{-1}\cdot) ) } {
(({3}/{2})\sqrt{{2}/{\alpha}} )^2} \leq Cu \biggl(\frac{2}{3}e_1 \biggr) \cdot
\frac{ \alpha}{\gamma},
\end{equation}
where $e_1$ is a unit vector, and $u(x)$ is the unique positive
(radial) solution of the
singular elliptic boundary value problem
\[
\Delta u(x)=u^2(x), \quad x\in B(0,1) \quad\mbox{and}\quad u(x)\to\infty\qquad\mbox{as
} |x|\to1.
\]

Next denote by $\lambda=\frac{2\vep}{\sqrt{\gamma}}$. Since
$Z(t):=\exp
(\lambda M_t(1)-\frac{\lambda^2}{2}\int_0^tV_s \,ds )$ is a
supermartingale (being a nonnegative local martingale),
%
%
\begin{eqnarray}\label{RN2bnd}
&&E^{\alpha,0,0,\gamma} \biggl(\exp \biggl(\frac{2\vep}{\sqrt {\gamma
}}M_{\tau_3}(1)
- \frac{\vep^2}{\gamma}\int_0^{\tau_3}
V_s \,ds \biggr)\cdot\mathbf{1}_{\{\tau_3\leq1/(4\vep)\wedge
T_{\vep
^{-1/2}}(V)\}} \biggr)
\nonumber\hspace*{-30pt}\\
&&\qquad= E^{\alpha,0,0,\gamma} \biggl(Z \bigl(\tau_3\wedge(4
\vep)^{-1} \bigr)\exp \biggl(\frac{\lambda^2}{4}\!\int_0^{\tau_3}\!V_s
\,ds \biggr)\cdot\mathbf{1}_{\{\tau_3\leq1/(4\vep)\wedge T_{\vep
^{-1/2}}(V)\}}\! \biggr)\hspace*{-30pt}
\nonumber
\\[-8pt]
\\[-8pt]
\nonumber
&&\qquad
\le\exp \biggl(\frac{\lambda^2}{4}\cdot\frac{1}{4\vep
^{3/2}}
\biggr)E^{\alpha,0,0,\gamma}(Z_0)\hspace*{-30pt}
\\
 &&\qquad=\exp \biggl(\frac{\sqrt\vep}{4\gamma} \biggr),\nonumber\hspace*{-30pt}
\end{eqnarray}
where optional sampling is used in the next to last line.
Now insert \eqref{Vhitbnd}, \eqref{eqprobdiffuse} and \eqref
{RN2bnd} into
\eqref{masterbnd} to complete the proof.
\end{pf}

\begin{lemma}\label{lemmaprobext}There exists a constant $C>0$
such that
\[
P \bigl(T_1(V)\leq1/(8\vep) \mbox{ and } T_0(V)\geq1/(4
\vep) \bigr)\leq C\frac{\vep}{\gamma}.
\]
\end{lemma}
\begin{pf}
Recall that $V$ satisfies \eqref{eqV}.
Applying Proposition \ref{propBRWenvelope} again, on
$\{T_1(V)<\infty\}$ we may define an
$\mathcal{F}_{T_1(V)+t}$-adapted solution $\ol{V}$ of
\[
\ol{V}_t=1 + \vep\int_0^t
\ol{V}_s \,ds + \sqrt{\gamma}\int_0^t
\sqrt{\ol{V}_s} \,dW'_s,
\]
where $W'$ is an $\mathcal{F}_{T_1(V)+t}$-Brownian motion and
$\ol{V}_t\geq V_{T_1(V) + t}$ for all $t\geq0$, \mbox{almost surely} on
$\{T_1(V)<\infty\}$. Therefore
\[
P \bigl(T_1(V)\leq1/(8\vep) \mbox{ and } T_0(V)\geq1/(4
\vep) \bigr) \leq P \bigl(T_0(\ol{V})\geq1/(8\vep) \bigr).
\]
By Exercise \mbox{II.5.3.} in \citet{PerkinsSFnotes} the last term
equals
\[
1-\exp \biggl(\frac{-2\vep}{\gamma(1-e^{-1/8})} \biggr) \leq\frac
{2\vep
}{\gamma(1-e^{-1/8})}.
\]
\upqed\end{pf}

Recall that $Y$ is defined in \eqref{eqY}.
%
\begin{lemma}\label{eqextprobY}
There exist constants $C_1,C_2>0$ such that for all
$\beta\vep^{-2}\geq20\mbox{,}000$ and $0<\vep\le 1/4$,
\[
P \bigl(T_0(Y)>1/(8\vep) \bigr)\leq C_1\exp
\biggl(-C_2\frac{\beta^2}{\vep^{3}\gamma} \biggr).
\]
\end{lemma}
\begin{pf} Assume $\beta,\vep$ are as above and recall that
$p_c=1/(2\cdot3^{d})$.
\begin{eqnarray*}
P \bigl(T_0(Y)>1/(8\vep) \bigr)&\leq& P(Y_{1/(8\vep)} >0 )
\\
&=&P \biggl(B_{1/(8\vep)}\geq\frac{1}{\sqrt{\gamma}} \biggl(\frac
{p_c}{64}
\cdot\beta\vep^{-2} - 2-\frac{1}{8} \biggr) \biggr)
\\
&\leq &P \biggl(B_1\geq\frac{\sqrt{8\vep} p_c}{64}\frac{1}{\sqrt {\gamma
}} \bigl(
\beta\vep^{-2} - 10\mbox{,}000 \bigr) \biggr)
\\
&\le& P \biggl(B_1\geq\frac{\sqrt{8\vep} p_c}{128}\frac{1}{\sqrt {\gamma
}} \cdot
\bigl(\beta\vep^{-2} \bigr) \biggr).
\end{eqnarray*}
The result follows.
\end{pf}

\begin{lemma}\label{lemmaextY} There exists $C>0$ such that if $
\vep\leq\min(1/2, p_c \beta)$, then
\[
P \bigl(T_{4}(Y)<T_{0}(Y) \bigr) \leq C\exp \biggl(-
\frac{1}{8\gamma} \biggr) + \exp \biggl(-2p_c\frac{\beta
}{\gamma}
\biggr).
\]
\end{lemma}
\begin{pf}
Recall that $Y$ satisfies
\[
Y_t=2 + \vep t -p_c \beta t^2 + \sqrt{
\gamma} B_t.
\]
Hence if we define $\widetilde{Y}_t$ by
\[
\widetilde{Y}_t= 2 + \vep t + \sqrt{\gamma} B_t,
\]
then $Y_t\leq\widetilde{Y}_t$. Note also that
\[
P \bigl(T_{4}(Y)<T_{0}(Y) \bigr)\leq P
\bigl(T_{3}(\widetilde{Y}) \leq1 \bigr) + P \bigl(T_{3}(
\widetilde{Y}) \geq1, T_{4}(Y)<T_{0}(Y) \bigr).
\]
We first estimate $P(T_{3}(\widetilde{Y}) \leq1)$ as
%
%
\begin{eqnarray}
\label{eqtYexplore} P \bigl(T_{3}(\widetilde{Y}) \leq1 \bigr) &=&P
\Bigl(\max_{t\leq1}(\vep t + \sqrt{\gamma} B_t)\geq1
\Bigr)
\nonumber\\
&\leq& P \biggl(\max_{t\leq1} B_t \geq
\frac{1-\vep}{\sqrt{\gamma}} \biggr)
\\
&\leq& C\exp \biggl(-\frac{1}{8\gamma} \biggr),\nonumber
\end{eqnarray}
provided that $\vep\leq1/2$, where $C>0$ is some constant
independent of $\vep$ and~$\gamma$.

We now work with $P(T_{3}(\widetilde{Y}) \geq1,
T_{4}(Y)<T_{0}(Y))$. Define
\[
\ol{Y}_t = Y_1 - p_c\beta t +\sqrt{\gamma}
\ol{B}_t,
\]
where $\ol{B}_t = B_{t+1}-B_1$. If $\vep\le p_c\beta$ and $t\ge1$,
then
\begin{eqnarray*}
Y_t &=&Y_1+\vep(t-1)-p_c\beta
\bigl(t^2-1 \bigr)+\sqrt\gamma\ol{B}_{t-1}
\\
&=&Y_1+(t-1) \bigl[\vep-p_c\beta(t+1) \bigr]+\sqrt\gamma
\ol{B}_{t-1}
\\
&\le &Y_1+(t-1)[p_c\beta-2p_c\beta]+\sqrt
\gamma\ol{B}_{t-1}
\\
&=&\ol{Y}_{t-1}.
\end{eqnarray*}
Furthermore, since $Y_t\leq\widetilde{Y}_t$, on the event
$\{T_{3}(\widetilde{Y}) \geq1\}$, $Y_1\leq3$ and $T_4(Y)>1$. Therefore
\begin{eqnarray*}
P \bigl(T_{3}(\widetilde{Y}) \geq1, T_{4}(Y)<T_{0}(Y)
\bigr) &\leq P (3- p_c\beta t +\sqrt{\gamma} \ol{B}_t
\mbox{ hits } 4 \mbox{ before } 0 ).
\end{eqnarray*}
The latter probability can be explicitly calculated using scale
functions: if we let
\[
s(x)=\int_0^x \exp \biggl(\int
_0^y \frac{2p_c\beta}{\gamma} \,dz \biggr) \,dy =
\frac{\gamma}{2p_c\beta} \biggl(\exp \biggl(\frac{2p_c\beta
}{\gamma} x \biggr) -1 \biggr),
\]
then
\begin{eqnarray*}
P \biggl(3- \frac{p_c\beta}{2}\cdot t +\sqrt{\gamma} \ol{B}_t
\mbox{ hits } 4 \mbox{ before } 0 \biggr) &=&\frac{s(3)-s(0)}{s(4)-s(0)}=
\frac{\exp(({2p_c\beta}/{\gamma})
\cdot3 ) - 1} {
\exp(({2p_c\beta}/{\gamma})\cdot4 )-1}
\\
&\leq& \exp \biggl(-2p_c\frac{\beta}{\gamma} \biggr).
\end{eqnarray*}
\upqed\end{pf}

\begin{pf*}{Proof of Proposition \ref{propextbeatdiffuse}} The hypotheses
of the above four lemmas are satisfied under
Assumption \ref{asmtpscaled} for small enough $\vep_0$ and large
enough $\zeta$. The bounds obtained in all four lemmas can also be
made as small as we like, again by taking $\vep_0$ small enough
and $\zeta$ large enough. By inserting
these bounds into \eqref{eqVtauposdec}, we obtain Proposition \ref
{propextbeatdiffuse}.
\end{pf*}

\subsection{Proof of extinction for the original equation}
By Proposition \ref{propthetaccommon} and Proposition \ref
{propextscaled},
in order to show extinction for $X$ defined by the original equation
\eqref{eqDWmgSIR}, it suffices to show that when $\theta>0$ is
sufficiently small, there exists a scaling as in Lemma \ref
{lemmascaling} such that the
parameters in the scaled equation satisfy
Assumption \ref{asmtpscaled}. This is the content of the next
lemma.

\begin{lemma}\label{lemmascaleok} For any fixed constants
$0<\vep_0<\zeta$, for all $\theta>0$ sufficiently small, there exist a
scaling of $X$, as in Lemma \ref{lemmascaling} with $K=0$, such that the
parameters in the scaled equation
satisfy Assumption \ref{asmtpscaled}.
\end{lemma}
\begin{pf} By Lemma \ref{lemmascaling}
we want to find positive constants $a, b$ and $c$ such that
%
%
\begin{equation}
\label{eqparrel} \alpha=ab^2, \qquad\vep=a\theta,\qquad \beta=
\frac{a^2b^d}{c}\quad \mbox{and}\quad\gamma= ac
\end{equation}
satisfy Assumption \ref{asmtpscaled}. We will only look at power
functions, that is,
\[
a=\theta^x,\qquad b= \theta^y \quad\mbox{and}\quad c=
\theta^z,
\]
and show that for appropriate (real) choices of $x, y$ and $z$,
Assumption \ref{asmtpscaled} is satisfied provided that $\theta$
is sufficiently small. We have that
\[
\alpha= \theta^{x + 2y},\qquad \vep=\theta^{1+ x},\qquad \beta=
\theta^{2x + dy
- z}\quad \mbox{and}\quad\gamma= \theta^{x+z}.
\]
Looking back at the conditions in Assumption \ref{asmtpscaled},
we see that it is sufficient that
\[
\cases{1+x> 2x + dy-z,\vspace*{2pt}
\cr
1 + x >0,\qquad x+2y-(x+z)>0, \qquad (1+x)/2 - (x +
z)>0,\vspace*{2pt}
\cr
(2x+dy-z)-2(1+x)<0, \qquad 2(2x+dy-z)-3(1+x)-(x+z)<0,
\vspace*{2pt}
\cr
x+z<0, \qquad 2x+dy-z - (x+z)<0, }
\]
that is,
\[
\cases{1+z> x + dy,\vspace*{2pt}
\cr
x >-1, \qquad 2y>z,\qquad 1-x>2z,\vspace*{2pt}
\cr
dy-z<2,\qquad 2\,dy<3z+3,\vspace*{2pt}
\cr
x+z<0, \qquad x+dy<2z. }
\]
There is an abundance of choices, for example, $x= -3/4$ and $y=z=1/2$
will do.
\end{pf}

\section{A strong form of local extinction}\label{seclocalExtinction}

Theorem \ref{thmmlocext} is a direct consequence of
Proposition \ref{proposition1DExtinction} and the following result.

\begin{thmm}\label{theoremlocalExtinction}
Assume that $d=2$ or $d=3$. If the initial mass distribution $\mu$ satisfies
Assumption \ref{asmtpinireg1}, then for any value of $\theta$ the
epidemic $X$ [the
solution to the martingale problem (1.1)] dies out locally, that is,
with probability one, for every compact subset $K\subset\zz{R}^{d}$,
%
%
\begin{equation}
\label{eqlocalExtinction} X_t(K)=0 \qquad\mbox{for large enough }t.
\end{equation}
\end{thmm}

The remainder of this section will be devoted to the proof of this
theorem. Observe at the outset that it suffices to show that the
property \eqref{eqlocalExtinction} holds when $K$ is a ball of radius
$\varrho=\varrho(\theta) >0 $ centered at a point with rational
coordinates, because any compact $K$ is covered by finitely many such
balls. Moreover, it suffices to consider only balls centered at the
origin, because the initial mass distribution $\mu$ can always be
re-centered. Thus, our objective is to prove that the epidemic dies
out in $K=B_{\varrho} (0)$.

\subsection{Re-infection at large times}\label{ssecconfinement}

The proof of Theorem \ref{theoremlocalExtinction} will have three
parts: first, we will show that \eqref{eqlocalExtinction} could fail
only if the ball $B_{\varrho} (0)$ were re-infected from outside the ball
$B_{3\varrho} (0)$ at indefinitely large times. Second, we will show (in
Section~\ref{ssecflux} below) that boundedness
of $EL_{\infty} (B_{3\varrho} (0))$, by Proposition~\ref
{propLocextL1}, implies that the mean mass flux through the sphere of radius
$2 \varrho$ is finite. Finally, we will show (in Section~\ref{ssecoccTimeBlowup}) the
finite total mean mass flux through the sphere of radius $2 \varrho$ will
imply that reinfection of $B_{\varrho} (0)$ from outside $B_{3\varrho}
(0)$ at arbitrarily
large times cannot occur.

To give precise meaning to the notions of ``re-infection from
outside'' and ``mass flux through a boundary'' we must bring in the
\emph{historical process} $H$ associated with the spatial epidemic
$X$. [For a rigorous development of the basic theory, for
Dawson--Watanabe processes without interaction, see \citet{DP91},
for interactive processes including our setting, see \citet{P95}
and for an overview of both, see
\citet{PerkinsSFnotes}.] Recall that for each time
$t$ the state $H_{t}$ is a random measure on the space of continuous
paths $C ([0,t],\zz{R}^{d})$ that projects to $X_{t}$ via the time-$t$
evaluation mapping. As in the above references, for $w\in\mathcal
{C}:=C([0,\infty),\zz{R}^d)$ we
set $w^t(\cdot)=w(\cdot\wedge t)$, and identify $C([0,t],\zz{R}^d)$
with $\{w\in C([0,\infty),\zz{R}^d)\dvtx w=w^t\}$.

Theorem 5.11(a) of \citet{P95} gives a version of Dawson's Girsanov
theorem for historical processes. It is then easy to adapt the proof of
Theorem \ref{thmmexistuniq} to see that Theorem 5.11(a) of \citet{P95}
will apply with the drift function $g$ there equal to $\theta-L^X_s(w_s)$.
This gives a solution $H_t$ to a well-posed historical martingale
problem so that
$X_t(\phi)=\int\phi(w_t)H_t(dw)$ is the unique solution to \eqref
{eqDWmgSIR}. It also shows that the law of $H$ is absolutely
continuous to the law of the historical process associated with
super-Brownian motion on the filtration up to time $t$, for each $t>0$.

For a fixed $\varrho>0$ let
\[
\eta(w)=\eta_\varrho(w)=\inf\bigl\{t\ge0\dvtx |w_t|\ge3\varrho\bigr\}
\]
be the exit time of the path $w$ from the interior of $B_{3\varrho
}(0)$. At time $t$ color the path $(w_s)_{s\le t}$ red if $\eta\le t$,
and otherwise color it blue. This gives a decomposition,
%
%
\begin{equation}
\label{RBdecomp} H_t(\cdot)=H^R_t(
\cdot)+H^B_t(\cdot):= H_t \bigl(\cdot\cap\{
\eta \le t\} \bigr)+H_t \bigl(\cdot\cap\{\eta>t\} \bigr).
\end{equation}
Projecting via the time-$t$
evaluation, we obtain the decomposition
\[
\label{eqredBlueDecomp}
X_{t}(\cdot)=X^{R}_{t}(
\cdot)+X^{B}_{t}(\cdot):= H_t^R(w_t
\in\cdot)+H^B_t(w_t\in\cdot).
\]

\begin{prop}\label{propositionconfinedEpidemics}
For each value $\theta\in
\zz{R}$ there exists $\varrho=\varrho(\theta)>0$ such that for
any initial mass distribution
$\mu$ satisfying Assumption \ref{asmtpinireg1}, the process $H^{B}$
in the
red/blue decomposition \eqref{RBdecomp} will die out with
probability one.
\end{prop}

\begin{pf}
Arguing as in Proposition IV.1.4 of \citet{PerkinsSFnotes}, but using
historical processes, one can construct our historical epidemic process
$H$ and the historical process $\overline H$ for a drift-$\theta$
super-Brownian motion, $\ol{X}$, on a common probability space so that
$H_0=\overline H_0$ and $H_t\le\overline H_t$ for all $t\ge0$. We
decompose $\overline H=\overline{H}^R+\overline{ H}^B$ as in \eqref
{RBdecomp}, thus inducing a corresponding decomposition, $\ol{X}=\ol
{X}^R+\ol{X}^B$. Then $\ol{X}^B$ will be the drift-$\theta$
superprocess associated with Brownian motion killed when it exits the
interior of $B_{3\varrho}(0)$. Therefore if $Q_x$ is Wiener measure
starting at $x$, and $\eta_\varrho$ is also the corresponding exit time
for the Brownian path, then for $t>0$,
%
%
\begin{equation}
E \bigl(\bigl|\ol{X}^B_t\bigr| \bigr)=e^{\theta t}\int
Q_x(t<\eta_\varrho)\,d\mu(x)\le e^{\theta t}|
\mu|Q_0(t<\eta_\varrho).
\end{equation}
[A careful proof of this could use the appropriate version of
Proposition \ref{propositionfreezing}(c) below with $\eta_\varrho$ in
place of $\tau_k$ and $\overline H$ in place of $H$.]
For $\varrho>0$ sufficiently small (how
small will depend on $\theta$), this expectation decays exponentially
with $t$, by elementary estimates on the transition kernel for killed
Brownian motion. [In particular, $3\varrho>0$ must be small enough
that the first eigenvalue of $-\Delta/2$ with Dirichlet boundary
conditions on $\partial B_{3\varrho} (0)$ is strictly greater than
$\theta$.]

It remains to show that the exponential decay of $E|\ol{X}^{B}_{t}|$
implies that $\ol{X}^{B}$ dies out almost surely.
Let $Z$ denote a Feller branching process with drift $\theta$. For
$n\in
\zz{N}$,
the fact that the total mass process $|\ol{X}^B|$ is dominated by the
total mass process without killing
on $\partial B_{3\varrho}(0)$ implies
%
%
\begin{eqnarray}
\label{superext}
P \bigl(\bigl|\ol{X}^B_{n+1}\bigr|>0
\bigr) &\le &E \bigl(P_{\ol{X}^B_n} \bigl(\bigl|\ol{X}_1^B\bigr|>0
\bigr) \bigr)\le E \bigl(P_{|\ol{X}^B_n|}(Z_1>0) \bigr)\nonumber\\
& =&E \biggl[1-
\exp \biggl(\frac{-2\theta|\ol{X}^B_n|}{1-e^{-\theta}} \biggr) \biggr]
\\
&\le& C(\theta)E \bigl(\bigl|\ol{X}^B_n\bigr| \bigr),\nonumber
\end{eqnarray}
where
Exercise II.5.3
of \citet{PerkinsSFnotes} is used in the next to last line. The
exponential decay in the mean on the right-hand side now shows that
$\overline H^B$, and hence the smaller $H^B$, dies out a.s. by a
Borel--Cantelli argument.
\end{pf}

For future reference we state a time shifted version of the above. Let
$T>0$, define
%
%
\begin{equation}
\label{dfnsigmaT} \sigma_T=\inf\bigl\{t\ge T\dvtx |w_t|\ge3
\varrho\bigr\}
\end{equation}
and for $t\ge T$ set
\[
H^{B,T}_t(\cdot)=H_t \bigl(\cdot\cap\{
\sigma_T>t\} \bigr).
\]

\begin{prop}\label{propositionconfinedEpidemicsT}
For $\mu$, $\theta$ and $\varrho(\theta)$ as in Proposition \ref
{propositionconfinedEpidemics}, the process $H^{B,T}$ will die out
with probability one.\vadjust{\goodbreak}
\end{prop}

\begin{pf}
One proceeds just as above but conditional on the past up to $T$,
$\overline H_t, t\ge T$
will be the historical process associated with a drift-$\theta$
super-Brownian motion
starting at $\overline H_T=H_T$.
\end{pf}

\emph{Assume} for the remainder of the proof that $\varrho=\varrho
(\theta)>0$ is small enough that the
conclusions of Propositions \ref{propositionconfinedEpidemics} and
\ref
{propositionconfinedEpidemicsT} hold.
Then for any fixed $T\geq0$, all mass in the spatial epidemic $X_{t}$ will
eventually be descended from the mass in $X_{T}$ outside of
$B_{3\varrho}(0)$.
This obviously implies that
if local extinction
\eqref{eqlocalExtinction} fails for $K=B_{\varrho} (0)$ then the ball
$B_{\varrho} (0)$ must be re-infected by mass from outside
$B_{3\varrho}(0)$ at arbitrarily large times.

\subsection{Finite mass flux}\label{ssecflux}

We will control the re-infections of $B_\varrho(0)$ from outside
$B_{3\varrho}(0)$ by bounding the total ``mass flux'' (to be made
precise below) through $\partial B_{2\varrho}(0)$. For any continuous
path $w$ in $\zz{R}^{d}$ define
$\nu_{0}<\tau_{1}<\nu_{1}<\cdots$ to be the successive times of
passage between the spheres $\partial B_{3\varrho} (0)$ and $\partial
B_{2\varrho}
(0)$ [i.e., $\nu_{0}$ is the first hitting time of $B_{3\varrho} (0)$,
$\tau_{1}$ the first hitting time of $\partial B_{2\varrho} (0)$ after
$\nu_{0}$, and so on].
Now for each $k=1,2,\ldots$ define $H_t^{k}$ to be an
associated historical process in which historical mass frozen at time
$\tau_k$ is collected as $\tau_k$ occurs for $\tau_k<t$. For general
superprocesses these are the historical random measures constructed by
\citet{D91} (Theorem 1.5) using log Laplace equations. We will follow
Theorem 2.23 and Remark 2.25 of \citet{P95} which gives a recipe for
their construction and associated stochastic analysis, using historical
stochastic calculus, and does so in a more general interactive
framework which includes our spatial epidemic processes.

$C^2_b(\zz{R}^d)$ denotes the space of bounded continuous functions on
$\zz{R}^d$ with bounded continuous partials of order $2$ or less.

\begin{prop}\label{propositionfreezing}
For each $k\in\zz{N}$, there is a nondecreasing continuous $\sM(\sC
)$-valued process, $H^k$, and hence an associated random measure on
$[0,\infty)\times\sC$ (also denoted by $H^k$), satisfying $H^k_0=0$
and the following properties:
\begin{longlist}[(a)]
\item[(a)] $w=w^{\tau_k}, \tau_k(w)= t \mbox{ and so }w_t\in
\partial B_{2\varrho}(0), \mbox{for }H^k-a.a.\ (t,w)$ a.s.

\item[(b)] If $\psi$ is a bounded measurable function on $\sC$, then
with probability 1 for all $t\ge0$,
\begin{eqnarray*}
&&\int\psi \bigl(w^{\tau_k} \bigr)\indic_{(t>\tau_k)}H_t(dw)\\
&&\qquad=
\int_0^t\int\psi \bigl(w^{\tau_k}
\bigr) \indic_{(s>\tau_k)} \,dM^H(s,w)
\\
&&\qquad\quad{}+\int_0^t\int\psi \bigl(w^{\tau_k}
\bigr)\indic_{(s>\tau_k)} \bigl[\theta-L^X_s(w_s)
\bigr]H_s(dw)\,ds+H^k_t(\psi),
\end{eqnarray*}
where $M^H$ is the orthogonal martingale measure associated with $H$.

\item[(c)] If $X^k_t(\cdot)=\int_0^t\int\indic_{(w_{\tau_k}\in
\cdot
)}H^k(ds,dw)$ and $\phi\in C_b^2(\zz{R}^d)$, then with probability~1
and for all $t\ge0$,
\begin{eqnarray*}
&&\int\phi(w_t)\indic_{(t>\tau_k)}H_t(dw)\\
&&\qquad=\int
_0^t\int\phi(w_s)
\indic_{(s>\tau_k)} \,dM^H(s,w)
\\
&&\qquad\quad{}+\int_0^t\int\indic_{(s>\tau_k)} \biggl[
\frac{\Delta\phi}{2}(w_s)+\phi(w_s) \bigl(
\theta-L^X_s(w_s) \bigr)
\biggr]H_s(dw)\,ds
\\
&&\qquad\quad{}+X^k_t(\phi).
\end{eqnarray*}

\item[(d)] For any fixed $t\ge0$ and bounded Borel $\psi\dvtx  \sC\to
\zz
{R}$, if
\[
A_n(t,\psi)=\sum_{i=1}^\infty
\indic_{(i2^{-n}<t)}\int\psi \bigl(w^{\tau_k} \bigr) \indic
_{((i-1)2^{-n}\le\tau_k<i2^{-n})} H_{i 2^{-n}}(dw),
\]
then $A_n(t,\psi)\rightarrow H_t^k(\psi)$ in probability as $n\to
\infty
$. If $A_n$ also denotes the measure on $[0,\infty)\times\sC$
associated with $A_n(t+,\psi)$, there is a subsequence $\{n_j\}$ so
that $A_{n_j}|_{[0,T]\times\sC}\to H^k$ in $\sM([0,T]\times\sC)$ for
all $T>0$ a.s.
\end{longlist}
\end{prop}

\begin{pf} The above result is implicit in Remark 2.25 in \citet{P95}
and carried out for the total mass in Theorem 2.23 of the same
reference. We will sketch how the latter construction is easily
extended to the measure-valued process~$H^k$.

Let $\psi\ge0$ be a bounded Borel function on $\sC$ and in the setting
of Theorem 2.23 in \citet{P95}, set
%
%
\begin{equation}
\label{Cdfn} C(t,\omega,w)=\psi \bigl(w^{\tau_k} \bigr)\indic
_{(t>\tau_k(w))}.
\end{equation}
The above setting includes our historical epidemic process with the
function $\hat g$ on page 9
of this reference equal to $\theta-L^X_s(\omega,w_s)$ and the
integrator $Z^0$ on page 12 given by
\[
dZ^0(s,w)=dM^H(s,w)+\theta H_s(dw)
\,ds-L^X_s(w_s)H_s(dw)\,ds.
\]
Therefore for $\psi$ fixed, Theorem 2.23 in \citet{P95} implies (b) and
the first conclusion in (d) for some nondecreasing left-continuous
process $H_t^k(\psi)$ satisfying $H_0^k(\psi)=0$. To derive (c) from
(b) [with $\psi(w)=\phi(w_{\tau_k})$], we need to show that
\begin{eqnarray*}
\int \bigl(\phi(w_t)-\phi(w_{\tau_k}) \bigr)
\indic_{(t>\tau_k)}H_t(dw)&= &\int_0^t
\int \bigl(\phi(w_s)-\phi(w_{\tau_k}) \bigr)
\indic_{(s>\tau_k)} \,dZ^0(s,w)
\\
&&{}+\int_0^t\int\frac{\Delta\phi}{2}(w_s)
\indic_{(s>\tau_k)} H_s(dw)\,ds,
\end{eqnarray*}
and this follows easily from the historical stochastic calculus in
Chapter~2 of \citet{P95}.

Consider next the continuity of $H^k_t(\psi)$ in $t$ for $\psi\ge0$
bounded and Borel. By~(IV.48) of \citet{DM82}, it suffices to show that
if $T_n\downarrow T$ are bounded $(\mathcal{F}_t)$-stopping times, then
\[
\lim_{n\to\infty} E \bigl(H^k_{T_n}(
\psi)-H^k_T(\psi) \bigr)=0.
\]
Arguing as in (2.44) of \citet{P95}, this reduces to showing
%
%
\begin{equation}
\label{cont1} \lim_{n\to\infty}E \bigl(H_{T_n}(T\le
\tau_k<T_n) \bigr)=0
\end{equation}
and
%
%
\begin{equation}
\label{cont2} \lim_{n\to\infty}E \bigl(H_s(
\tau_k=T) \bigr)=0\qquad\mbox{for each }s>0.
\end{equation}
We consider only \eqref{cont1} as the proof of \eqref{cont2} will then
be clear. Using the weak continuity of $H$, one easily sees that
%
%
\begin{eqnarray}
\label{cont3} \limsup_{n\to\infty}E \bigl(H_{T_n}(T\le
\tau_k<T_n) \bigr)&\le& E \bigl(H_T\bigl(|w_T|=2
\varrho,\tau_k\le T\bigr) \bigr)
\nonumber
\\[-8pt]
\\[-8pt]
\nonumber
& \le& E \bigl(H_T\bigl(|w_T|=2
\varrho\bigr)\indic_{(0<T)} \bigr),
\end{eqnarray}
where we used $\tau_k(w)>0$.
Theorem III.5.1 of \citet{PerkinsSFnotes} and our absolute continuity
of $X$ with respect to super-Brownian motion show that
%
%
\begin{eqnarray}
\label{tauknonsconst}&& P \bigl(H_t\bigl(|w_t|=2\varrho\bigr)>0
\mbox{ for some }t>0 \bigr)
\nonumber
\\[-8pt]
\\[-8pt]
\nonumber
&&\qquad\le P \bigl(X_t \bigl(\partial
B_{2\varrho}(0) \bigr)>0\mbox{ for some }t>0 \bigr)=0.
\end{eqnarray}
This implies the right-hand side of \eqref{cont3} is zero, and so
\eqref
{cont1} is proved, thus giving the continuity of $H^k_t(\psi)$ for each
$\psi$ as above.

Next we construct $H^k$ as a measure-valued process. Choose a countable
determining class $\mathcal{D}$ of bounded continuous functions on
$\sC
$ containing the constant $1$. For each $\psi\in\mathcal{D}$ there
is a
subsequence $\{n_j\}$ so that
%
%
\begin{equation}
\label{psiconv} \sup_{t\le T}\bigl |A_{n_j}(t,
\psi)-H^k_t(\psi)\bigr|=0 \qquad\mbox{for all }T>0 \mbox{ a.s.}
\end{equation}
This holds by the first part of (d), monotonicity in $t$ and the a.s.
continuity of the limit.
By diagonalization we assume the same subsequence works for all $\psi
\in
\mathcal{D}$.
It is then easy to check that $A_{n_j}|_{[0,T]\times\sC}\wkc
H^k|_{[0,T]\times\sC}$ as finite measures on $[0,T]\times\sC$ for all
$T>0$. Formally we may use Jakubowski's theorem [Theorem II.4.1 of
\citet{PerkinsSFnotes}] and the fact that the required compact
containment condition follows easily from the modulus of continuity for
the historical paths of super-Brownian motion [Theorem III.1.3 of
\citet
{PerkinsSFnotes}] and the usual absolute continuity argument. Implicit
in the above notation is the fact that the limiting random measure
$H^k$ is related to the processes $H^k(\psi)$ constructed earlier by
\[
\int_0^t\int\psi(w)H^k(ds,dw)=H^k_t(
\psi)\qquad\mbox{for all }t\ge0 \mbox{ a.s.}
\]
This gives the existence of the required process $H^k$ satisfying
properties (b)--(d).

We have $ \tau_k(w)\le t$, $w=w(\cdot\wedge\tau_k)$, and so $w_t\in
\partial B_{2\varrho}(0)$ for $A_n(dt,dw)$-a.s., and taking weak limits
in $n$ we obtain (a) except with $\tau_k\le t$ $H^k$-a.s.
To see that $\tau_k=t$ $H^k$-a.s., it suffices to fix $t\ge
\varepsilon
>0$ and show $H^k((t-\varepsilon,t]\times\{\tau_k<t-\varepsilon\})=0$
a.s. This is easily derived from (b) with $\psi=\indic_{(\tau
_k<t-\varepsilon)}$ and a bit of historical stochastic calculus.
\end{pf}

We may repeat the above construction with minor changes for the
stopping times $\nu_k$ in place of $\tau_k$ and so obtain continuous
nondecreasing $\sM(\sC)$-valued processes $\{\hat H^k\dvtx k\in\zz{N}\}$
and their projections $\{\hat X^k\dvtx k\in\zz{N}\}$ which are $\sM(\zz
{R}^d)$-valued processes supported on $\partial B_{3\varrho}(0)$. We
identify $X^k$ and $\hat X^k$ with the corresponding random measure on
$[0,\infty)\times\zz{R}^d$.

For future reference we state a truncated version of Proposition \ref
{propositionfreezing}(c).
For $t\ge T>0$ define
\[
\hat X^{k,T}(\cdot)=\int_T^t\int
\indic_{(w_{\nu_k}\in\cdot)} \indic_{(T\le\tau_k)}\hat H^k(ds,dw).
\]

\begin{prop}\label{propositionfreezingT} If $T>0$ and $\phi\in
C_b^2(\zz{R}^d)$, then with probability 1 for all $t\ge T$,
%
%
\begin{eqnarray}
\label{tautrunc}&&\int\phi(w_t) \indic_{(T\le\tau_k<t)}
H_t(dw)
\nonumber\\[-2pt]
&&\qquad=\int_T^t\int\indic_{(T\le\tau_k<s)}
\phi(w_s)\,dM^H(s,w)
\nonumber
\\[-8pt]
\\[-8pt]
\nonumber
&&\qquad\quad+\int_T^t\int\indic_{(T\le\tau_k<s)}
\biggl[\frac{\Delta\phi
}{2}+\phi \bigl(\theta-L^X_s
\bigr) \biggr](w_s)H_s(dw)\,ds \\[-2pt]
&&\qquad\quad{}+ \bigl(X_t^k-X_T^k
\bigr) (\phi),\nonumber
\\[-2pt]
\label{nutrunc}&&\int\phi(w_t) \indic_{(T\le\tau_k, \nu
_k<t)}
H_t(dw)
\nonumber\\[-2pt]
&&\qquad=\int_T^t\int\indic_{(T\le\tau_k, \nu_k<s)}
\phi(w_s)\,dM^H(s,w)
\nonumber
\\[-8pt]
\\[-8pt]
\nonumber
&&\qquad\quad{} +\int_T^t\int
\indic_{(T\le\tau_k, \nu_k<s)} \biggl[\frac
{\Delta\phi}{2} +\phi \bigl(
\theta-L^X_s \bigr) \biggr](w_s)H_s(dw)
\,ds\\[-2pt]
&&\qquad\quad{} +\hat X_t^{k,T}(\phi).\nonumber
\end{eqnarray}
\end{prop}

\begin{pf} For \eqref{tautrunc} start with Proposition \ref
{propositionfreezing}(b) with $\psi(w^{\tau_k})=\break  \phi(w_{\tau
_k})\indic
_{(T\le\tau_k)}$, and then proceed as in the derivation of (c) above.
The fact that $\tau_k=s$ for $H^k$-a.a $(s,w)$ is used to get the form
of the final term. The proof of~\eqref{nutrunc} is similar.
\end{pf}

The total flux measure on $[0,\infty)\times\partial B_{2\varrho}(0)$
is $X^\tau=\sum_{k=1}^\infty X^k$, and similarly
we define $X^\nu=\sum_{k=1}^\infty\hat X^k$ on $[0,\infty)\times
B_{3\varrho}(0)$. At present these measures may be infinite.

As was already noted, our plan is to control the re-infections of
$B_\varrho(0)$ from outside $B_{3\varrho}(0)$ by bounding the total
flux, $|X^\tau|$, through $\partial B_{2\varrho}(0)$. We next bound
this flux in $L^1$ as a consequence
of Proposition \ref{propLocextL1} and Proposition \ref
{propositionfreezing}(c) above.

Color a path $w$ yellow at time $t$ if and only if $\tau_k<t\le\nu_k$,
for some $k\ge1$, that is, if and only if at time $t$ $w$ is engaged
in an excursion from $\partial B_{2\varrho}(0)$ to $\partial
B_{3\varrho
}(0)$. Let $H^Y_t$ be the restriction of $H_t$ to the yellow paths at
time $t$, that is,
$H^Y_t(A)=\int1_A(w) [\sum_{k=1}^\infty\indic_{(\tau_k<t\le\nu
_k)} ] H_t(dw)$,
and let $X^Y_t$ be the corresponding time-$t$ projection.

\begin{prop}\label{propositionfiniteflux} $E(|X^\tau|)<\infty$ and
$E(|X^\nu|)<\infty$.
\end{prop}
\begin{pf} We only prove the first conclusion as the proof of the
second is similar.

By differencing the decompositions in Proposition \ref
{propositionfreezing}(c) for times $\tau_k$ and $\nu_k$, we see that
for $\phi\in C_b^2(\zz{R}^d)$,
%
%
\begin{eqnarray}
\label{kdecomp} &&\int\phi(w_t) \indic_{(\tau_k<t\le\nu_k)}
\,dH_t(w)
\nonumber\\
&&\qquad=\int_0^t\int\phi(w_s)
\indic_{(\tau_k<s\le\nu_k)} \,dM^H(s,w)
\nonumber
\\[-8pt]
\\[-8pt]
\nonumber
&&\qquad\quad+\int_0^t\int\indic_{(\tau_k<s\le\nu_k)}
\biggl(\frac{\Delta
\phi}{2}(w_s)+\phi(w_s) \bigl(
\theta-L^X_s(w_s) \bigr)
\biggr)H_s(dw)\,ds\\
&&\qquad\quad{}+X^k_t(\phi)-\hat
X^k_t(\phi).\nonumber
\end{eqnarray}
Let $0\le\phi_0\le1$ be as above with support in the interior of
$B_{3\varrho}(0)$ and so that $\phi_0=1$ on $B_{2\varrho}(0)$. Then
$\hat X_t^k(\phi_0)=0$ and $X^k_t(\phi_0)=|X_t^k|$ for all $k$, $t$.
Take expectations in the above with $\phi=\phi_0$, and then sum over
$k$ to conclude that
\begin{eqnarray*}
E \bigl(X^Y_t(\phi_0) \bigr)&=&E \biggl(\int
_0^t\int\frac{\Delta\phi_0}{2}(w_s)+
\phi_0(w_s) \bigl[\theta-L^X_s(w_s)
\bigr] H^Y_s(dw)\,ds \biggr)\\
&&{}+E\bigl(X^\tau
\bigl([0,t]\times\sC \bigr)\bigr).
\end{eqnarray*}
Rearrange the above, and use $X^Y_t\le X_t$ and then \eqref
{eqDWmgSIR} to see that
\begin{eqnarray*}
&&E \bigl(X^\tau \bigl([0,t]\times\sC \bigr) \bigr)\\
&&\qquad\le E
\bigl(X_t(\phi_0) \bigr)+ E \biggl(\int
_0^tX_s \biggl(
\frac{|\Delta\phi_0|}{2}+\phi_0 \bigl(L^X_s+
\theta^- \bigr) \biggr) \,ds \biggr)
\\
&&\qquad\le\mu(\phi_0)+E \bigl( \bigl\langle L^X_t,|
\Delta\phi_0|+\phi_0 \theta^+ \bigr\rangle \bigr).
\end{eqnarray*}
The right-hand side remains bounded as $t\to\infty$ by
Proposition \ref
{propLocextL1}, and so the result follows.
\end{pf}

\subsection{Local extinction}\label{ssecoccTimeBlowup}

Recall that $\eta(w)=\eta_\varrho(w)$ is the exit time of $w$ from the
interior of $B_{3\varrho}(0)$. For any path $w\in\sC$, if $\eta\le
t$ and
$|w_t|<2\varrho$, then for some $k\ge1$, $\tau_k<t\le\nu_k$. That is,
if you exit from the interior of $B_{3\varrho}(0)$ before time $t$ and
at time $t$ are back in the interior of $B_{2\varrho}(0)$, then $t$
must fall in one of the excursions from $\partial B_{2\rho}(0)$ to
$\partial B_{3\rho}(0)$. Therefore if $\phi_1\in C^\infty_c(\zz{R}^d)$,
takes values in $[0,1]$, has support in $B_{(3/2)\varrho}(0)$, and
$\phi
_1=1$ on $B_\varrho(0)$, and $T>0$, then for all $t\ge T$,
\begin{eqnarray*}
\nonumber
X_t(\phi_1)&=&\int
\phi_1(w_t) \indic_{(t<\eta)} H_t(dw)+
\sum_{k=1}^\infty\int\phi_1(w_t)
\indic_{(\tau_k<t\le\nu_k)} H_t(dw)
\\
&=&X_t^B(\phi_1)+\sum
_{k=1}^\infty\int\phi_1(w_t)
\indic_{(\tau_k<T \le
t\le\nu_k)} H_t(dw)
\\
&&{}+\sum
_{k=1}^\infty\int\phi_1(w_t)
\indic_{(T\le\tau_k<t\le\nu_k)} H_t(dw)
\\
&:=& X_t^B(\phi_1)+{\hat X}^{Y,T}_t(
\phi_1)+X^{Y,T}_t(\phi_1).
\end{eqnarray*}
We have decomposed $X^{Y}$ according to whether or not the $k$th return
to $B_{2\varrho}(0)$ occurs before time $T$ or after it.

We have already shown (Proposition \ref{propositionconfinedEpidemics})
that $X^B$ dies out a.s. Recall the $\sigma_T$ defined in \eqref
{dfnsigmaT}. Clearly $\tau_k<T\le t\le\nu_k$ implies $\sigma_T>t$
for $H_t$-a.a. $w$ for all $t\ge T$ a.s. [recall \eqref
{tauknonsconst}], and so by Proposition \ref{propositionconfinedEpidemicsT},
\[
\hat X_t^{Y,T}(\phi_1)\le\int
\phi_1(w_t)H_t^{B,T}(dw)=0\qquad\mbox{for
large }t \mbox{ a.s.} \mbox{ for each $T>0$. }
\]
Therefore to complete the proof of Theorem \ref
{theoremlocalExtinction} it suffices to show the following:

\begin{prop}\label{propositionflux}$\lim_{T\to\infty
}P(X^{Y,T}_t(\phi
_1)>0 \mbox{ for some }t\ge T)=0$.
\end{prop}
To prove this result, we first recall a standard method used to compute
hitting probabilities for a super-Brownian motion $\overline X$ with
drift $\theta$. For $\lambda>0$ let $U(t,x)=U^\lambda(t,x)$ be the
unique nonnegative solution of
%
%
\begin{equation}
\label{dualeq} \frac{\partial U}{\partial t}=\frac{\Delta
}{2}U_t+\theta
U_t-U^2_t/2+\lambda\phi_1,\qquad
U_0\equiv0,
\end{equation}
which is bounded on $[0,T]\times\zz{R}^d$ for all $T$; for example,
see Theorem II.5.11(b) in \citet{PerkinsSFnotes}. The duality for
superprocesses [e.g., see Theorem II.5.11(c) in \citet{PerkinsSFnotes}]
implies that for all initial measures $\nu$,
\[
E_\nu \biggl(\exp \biggl(-\lambda\int_0^t
\overline X_s(\phi_1) \,ds \biggr) \biggr)=\exp \bigl(-
\nu \bigl(U_t^\lambda \bigr) \bigr).
\]
It follows that $U^\lambda(t,x)$ increases as $\lambda,t\to\infty$
to a
Borel function $U^\infty(x)\ge0$ satisfying
%
%
\begin{equation}
\label{superhitting} P_\nu \bigl(\overline X_s(
\phi_1)>0\mbox{ for some }s\ge0 \bigr)=1-\exp \bigl(-\nu
\bigl(U^\infty \bigr) \bigr).
\end{equation}
Next use the fact that $\overline X$ propagates locally at a finite
rate [see Theorem III.1.3 of \citet{PerkinsSFnotes}] and dies out in
small time with high probability if $|\overline X_0|$ is small [recall
\eqref{superext}], to see that for $\vep$ small,
\[
\sup_{|x|\ge2\varrho}P_{\vep\delta_x} \bigl(\overline
X_s( \phi_1)>0\mbox{ for some }s\ge0 \bigr)\le
\frac{1}{2}.
\]
It therefore follows from \eqref{superhitting} that
%
%
\begin{equation}
\label{Uinfbound} \sup_{|x|\ge2\varrho}U^\infty(x)=C_\varrho<
\infty.
\end{equation}

\begin{pf*}{Proof of Proposition \ref{propositionflux}}
Fix $T>0$. By differencing the decompositions in Proposition \ref
{propositionfreezingT}, we have for $\phi\in C_b^2(\zz{R}^d)$, with
probability 1 for all $t\ge T$,
%
%
\begin{eqnarray}
\label{kTdecomp} &&\int\phi(w_t) \indic_{(T\le\tau_k<t\le\nu_k)}
H_t(dw)
\nonumber\\
&&\qquad=\int_T^t\int\indic_{(T\le\tau_k<s\le\nu_k)}
\phi(w_s) \,dM^H(s,w)
\nonumber
\\[-8pt]
\\[-8pt]
\nonumber
&&\qquad\quad{} +\int_T^t\int
\indic_{(T\le\tau_k<s\le\nu_k)} \biggl[\frac
{\Delta\phi}{2}+\phi \bigl(\theta-L^X_s
\bigr) \biggr](w_s)H_s(dw)\,ds\\
&&\qquad\quad{}+ \bigl[X^k_t-X^k_T
\bigr](\phi)-\hat X_t^{k,T}(\phi).\nonumber
\end{eqnarray}
Fix $u>T$. Arguing as in Proposition II.5.7 of \citet{PerkinsSFnotes}
it is easy to extend~\eqref{kTdecomp} to time-dependent test functions
on $[0,u]\times\zz{R}^d$, including
$V(t,x)=U^\lambda(u-t,x)$; see also Theorem II.5.11(b) of \citet
{PerkinsSFnotes} for the regularity of the above $V$. One gets an
additional term involving $\frac{\partial V}{\partial t}$, and so with
the above choice of $V$, equation
\eqref{dualeq} shows that the function in the square brackets in the
second integral in \eqref{kTdecomp} becomes
\[
\frac{\partial V}{\partial s}+\frac{\Delta V}{2}+\theta V_s-L^X_sV_s=
\frac{V_s^2}{2}-\lambda\phi_1-L^X_sV_s.
\]
Therefore
for $T\le t\le u$,
\begin{eqnarray*}
&&\int V_t(w_t)\indic_{(T\le\tau_k<t\le\nu_k)} H_t(dw)
\\
&&\qquad=\int_T^tV_s(w_s)
\indic_{(T\le\tau_k<s\le\nu_k)}\, dM^H(s,w)\\
&&\quad\qquad{}+\int_T^t
\int\indic_{(T\le\tau_k<s\le\nu_k)}
 \biggl[\frac{V_s^2}{2}-\lambda\phi_1-L^X_sV_s
\biggr](w_s)H_s(dw)\,ds\\
&&\qquad\quad{}+\int_T^t
\int V_s(x) \bigl[X^k-\hat X^{k,T}
\bigr](ds,dx).
\end{eqnarray*}
Rearrange the above and sum over $k$ (using Proposition \ref
{propositionfiniteflux}) to see that
if $X^{\nu,T}=\sum_{k=1}^\infty\hat X^{k,T}_t$, and then
for $T\le t\le u$,
%
%
\begin{eqnarray}
\label{XVdecomp}&&X_t^{Y,T}(V_t)+\int
_T^t\lambda X^{Y,T}_s(
\phi_1)\,ds\nonumber
\\
&&\qquad=M_t^{Y,T}(V)+\int_T^tX^{Y,T}_s
\biggl(\frac
{V_s^2}{2}-L^X_sV_s
\biggr)\,ds\\
&&\qquad\quad{}+\int_T^t\int V_s(x)
\bigl[X^\tau(ds,dx)-X^{\nu,T}(ds,dx) \bigr],\nonumber
\end{eqnarray}
where $M^{Y,T}_t(V)$ is a continuous martingale starting at $0$ at time
$T$ and satisfying
$\langle M^{Y,T}(V)\rangle_t=\int_0^tX^{Y,T}_s(V_s^2)\,ds$.
Using Proposition \ref{propositionfiniteflux} we see the last term is
continuous in $t$, and it then follows easily that each of the terms in
\eqref{XVdecomp} is continuous.
Now apply It\^o's lemma
to $\exp(-X^{Y,T}_t(U^\lambda_{u-t})-\lambda\int_T^tX^{Y,T}_s(\phi
_1)\,ds )$,
and take expectations at $t=u$, where $V_u=U_0=0$
and note that $X_T^{Y,T}\equiv0$
to deduce that
\begin{eqnarray*}
&&E \biggl(1-\exp \biggl(-\lambda\int_T^u
X_s^{Y,T}(\phi_1)\,ds \biggr) \biggr)
\\
&&\qquad=E \biggl(\int_T^u\exp \biggl(-X^{Y,T}_t
\bigl(U^\lambda_{u-t} \bigr)-\lambda\int_T^tX^{Y,T}_s(
\phi_1)\,ds \biggr)\\
&&\qquad\qquad{}\times\int
U^\lambda_{u-t}(x) \bigl[X^\tau(dt,dx)-X^{\nu
,T}(dt,dx)-L^X_t(x)X^{Y, T}_t(dx)
\,dt \bigr] \biggr).
\end{eqnarray*}
Let $u, \lambda\uparrow\infty$, and drop the last two negative terms to
show that
%
%
\begin{eqnarray}
\label{XYsurv} P \biggl(\int_T^\infty
X^{Y,T}_s(\phi_1)\,ds>0 \biggr)&\le &E \biggl(
\int_T^\infty\int U^\infty(x)X^\tau(dt,dx)
\biggr)
\nonumber
\\[-8pt]
\\[-8pt]
\nonumber
&\le& C_\varrho E \bigl(X^\tau \bigl([T,\infty )
\times \zz{R}^d \bigr)\bigr).
\end{eqnarray}
Bound \eqref{Uinfbound} on $U^\infty$ for $|x|\ge2\varrho$ is used in
the last inequality. If we sum \eqref{kTdecomp} over $k$ we may argue
as in the analysis of \eqref{XVdecomp} to see that $X_t^{Y,T}(\phi_1)$
is continuous in $t$. This and the fact that the upper bound in \eqref
{XYsurv} approaches zero as $T\to\infty$ by Proposition \ref
{propositionfiniteflux} imply the required result.
\end{pf*}

\setcounter{equation}{0}

\begin{appendix}\label{app}
\section*{\texorpdfstring{Appendix: Proof of \lowercase{\protect\eqref{hypsmoothness}} for $\lowercase{d}=2,3$}
{Appendix: Proof of (2.17) for $d=2,3$}} The
main step is to show that for any fixed $t>0$,
\[
\lim_{\vep\to0} \limsup_N\hspace*{-1pt} \max
_{x\in\zz{Z}^d/\sqrt{N^\alpha\sigma^2}} \frac{ \sum_{i=1}^{[N^\alpha
\cdot|\mu|]} \indic_{|x-X_i|\leq\vep}  G_{[N^{\alpha}t]}(x\sqrt{N^\alpha\sigma^2}-[X_i\sqrt{N^\alpha
\sigma
^2}])}{{N^{\alpha(2-d/2)}}} = 0.
\]
The result would then follow easily by using the monotonicity in $t$,
the SLLN and the local central limit theorems.
Using inequality (19) in Lemma 2 of \citet{lz10} one can show that
\[
G_{[N^{\alpha}t]} (y) \leq C_1 N^{\alpha(1-d/2)} q^N
\bigl(y/\sqrt{N^{\alpha}\sigma^2} \bigr)\qquad \mbox{for all } N
\mbox{ and all } y\in\zz{Z}^d,
\]
where
$q^N(x)=\int_{1/(bN^{\alpha})}^{t/b} p_s(x)  \,ds$ for $x\in\zz{R}^d$,
and $b>0$ and $C_1=C_1(b)>0$ are both constants. Hence it suffices to show
\[
\lim_{\vep\to0} \limsup_N \max
_{x\in\zz{Z}^d/\sqrt{N^\alpha\sigma^2}} \frac{ \sum_{i=1}^{[N^\alpha
\cdot|\mu|]} \indic_{|x-X_i|\leq\vep} q^N(x-[X_i\sqrt{N^\alpha\sigma^2}]/\sqrt{N^{\alpha}\sigma
^2})}{N^{\alpha
}} =0.
\]
Let $h(r)=1/r$ when $d=3$, and $h(r)=\log(1/r)$ when $d=2$. Routine
calculations show that there exists a constant $C_2>0$ such that for
all $0\neq|x|$ small and for all $N$ sufficiently large,
%
%
\begin{equation}
\label{eqgbd} 1/C_2 h\bigl(|x|\bigr)\leq q^N(x) \leq
C_2 \bigl( h\bigl(|x|\bigr)\wedge h \bigl(N^{-\alpha/2} \bigr) \bigr).
\end{equation}
It follows that there exists $C_3>0$ such that for all $\vep>0$ small
enough, for all $N$ large enough, for all $ |z|\leq2\vep$,
%
%
\begin{equation}
\label{gslowvary} q^N(z) \leq C_3 q^N (v)\qquad
\mbox{for all } |v-z|\leq1/\sqrt{N^{\alpha
} \sigma^2}.
\end{equation}
Combining this with bound \eqref{eqrounddist}, we see that it
suffices to show
%
%
\begin{eqnarray}
\label{eqmaxfinite} \lim_{\vep\to0} \limsup_N
\sup_{x\in\zz{R}^d} Q_\vep^N(x)=0,
\nonumber
\\[-8pt]
\\[-8pt]
\eqntext{\mbox{where }
Q_\vep^N(x):= \frac{ \sum_{i=1}^{[N^\alpha\cdot
|\mu|]}
\indic_{|x-X_i|\leq\vep}  q^N(x-X_i)}{N^{\alpha}}.}
\end{eqnarray}

Next, for each $j=1,\ldots,[N^\alpha|\mu|]$, let
\[
\hat{Q}_\vep^N(j)=N^{-\alpha}\sum
_{i=1,i\neq j}^{[N^\alpha|\mu|]}\indic_{|X_i-X_j|\le\vep}h\bigl(|X_i-X_j|\bigr)
\wedge h \bigl(N^{-\alpha/2} \bigr).
\]

\begin{lemmas}\label{hatQ} There is a $C_4$ so that for all $\vep$ small
enough and all $N$ large enough,
\[
\sup_{x\in\bR^d}Q_\vep^N(x)\le
C_4 \biggl[\max_{j\le[N^\alpha|\mu|]}\hat Q_{2\vep}^N(j)+
\frac{h(N^{-\alpha/2})}{N^{\alpha}} \biggr].
\]
\end{lemmas}
\begin{pf} The upper bound in \eqref{eqgbd} shows that for $\vep$
small enough and $N$ large enough (which is assumed in the rest of this proof),
\[
Q^N_\vep(x)\le C_2N^{-\alpha}\sum
_{i=1}^{[N^\alpha|\mu|]}\indic_{|x-X_i|\le\vep}h\bigl(|x-X_i|\bigr)
\wedge h \bigl(N^{-\alpha/2} \bigr).
\]
Fix $x\in\bR^d$ and choose $j\in\{1,\dots,[N^\alpha|\mu|]\}$ so that
$|X_j-x|=\break  \min_{1\le i\le[N^\alpha|\mu|]}|X_i-x|$. Then
\[
|X_i-X_j|\le|X_i-x|+|x-X_j|
\le2|X_i-x|,
\]
and so $|x-X_i|\le\vep$ implies $|X_i-X_j|\le2\vep$, and therefore,
%
%
\begin{eqnarray}
\label{hatQbnd1}&&Q_\vep^N(x)
\nonumber
\\[-8pt]
\\[-8pt]
\nonumber
&&\qquad
\le C_2N^{-\alpha}
\sum_{i=1}^{[N^\alpha|\mu|]}\indic_{|X_i-X_j|\le2\vep}
\bigl[h\bigl(|X_i-X_j|/2\bigr)\wedge h \bigl(N^{-\alpha/2}
\bigr) \bigr].
\end{eqnarray}
We may assume $\vep\le1/4$. It follows that in the above summation
$h(|X_i-X_j|/2)\le2 h(|X_i-X_j|)$ for $d=2$ and this is obvious for
$d=3$. Therefore by~\eqref{hatQbnd1},
\[
Q^N_\vep(x)\le2C_2 \biggl[\hat
Q_{2\vep}^N(j)+\frac{h(N^{-\alpha
/2})}{N^{\alpha}} \biggr],
\]
where we have separated out the $i=j$ term in the summation on the
right-hand side of \eqref{hatQbnd1}. The result follows with $C_4=2C_2$
upon taking the max over $j$ on the right.
\end{pf}

Therefore to show \eqref{eqmaxfinite}, it suffices to establish
%
%
\begin{equation}
\label{eqhatQbnd2} \lim_{\vep\to0}\limsup_N
\max_{1\le j\le[N^\alpha|\mu|]}\hat Q_\vep^N(j)=0.
\end{equation}

Let $C_\mu f_d(r)$ denote the function arising on the right-hand side
of \eqref{conc}. Let $r_n=2^{-n}$, and define
\[
M_{n,j}=\#\bigl\{i\neq j\dvtx |X_i-X_j|\le
r_n\bigr\}.\vadjust{\goodbreak}
\]
If $K_N=[N^\alpha|\mu|]-1$ and
\[
p_n(x)=P\bigl(|X_1-x|\le r_n\bigr)\le C
f_d(r_n)
\]
[by \eqref{conc}], then conditional on $X_j$, $M_{n,j}$ is
$\operatorname{binomial}(K_N,p_n(X_j))$. Therefore a square function inequality for
martingales [see Theorem 21.1 in \citet{B73}] implies that for any $q>0$
there is a $C'_q$ so that
%
%
\begin{eqnarray}
\label{qmom} &&E \bigl(\bigl|M_{n,j}-K_Np_n(X_j)\bigr|^q
| X_j \bigr)\nonumber\\
&&\qquad\le C'_q \bigl(
\bigl(K_Np_n(X_j) \bigl(1-p_n(X_j)
\bigr) \bigr)^{q/2}+1 \bigr)
\\
&&\qquad\le C_q \bigl(N^{\alpha
q/2}f_d(r_n)^{q/2}+1
\bigr).\nonumber
\end{eqnarray}
Choose $n_0$ so that $r_{n_0}\le N^{-\alpha/2}<r_{n_0-1}$, and define
$\Lambda_N$ to be the complement of
\[
\bigcup_{n=1}^{n_0} \Bigl\{\max
_{j\le[N^\alpha|\mu
|]}\bigl|M_{n,j}-K_Np_n(X_j)\bigr|>N^\alpha
f_d(r_n) \Bigr\}.
\]

Use  \eqref{qmom} and Markov's inequality, and then $f_d(r)^{-1}\le
c(\log(1/r))^3r^{-1}$ for $r\in(0,1/2]$ to see that
\begin{eqnarray*}
P \bigl(\Lambda_N^c \bigr)&\le& c\sum
_{n=1}^{n_0}N^\alpha|\mu| \cdot
N^{-q\alpha
}f_d(r_n)^{-q}
\bigl(N^{\alpha q/2}f_d(r_n)^{q/2}+1 \bigr)
\\[-1pt]
&\le& cN^{\alpha(1-q)}\sum_{n=1}^{n_0}
\bigl(N^{\alpha
q/2}f_d(r_n)^{-q/2}+f_d(r_n)^{-q}
\bigr)
\\[-1pt]
&\le& c N^{\alpha(1-q)}\sum_{n=1}^{n_0}
\bigl(N^{\alpha
q/2}r_n^{-q/2} \bigl(
\log(1/r_n) \bigr)^{3q/2}+r_n^{-q}
\bigl(\log(1/r_n) \bigr)^{3q} \bigr).
\end{eqnarray*}
Recalling the choice of $n_0$ we can bound the above by
\begin{eqnarray*}
&&cN^{\alpha(1-q)} \bigl(\log(1/r_{n_0}) \bigr)^{3q}
\bigl(N^{\alpha
q/2}2^{n_0q/2}+2^{n_0q} \bigr) \\[-1pt]
&&\qquad\le
cN^{\alpha(1-q)}(\log N)^{3q} \bigl(N^{3\alpha q/4}+N^{\alpha
q/2}
\bigr)
\\[-1pt]
&&\qquad\le cN^{\alpha(1-(q/4))}(\log N)^{3q}.
\end{eqnarray*}
So choose $q$ large enough so that $\alpha(1-(q/4))<-2$ to conclude
that for all $N$ large enough,
%
%
\begin{equation}
\label{Plambdabnd} P \bigl(\Lambda_N^c \bigr)\le
CN^{-2}.
\end{equation}

Now for $\frac{1}{2}> \vep\ge N^{-\alpha/2}$ choose $n_1\in\{
2,\dots
,n_0\}$ so that $2^{-n_1}\le\vep<2^{1-n_1}$. On $\Lambda_N$, for
$1\le
j\le[N^\alpha|\mu|]$ we have [use $h(r)f_d(r)=C_\mu\log(1/r)^{-2}$]\vspace*{-1pt}
%
%
\begin{eqnarray}\label{hatQbound3}
\nonumber
\hat Q_\vep^N(j) &=&N^{-\alpha}\sum
_{i=1,i\neq j}^{[N^\alpha|\mu|]}\indic_{|X_i-X_j|\le
\vep}
\bigl(h\bigl(|X_i-X_j|\bigr)\wedge h \bigl(N^{-\alpha/2}
\bigr) \bigr)
\\[-1pt]
\nonumber
&\le& c \Biggl(\frac{M_{n_0,j}}{N^\alpha}h \bigl(N^{-\alpha/2} \bigr)+\sum
_{n=n_1-1}^{n_0-1}\frac{M_{n,j}}{N^\alpha}h(r_{n+1})
\Biggr)
\\[-1pt]
\nonumber
&\le& c \Biggl(\frac{K_N}{N^\alpha}p_{n_0}(X_j)h
\bigl(N^{-\alpha
/2} \bigr)+f_d(r_{n_0})h
\bigl(N^{-\alpha/2} \bigr)
\\[-1pt]
&&\hspace*{9pt}{} +\sum_{n=n_1-1}^{n_0-1} \biggl(
\frac{K_N}{N^\alpha
}p_n(X_j)h(r_n)+f_d(r_n)h(r_n)
\biggr) \Biggr)
\\[-1pt]
\nonumber
&\le& c \Biggl(f_d(r_{n_0})h(r_{n_0})+
\sum_{n=n_1-1}^{n_0-1}f_d(r_n)h(r_n)
\Biggr)
\\[-1pt]
\nonumber
&\le& c\sum_{n=n_1-1}^{n_0}(
\log1/r_n)^{-2}
\\[-1pt]
&\le& cn_1^{-1}\le c(\log1/
\vep)^{-1}.\nonumber
\end{eqnarray}
By \eqref{Plambdabnd}, the Borel--Cantelli lemma and \eqref
{hatQbound3}, we conclude that with probability~$1$ for all $N$ large
enough, we have
\[
\max_{1\le j\le[N^\alpha|\mu|]}\hat Q_\vep^N(j)\le c(\log1/
\vep)^{-1}\qquad\mbox{for }N^{-\alpha/2}\le\vep< 1/2.
\]
This implies \eqref{eqhatQbnd2}, and we are done.
\end{appendix}

\section*{Acknowledgments} We are grateful to the
Associate Editor and referee for their very careful reading of the
paper and constructive suggestions.

%
%



\printaddresses


\begin{thebibliography}{36}

\bibitem[\protect\citeauthoryear{Aldous}{1997}]{aldous1999}
\begin{barticle}[mr]
\bauthor{\bsnm{Aldous},~\bfnm{David}\binits{D.}}
(\byear{1997}).
\btitle{Brownian excursions, critical random graphs and the multiplicative
  coalescent}.
\bjournal{Ann. Probab.}
\bvolume{25}
\bpages{812--854}.
\bid{doi={10.1214/aop/1024404421}, issn={0091-1798}, mr={1434128}}
\bptnote{check year}%
\bptok{imsref}%
\end{barticle}
\endbibitem

\bibitem[\protect\citeauthoryear{Bailey}{1967}]{ba67}
\begin{bincollection}[auto:STB|2013/06/05|13:45:01]
\bauthor{\bsnm{Bailey},~\bfnm{N.~T.~J.}\binits{N.~T.~J.}}
(\byear{1967}).
\btitle{The simulation of stochastic epidemics in two dimensions}.
In \bbooktitle{Proceedings of the Fifth Berkeley Symposium on Mathematical
  Statistics and Probability ({U}niv. California, Berkeley, CA, 1967), Vol. IV:
  Probability Theory, Berkeley, CA}
\bpages{237--257}.
\bpublisher{Univ. California Press}, \blocation{Berkeley}.
\bptok{imsref}%
\end{bincollection}
\endbibitem

\bibitem[\protect\citeauthoryear{Barlow, Evans and Perkins}{1991}]{BEP91}
\begin{barticle}[mr]
\bauthor{\bsnm{Barlow},~\bfnm{Martin~T.}\binits{M.~T.}},
  \bauthor{\bsnm{Evans},~\bfnm{Steven~N.}\binits{S.~N.}} \AND
  \bauthor{\bsnm{Perkins},~\bfnm{Edwin~A.}\binits{E.~A.}}
(\byear{1991}).
\btitle{Collision local times and measure-valued processes}.
\bjournal{Canad. J. Math.}
\bvolume{43}
\bpages{897--938}.
\bid{doi={10.4153/CJM-1991-050-6}, issn={0008-414X}, mr={1138572}}
\bptok{imsref}%
\end{barticle}
\endbibitem

\bibitem[\protect\citeauthoryear{Bramson, Durrett and Swindle}{1989}]{BDS89}
\begin{barticle}[mr]
\bauthor{\bsnm{Bramson},~\bfnm{M.}\binits{M.}},
  \bauthor{\bsnm{Durrett},~\bfnm{R.}\binits{R.}} \AND
  \bauthor{\bsnm{Swindle},~\bfnm{G.}\binits{G.}}
(\byear{1989}).
\btitle{Statistical mechanics of crabgrass}.
\bjournal{Ann. Probab.}
\bvolume{17}
\bpages{444--481}.
\bid{issn={0091-1798}, mr={0985373}}
\bptok{imsref}%
\end{barticle}
\endbibitem

\bibitem[\protect\citeauthoryear{Burkholder}{1973}]{B73}
\begin{barticle}[mr]
\bauthor{\bsnm{Burkholder},~\bfnm{D.~L.}\binits{D.~L.}}
(\byear{1973}).
\btitle{Distribution function inequalities for martingales}.
\bjournal{Ann. Probab.}
\bvolume{1}
\bpages{19--42}.
\bid{mr={0365692}}
\bptok{imsref}%
\end{barticle}
\endbibitem

\bibitem[\protect\citeauthoryear{Cox and Durrett}{1988}]{cd88}
\begin{barticle}[mr]
\bauthor{\bsnm{Cox},~\bfnm{J.~T.}\binits{J.~T.}} \AND
  \bauthor{\bsnm{Durrett},~\bfnm{Richard}\binits{R.}}
(\byear{1988}).
\btitle{Limit theorems for the spread of epidemics and forest fires}.
\bjournal{Stochastic Process. Appl.}
\bvolume{30}
\bpages{171--191}.
\bid{doi={10.1016/0304-4149(88)90083-X}, issn={0304-4149}, mr={0978353}}
\bptok{imsref}%
\end{barticle}
\endbibitem

\bibitem[\protect\citeauthoryear{Daley and Gani}{1999}]{daley-gani1999}
\begin{bbook}[auto:STB|2013/06/05|13:45:01]
\bauthor{\bsnm{Daley},~\bfnm{D.~J.}\binits{D.~J.}} \AND
  \bauthor{\bsnm{Gani},~\bfnm{J.}\binits{J.}}
(\byear{1999}).
\btitle{Epidemic Modelling}.
\bpublisher{Cambridge Univ. Press}, \blocation{Cambridge}.
\bptok{imsref}%
\end{bbook}
\endbibitem

\bibitem[\protect\citeauthoryear{Dawson and Perkins}{1991}]{DP91}
\begin{barticle}[mr]
\bauthor{\bsnm{Dawson},~\bfnm{Donald~A.}\binits{D.~A.}} \AND
  \bauthor{\bsnm{Perkins},~\bfnm{Edwin~A.}\binits{E.~A.}}
(\byear{1991}).
\btitle{Historical processes}.
\bjournal{Mem. Amer. Math. Soc.}
\bvolume{93}
\bpages{iv+179}.
\bid{doi={10.1090/memo/0454}, issn={0065-9266}, mr={1079034}}
\bptok{imsref}%
\end{barticle}
\endbibitem

\bibitem[\protect\citeauthoryear{Dellacherie and Meyer}{1982}]{DM82}
\begin{bbook}[mr]
\bauthor{\bsnm{Dellacherie},~\bfnm{Claude}\binits{C.}} \AND
  \bauthor{\bsnm{Meyer},~\bfnm{Paul-Andr{\'e}}\binits{P.-A.}}
(\byear{1982}).
\btitle{Probabilities and Potential. {B}: Theory of Martingales}.
\bseries{North-Holland Mathematics Studies}
\bvolume{72}.
\bpublisher{North-Holland}, \blocation{Amsterdam}.
\bid{mr={0745449}}
\bptok{imsref}%
\end{bbook}
\endbibitem

\bibitem[\protect\citeauthoryear{Dolgoarshinnykh and
  Lalley}{2006}]{dolgoarshinnykh-lalley-2006}
\begin{barticle}[mr]
\bauthor{\bsnm{Dolgoarshinnykh},~\bfnm{R.~G.}\binits{R.~G.}} \AND
  \bauthor{\bsnm{Lalley},~\bfnm{Steven~P.}\binits{S.~P.}}
(\byear{2006}).
\btitle{Critical scaling for the {SIS} stochastic epidemic}.
\bjournal{J. Appl. Probab.}
\bvolume{43}
\bpages{892--898}.
\bid{doi={10.1239/jap/1158784956}, issn={0021-9002}, mr={2274810}}
\bptok{imsref}%
\end{barticle}
\endbibitem

\bibitem[\protect\citeauthoryear{Durrett}{1995}]{DurrettLect}
\begin{bincollection}[mr]
\bauthor{\bsnm{Durrett},~\bfnm{Rick}\binits{R.}}
(\byear{1995}).
\btitle{Ten lectures on particle systems}.
In \bbooktitle{Lectures on Probability Theory ({S}aint-{F}lour, 1993)}.
\bseries{Lecture Notes in Math.}
\bvolume{1608}
\bpages{97--201}.
\bpublisher{Springer}, \blocation{Berlin}.
\bid{doi={10.1007/BFb0095747}, mr={1383122}}
\bptok{imsref}%
\end{bincollection}
\endbibitem

\bibitem[\protect\citeauthoryear{Durrett and Perkins}{1999}]{DP99}
\begin{barticle}[mr]
\bauthor{\bsnm{Durrett},~\bfnm{Richard}\binits{R.}} \AND
  \bauthor{\bsnm{Perkins},~\bfnm{Edwin~A.}\binits{E.~A.}}
(\byear{1999}).
\btitle{Rescaled contact processes converge to super-{B}rownian motion in two
  or more dimensions}.
\bjournal{Probab. Theory Related Fields}
\bvolume{114}
\bpages{309--399}.
\bid{doi={10.1007/s004400050228}, issn={0178-8051}, mr={1705115}}
\bptok{imsref}%
\end{barticle}
\endbibitem

\bibitem[\protect\citeauthoryear{Dynkin}{1991}]{D91}
\begin{barticle}[mr]
\bauthor{\bsnm{Dynkin},~\bfnm{E.~B.}\binits{E.~B.}}
(\byear{1991}).
\btitle{Branching particle systems and superprocesses}.
\bjournal{Ann. Probab.}
\bvolume{19}
\bpages{1157--1194}.
\bid{issn={0091-1798}, mr={1112411}}
\bptok{imsref}%
\end{barticle}
\endbibitem

\bibitem[\protect\citeauthoryear{Evans and Perkins}{1991}]{EP91}
\begin{barticle}[mr]
\bauthor{\bsnm{Evans},~\bfnm{Steven~N.}\binits{S.~N.}} \AND
  \bauthor{\bsnm{Perkins},~\bfnm{Edwin}\binits{E.}}
(\byear{1991}).
\btitle{Absolute continuity results for superprocesses with some applications}.
\bjournal{Trans. Amer. Math. Soc.}
\bvolume{325}
\bpages{661--681}.
\bid{doi={10.2307/2001643}, issn={0002-9947}, mr={1012522}}
\bptok{imsref}%
\end{barticle}
\endbibitem

\bibitem[\protect\citeauthoryear{Feller}{1951}]{Feller51}
\begin{binproceedings}[mr]
\bauthor{\bsnm{Feller},~\bfnm{William}\binits{W.}}
(\byear{1951}).
\btitle{Diffusion processes in genetics}.
In \bbooktitle{Proceedings of the {S}econd {B}erkeley {S}ymposium on
  {M}athematical {S}tatistics and {P}robability, 1950}
\bpages{227--246}.
\bpublisher{Univ. California Press}, \blocation{Berkeley, CA}.
\bid{mr={0046022}}
\bptok{imsref}%
\end{binproceedings}
\endbibitem

\bibitem[\protect\citeauthoryear{Garsia}{1972}]{Garsia72}
\begin{binproceedings}[mr]
\bauthor{\bsnm{Garsia},~\bfnm{Adriano~M.}\binits{A.~M.}}
(\byear{1972}).
\btitle{Continuity properties of {G}aussian processes with multidimensional
  time parameter}.
In \bbooktitle{Proceedings of the {S}ixth {B}erkeley {S}ymposium on
  {M}athematical {S}tatistics and {P}robability ({U}niv. {C}alifornia,
  {B}erkeley, {C}A, 1970/1971), {V}ol. {II}: {P}robability Theory}
\bpages{369--374}.
\bpublisher{Univ. California Press}, \blocation{Berkeley, CA}.
\bid{mr={0410880}}
\bptok{imsref}%
\end{binproceedings}
\endbibitem


\bibitem[\protect\citeauthoryear{Iscoe}{1988}]{Iscoe88}
\begin{barticle}[mr]
\bauthor{\bsnm{Iscoe},~\bfnm{I.}\binits{I.}}
(\byear{1988}).
\btitle{On the supports of measure-valued critical branching {B}rownian
  motion}.
\bjournal{Ann. Probab.}
\bvolume{16}
\bpages{200--221}.
\bid{issn={0091-1798}, mr={0920265}}
\bptok{imsref}%
\end{barticle}
\endbibitem

\bibitem[\protect\citeauthoryear{Kermack and McKendrick}{1927}]{KM27}
\begin{barticle}[auto:STB|2013/06/05|13:45:01]
\bauthor{\bsnm{Kermack},~\bfnm{W.}\binits{W.}} \AND
  \bauthor{\bsnm{McKendrick},~\bfnm{A.}\binits{A.}}
(\byear{1927}).
\btitle{A contribution to the mathematical theory of epidemics}.
\bjournal{Proc. Roy. Soc. London A}
\bvolume{115}
\bpages{700--721}.
\bptok{imsref}%
\end{barticle}
\endbibitem

\bibitem[\protect\citeauthoryear{Lalley}{2009}]{lalley09}
\begin{barticle}[mr]
\bauthor{\bsnm{Lalley},~\bfnm{Steven~P.}\binits{S.~P.}}
(\byear{2009}).
\btitle{Spatial epidemics: Critical behavior in one dimension}.
\bjournal{Probab. Theory Related Fields}
\bvolume{144}
\bpages{429--469}.
\bid{doi={10.1007/s00440-008-0151-0}, issn={0178-8051}, mr={2496439}}
\bptok{imsref}%
\end{barticle}
\endbibitem

\bibitem[\protect\citeauthoryear{Lalley and Zheng}{2010}]{lz10}
\begin{barticle}[mr]
\bauthor{\bsnm{Lalley},~\bfnm{Steven~P.}\binits{S.~P.}} \AND
  \bauthor{\bsnm{Zheng},~\bfnm{Xinghua}\binits{X.}}
(\byear{2010}).
\btitle{Spatial epidemics and local times for critical branching random walks
  in dimensions 2 and 3}.
\bjournal{Probab. Theory Related Fields}
\bvolume{148}
\bpages{527--566}.
\bid{doi={10.1007/s00440-009-0239-1}, issn={0178-8051}, mr={2678898}}
\bptok{imsref}%
\end{barticle}
\endbibitem

\bibitem[\protect\citeauthoryear{Martin-L{\"o}f}{1998}]{martin-lof1998}
\begin{barticle}[mr]
\bauthor{\bsnm{Martin-L{\"o}f},~\bfnm{Anders}\binits{A.}}
(\byear{1998}).
\btitle{The final size of a nearly critical epidemic, and the first passage
  time of a {W}iener process to a parabolic barrier}.
\bjournal{J. Appl. Probab.}
\bvolume{35}
\bpages{671--682}.
\bid{issn={0021-9002}, mr={1659544}}
\bptok{imsref}%
\end{barticle}
\endbibitem

\bibitem[\protect\citeauthoryear{McKendrick}{1926}]{McK1926}
\begin{barticle}[auto:STB|2013/06/05|13:45:01]
\bauthor{\bsnm{McKendrick},~\bfnm{A.~G.}\binits{A.~G.}}
(\byear{1926}).
\btitle{Applications of mathematics to medical problems}.
\bjournal{Proc. Edinb. Math. Soc. (2)}
\bvolume{14}
\bpages{98--130}.
\bptok{imsref}%
\end{barticle}
\endbibitem

\bibitem[\protect\citeauthoryear{Mollison}{1977}]{mo77}
\begin{barticle}[mr]
\bauthor{\bsnm{Mollison},~\bfnm{Denis}\binits{D.}}
(\byear{1977}).
\btitle{Spatial contact models for ecological and epidemic spread}.
\bjournal{J. R. Stat. Soc. Ser. B Stat. Methodol.}
\bvolume{39}
\bpages{283--326}.
\bid{issn={0035-9246}, mr={0496851}}
\bptok{imsref}%
\end{barticle}
\endbibitem

\bibitem[\protect\citeauthoryear{Mueller and Tribe}{1994}]{MT94}
\begin{barticle}[mr]
\bauthor{\bsnm{Mueller},~\bfnm{Carl}\binits{C.}} \AND
  \bauthor{\bsnm{Tribe},~\bfnm{Roger}\binits{R.}}
(\byear{1994}).
\btitle{A phase transition for a stochastic {PDE} related to the contact
  process}.
\bjournal{Probab. Theory Related Fields}
\bvolume{100}
\bpages{131--156}.
\bid{doi={10.1007/BF01199262}, issn={0178-8051}, mr={1296425}}
\bptok{imsref}%
\end{barticle}
\endbibitem

\bibitem[\protect\citeauthoryear{Mueller and Tribe}{2011}]{MT11}
\begin{barticle}[mr]
\bauthor{\bsnm{Mueller},~\bfnm{Carl}\binits{C.}} \AND
  \bauthor{\bsnm{Tribe},~\bfnm{Roger}\binits{R.}}
(\byear{2011}).
\btitle{A phase diagram for a stochastic reaction diffusion system}.
\bjournal{Probab. Theory Related Fields}
\bvolume{149}
\bpages{561--637}.
\bid{doi={10.1007/s00440-010-0265-z}, issn={0178-8051}, mr={2776626}}
\bptok{imsref}%
\end{barticle}
\endbibitem

\bibitem[\protect\citeauthoryear{M{\"u}ller and Tribe}{1995}]{MT95}
\begin{barticle}[mr]
\bauthor{\bsnm{M{\"u}ller},~\bfnm{C.}\binits{C.}} \AND
  \bauthor{\bsnm{Tribe},~\bfnm{R.}\binits{R.}}
(\byear{1995}).
\btitle{Stochastic p.d.e.'s arising from the long range contact and long range
  voter processes}.
\bjournal{Probab. Theory Related Fields}
\bvolume{102}
\bpages{519--545}.
\bid{doi={10.1007/BF01198848}, issn={0178-8051}, mr={1346264}}
\bptok{imsref}%
\end{barticle}
\endbibitem


\bibitem[\protect\citeauthoryear{Perkins}{1995}]{P95}
\begin{barticle}[mr]
\bauthor{\bsnm{Perkins},~\bfnm{Edwin}\binits{E.}}
(\byear{1995}).
\btitle{On the martingale problem for interactive measure-valued branching
  diffusions}.
\bjournal{Mem. Amer. Math. Soc.}
\bvolume{115}
\bpages{vi+89}.
\bid{doi={10.1090/memo/0549}, issn={0065-9266}, mr={1249422}}
\bptok{imsref}%
\end{barticle}
\endbibitem

\bibitem[\protect\citeauthoryear{Perkins}{2002}]{PerkinsSFnotes}
\begin{bincollection}[mr]
\bauthor{\bsnm{Perkins},~\bfnm{Edwin}\binits{E.}}
(\byear{2002}).
\btitle{Dawson--{W}atanabe superprocesses and measure-valued diffusions}.
In \bbooktitle{Lectures on Probability Theory and Statistics ({S}aint-{F}lour,
  1999)}.
\bseries{Lecture Notes in Math.}
\bvolume{1781}
\bpages{125--324}.
\bpublisher{Springer}, \blocation{Berlin}.
\bid{mr={1915445}}
\bptok{imsref}%
\end{bincollection}
\endbibitem

\bibitem[\protect\citeauthoryear{Pinsky}{1995}]{Pinsky95}
\begin{barticle}[mr]
\bauthor{\bsnm{Pinsky},~\bfnm{Ross~G.}\binits{R.~G.}}
(\byear{1995}).
\btitle{On the large time growth rate of the support of supercritical
  super-{B}rownian motion}.
\bjournal{Ann. Probab.}
\bvolume{23}
\bpages{1748--1754}.
\bid{issn={0091-1798}, mr={1379166}}
\bptok{imsref}%
\end{barticle}
\endbibitem

\bibitem[\protect\citeauthoryear{Revuz and Yor}{1999}]{RevuzYor}
\begin{bbook}[mr]
\bauthor{\bsnm{Revuz},~\bfnm{Daniel}\binits{D.}} \AND
  \bauthor{\bsnm{Yor},~\bfnm{Marc}\binits{M.}}
(\byear{1999}).
\btitle{Continuous Martingales and {B}rownian Motion},
\bedition{3rd} ed.
\bseries{Grundlehren der Mathematischen Wissenschaften}
\bvolume{293}.
\bpublisher{Springer}, \blocation{Berlin}.
\bid{mr={1725357}}
\bptok{imsref}%
\end{bbook}
\endbibitem

\bibitem[\protect\citeauthoryear{Sugitani}{1989}]{Sugitani89}
\begin{barticle}[mr]
\bauthor{\bsnm{Sugitani},~\bfnm{Sadao}\binits{S.}}
(\byear{1989}).
\btitle{Some properties for the measure-valued branching diffusion processes}.
\bjournal{J.~Math. Soc. Japan}
\bvolume{41}
\bpages{437--462}.
\bid{doi={10.2969/jmsj/04130437}, issn={0025-5645}, mr={0999507}}
\bptok{imsref}%
\end{barticle}
\endbibitem


\bibitem[\protect\citeauthoryear{Walsh}{1986}]{Walsh86}
\begin{bincollection}[mr]
\bauthor{\bsnm{Walsh},~\bfnm{John~B.}\binits{J.~B.}}
(\byear{1986}).
\btitle{An introduction to stochastic partial differential equations}.
In \bbooktitle{\'{E}cole D'\'et\'e de Probabilit\'es de {S}aint-{F}lour,
  {XIV}---1984}.
\bseries{Lecture Notes in Math.}
\bvolume{1180}
\bpages{265--439}.
\bpublisher{Springer}, \blocation{Berlin}.
\bid{doi={10.1007/BFb0074920}, mr={0876085}}
\bptok{imsref}%
\end{bincollection}
\endbibitem

\end{thebibliography}
\end{document}